%% file: thesis_root.tex
\documentclass[12pt]{report} 

\usepackage{regina}
\usepackage{latexsym}
\usepackage{float}
\usepackage{amsfonts}
\usepackage{amssymb}
\usepackage{makeidx}     
\usepackage{graphicx}    
\usepackage{multicol}    
\usepackage{slashbox}
\usepackage{comment}
\usepackage{amstext,amsmath}
\usepackage{cite}
\usepackage{amsthm}
\usepackage{enumerate}
\usepackage{makeidx}
\usepackage[all]{xy}
\usepackage[vcentermath]{youngtab}
\usepackage{MnSymbol} 
\usepackage{extdash}
\usepackage{longtable}

\newtheoremstyle{plainsl}%
       {\topsep}
       {\topsep}
       {\slshape} 
       {}
       {\normalfont\bfseries}	
       {.}
       { }
       {}	

\theoremstyle{plainsl}

\makeatletter
\newtheoremstyle{break}
   {\topsep}{\topsep}%
   {\slshape}{}%
   {\bfseries}{}%
   { }
   {\thmname{#1}\thmnumber{\@ifnotempty{#1}{ }\@upn{#2}}%
    \thmnote{ {\bfseries(#3)}}{.}}%
\makeatother
\theoremstyle{break}

\newtheorem{theorem}{Theorem}[section]
\newtheorem{thm}[theorem]{Theorem}

\newtheorem{lem}[theorem]{Lemma}
\newtheorem{cor}[theorem]{Corollary}
\newtheorem{prop}[theorem]{Proposition}

\newtheorem{conj}[theorem]{Conjecture}

\newtheorem*{thma}{{E}rd{\H o}s-{K}o-{R}ado Theorem}

\newtheorem{example}[theorem]{Example}
\let\oldexample\example
\renewcommand{\example}{\oldexample\normalfont}
\newcommand\txtsl[1]{\textsl{#1}\index{#1}}
\newcommand{\p}{\mathbb{P}}
\newcommand{\F}{\mathbb{F}}
\newcommand{\xx}{\mathcal{X}}
\newcommand{\zz}{\mathbb{Z}}

\DeclareMathOperator\id{id}
\DeclareMathOperator\fix{fix}
\DeclareMathOperator\Fix{Fix}
\DeclareMathOperator\proj{proj}
\DeclareMathOperator\ID{\mathbf{id}}
\DeclareMathOperator\diam{diam}
\DeclareMathOperator\alt{Alt}
\DeclareMathOperator\sym{Sym}
\DeclareMathOperator\PGL{PGL}
\DeclareMathOperator\PSL{PSL}
\DeclareMathOperator\SL{SL}
\DeclareMathOperator\hl{hl}
\DeclareMathOperator\rank{rank}
\DeclareMathOperator\nul{null}

\DeclareMathOperator\Tr{tr}
\DeclareMathOperator\sgn{sgn}
\DeclareMathOperator\irr{Irr}
\DeclareMathOperator\im{Im}

\DeclareMathOperator\Spec{Spec}
\DeclareMathOperator\cc{\mathcal{C}}
\DeclareMathOperator\dd{\mathcal{D}}

\makeindex             

\title{Maximum Intersecting Families of Permutations}
\author{Bahman Ahmadi}
\dept{Mathematics}
\faculty{Science}
\submitdate{July 2013}
\copyrightyear{2013}
\copyrighttrue
\figurespagefalse
\tablespagefalse

\begin{document}
\beforepreface
\addcontentsline{toc}{chapter}{Abstract}
\include{abstract}
\addcontentsline{toc}{chapter}{Acknowledgments}
\include{acknowledgments}
\addcontentsline{toc}{chapter}{Dedication}
\include{dedication}
 \tableofcontents
\listoffigures 
 \listoftables
 
 \afterpreface

\include{chap-intro}
\include{chap-graph_theory}

\include{chap-rep_theory}
\include{chap-evals_of_Cayley}
\include{chap-EKR_perm_groups}

\include{chap-module_method}

\include{chap-single_CC}

\include{chap-future_works}

\appendix
\include{app}

\printindex
\addcontentsline{toc}{chapter}{\numberline{}Bibliography}
\bibliographystyle{plain}
\bibliography{bibliography}
\end{document}

%% file: abstract.tex
\chapter*{Abstract}
In extremal set theory, the  {E}rd{\H o}s-{K}o-{R}ado (EKR) theorem  gives an upper bound  on the size of intersecting $k$-subsets of the set $\{1,\dots,n\}$. Furthemore, it classifies the maximum-sized families of intersecting $k$-subsets.  It has been shown that similar theorems can be proved for other mathematical objects with a suitable notion of ``intersection''.  Let $G\leq \sym(n)$ be a permutation group with its permutation action on the set $\{1,\dots,n\}$. The intersection for the elements of $G$ is defined as follows: two permutations $\alpha,\beta\in G$ are \textsl{intersecting} if $\alpha(i)=\beta(i)$ for some $i\in\{1,\dots,n\}$. A subset $S$ of $G$ is, then, \textsl{intersecting} if  any pair of its elements is intersecting. We say $G$ has the \textsl{EKR property} if the size of any intersecting subset of $G$ is bounded above by the size of a point stabilizer in $G$. If, in addition, the only maximum-sized  intersecting subsets are the cosets of the point-stabilizers in $G$, then $G$ is said to have the \textsl{strict EKR property}. It was first shown by Cameron and Ku \cite{MR2009400} that the group $G=\sym(n)$ has the strict EKR property. Then Godsil and Meagher presented an entirely different proof of this fact using some algebraic properties of the symmetric group. A similar method was employed to prove that  the projective general linear group $\PGL(2,q)$, with its natural action on the projective line $\p_q$, has the strict EKR property. The main objective in this thesis is to formally introduce this method, which we call the module method, and show that this provides a standard way to prove {E}rd{\H o}s-{K}o-{R}ado theorems for other permutation groups. We then, along with proving {E}rd{\H o}s-{K}o-{R}ado theorems for various groups, use this method to prove some permutation groups have the strict EKR property. We will also show that this method can be useful in characterizing the maximum independent sets of some Cayley graphs. To explain the module method, we need some facts from representation theory of groups, in particular, the symmetric group. We will provide the reader with a sufficient level of background from representation theory as well as graph theory and linear algebraic facts about graphs.

%% file: acknowledgments.tex
\chapter*{Acknowledgments}
I would like to express my special appreciation and gratitude  to my  advisor  Dr. Karen Meagher who not only has been guiding me in my PhD research by her wonderful ideas and sharp knowledge, but also has been extremely supportive to me and my family during the hard days of the past four years. Dear Karen, I will never forget your sincerity and kindness.

I would also like to thank the committee members of my PhD, professor Shaun M. Fallat, professor Allen Herman and  Dr. Sandra Zilles for serving as my committee. 
As well, a special thanks to professor  Chris Godsil and  Dr. Mike Newman for some of their valuable  comments.

My research has been financially supported by the Faculty of Graduate Studies and Research, Department of Mathematics and Statistics, Dr. Karen Meagher, professor Steve Kirkland and the Pprovince of Saskatchewan. I appreciate all of their supports.

Finally I would like to  thank my wife. Fatemeh, words can not express how grateful I am to you for all of the sacrifices that you have made on my behalf. 


 

%% file: dedication.tex
\chapter*{Dedication}
I lovingly dedicate this thesis to  my beloved wife, {\bf Fatemeh}, who has sincerely  supported me in all steps of my life, and to our lovely son, {\bf Ilia}, who has brought lots of joys and happiness to our life since February 2012!

%% file: chap-intro.tex
\chapter{Introduction}\label{introduction}

The celebrated {E}rd{\H o}s-{K}o-{R}ado (abbreviated EKR) theorem is a very fundamental and important theorem in combinatorics which was essential to the development of extremal set theory. To state this theorem, let $k$ and $n$ be positive integers with $n>2k$. A family $\mathcal{A}$ of $k$-subsets (i.e. subsets of size $k$) of $\{1,\ldots,n\}$  is said to be an \textsl{intersecting system}\index{intersecting set system} if any two sets from $\mathcal{A}$ have non-trivial intersection. The {E}rd{\H o}s-{K}o-{R}ado Theorem, then, is as follows.

\begin{thma}\label{original_EKR}~\cite{TheOriginalEKR} 
Let $n>2k$. If $\mathcal{A}$ is an intersecting system of $k$-subsets of the set $\{1,\ldots,n\}$, then
\begin{equation}\label{EKR}
|\mathcal{A}|\leq \binom{n-1}{k-1}.
\end{equation}
Moreover, $\mathcal{A}$ meets the bound if and only if $\mathcal{A}$ is the collection of all $k$-subsets that contain a fixed $i\in \{1,\ldots,n\}$.\qed
\end{thma}

A family $\mathcal{A}$ of $k$-subsets of $\{1,\ldots,n\}$ all of whose elements contain a fixed element is usually called a \txtsl{trivially intersecting} family. There is, also, a generalized version of this theorem which classifies the \textsl{$t$-intersecting systems} of $k$-subsets; that is, the families of $k$-subsets of $\{1,\ldots,n\}$ in which every pair of sets has intersection at least of size $t$. For both the special and general cases of the EKR theorem, there are many different proofs. For a survey of some of these proofs see   \cite{doi:10.1137/0604042} or \cite{MR1059690}. 

The most astonishing characteristic of the EKR theorem is that similar results occur in many other situations. In other words, if we replace ``sets'' in the EKR theorem with some other objects and then define a suitable ``intersection'' property, we may encounter a similar behavior as in the EKR theorem. The following is a list of some of the theorems which were proved for different objects with relevant notion of intersections. We call them ``extensions'' of the EKR theorem or other ``versions'' of this theorem.
\begin{itemize}
\item In \cite{MR0382015} a version of  EKR theorem was proved for intersecting subspaces of a vector space;
\item In \cite{MR1722210} a version of the EKR theorem was proved for intersecting integer sequences;
\item In \cite{MR657053} a version of EKR theorem was proved for  intersecting blocks in a design;
\item In \cite{MR2009400} a version of EKR theorem was proved for intersecting permutations;
\item In \cite{MR2156694} a version of EKR theorem was proved for uniform set partition systems;
\item In \cite{MeagherS11} a version of EKR theorem was proved for intersecting permutations of the action of PGL$(2,q)$ on the projective line.
\end{itemize}
\bigskip
{\bf EKR for the intersecting permutations}

In the version of the EKR theorem for intersecting permutations given in  \cite{MR2009400}, the ``objects'' are all the permutations on the elements $\{1,\dots,n\}$, i.e. the elements of the symmetric group $\sym(n)$, and two permutations are said to ``intersect'' if they map $i$ to the same point, for some $i\in\{1,\dots,n\}$. A set of permutations is called ``intersecting'' if any pair of its elements intersect. It was proved that the maximum size of an intersecting set is $(n-1)!$ and then the only intersecting sets of this size were characterized to be the sets $S_{i,j}$ of all the permutations mapping $i$ to $j$, for any $i,j\in\{1,\dots,n\}$. The sets $S_{i,j}$ are, in fact, cosets of the point-stabilizers in $\sym(n)$ under the permutation action of $\sym(n)$ on $\{1,\dots,n\}$. In \cite{Karen}, the authors presented a new proof of this theorem which  relies mainly on the algebraic properties of the symmetric group, especially the irreducible representations of this group. It would, then,  be very natural to generalize this result to other permutation groups, i.e. proper subgroups of $\sym(n)$.  A similar method was used in \cite{MeagherS11} to prove a version of EKR theorem for the projective linear group $\PGL(2,q)$. The usefulness of this method, which is the core idea of what we will call  the ``module method'',  was the main motivation for the author of this thesis to generalize it to any $2$-transitive permutation group. This method, then  can be applied to different permutation groups. Using this formulation, then, along with proving EKR theorems for various groups,  we present a new proof for the version of the EKR theorem for the alternating group which was initially proved in \cite{KuW07}. Furthermore, we provide an important generalization of this method to the case where the intersection of two elements of a permutation group $G$ is defined with respect to any union of conjugacy classes of $G$ (in contrast to the old definition of the  intersection where the adjacency is defined with respect to the union of all the conjugacy classes of the derangements of $G$). Then we show that the module method can be generalized to this case. Using this generalization of the method, we prove an interesting version of EKR theorem for the alternating group $\alt(n)$ with respect to the conjugacy class of all $n$-cycles.
\\\\
{\bf Graph interpretations}

Any system $\mathcal{A}$ for which the bound in (\ref{EKR}) is achieved is called a \txtsl{maximum intersecting system}. In any version of the EKR theorem, one can associate an appropriate graph to the problem so that the problem of classifying the  maximum intersecting families is equivalent to the classification of maximum  independent sets in the graph. For instance, to the original EKR theorem we can associate the graph whose vertices are all the $k$-subsets of $\{1,\ldots,n\}$ and two vertices are adjacent if their corresponding sets do not intersect. This  graph is the well-known \textsl{Kneser graph}\index{Kneser graph} (see \cite[Chapter 7]{MR1829620} for a detailed discussion about this graph). A set of vertices of this graph is independent if and only if the corresponding sets form an intersecting system of $k$-subsets. 
In \cite{Newman}, there is an algebraic proof of the {E}rd{\H o}s-{K}o-{R}ado theorem using Kneser graphs. 

In the case of  permutation groups, we will use the so called ``derangement graphs'' whose vertices are all the elements of the group we are considering, and two vertices of the graph are adjacent if they don't intersect. The derangement graphs are in the family of normal Cayley graphs. This is  where the problem connects to representation theory as the irreducible characters of the group provide the required information on the eigenvalues of the derangement graphs. In addition to representation theory, two well-known bounds on the size of independent sets of the graphs, namely the ``ratio'' bound and the ``clique-coclique'' bound provide useful machinery we need in our method.
\\\\
{\bf Overview of the document}

This thesis consists of eight chapters, the first one being this introduction. In Chapter~\ref{graph theory} we provide the reader with a short review of the facts we will need from graph theory. We will also provide a brief introduction to spectral graph theory and will present a proof of the ratio bound theorem. In addition, we prove some new and useful  results in spectral graph theory which will be employed in later chapters. 

Chapter~\ref{rep. theory} is an introduction to  representation theory of finite groups. There, we mainly study some basic facts about representations, characters, irreducible representations and some examples of different types of representations. Then we turn our attention to the irreducible representations of the symmetric and the alternating groups which are of great significance in this research work. There are several references for representation theory of these groups. Using some of them, we provide a compact introduction to this topic which will give an  overview of the subject and the fundamental  facts which we need throughout the thesis. 

Chapter~\ref{evalues of Cayley} is devoted to introducing  Cayley graphs, investigating some of their basic properties and explaining how we can use the irreducible representations of a group to evaluate the eigenvalues of Cayley graphs based on that group; in other words we give a detailed proof of Theorem~\ref{Diaconis} which is a very nice and strong connection  between the character theory of finite groups and spectral graph theory. The paper \cite{MR626813}, where this result first appeared, is mainly in probability theory and the theme of the work is not very compatible with the literature of spectral graph theory.  The main purpose of Section~\ref{evals_of_normal} is to provide a detailed proof of this result which can be used as a standard proof for the researchers in this field. 
We will also discuss  some connections between the eigenvalues of a Cayley graph based on a given group and those of the corresponding quotient groups. This will provide some useful machinery for some future work in this area.

In Chapter~\ref{EKR_perm_groups} we officially define the concepts of EKR and strict EKR property for permutation groups, which correspond to the first and the second conditions in the EKR theorem. Furthermore, we prove  that some famous groups have  the EKR or strict EKR property. The terms EKR and strict EKR properties are not new  and have been used in many recent research works in this area; however, for the first time in the literature, we investigate these properties for a large list of families of groups and provide very detailed discussions in this subject. For example, we prove that all the cyclic and dihedral groups have the strict EKR property and all the Frobenius groups have the EKR property. In addition, we discuss the EKR and strict EKR properties for some group products.

We further investigate the EKR and strict EKR property of groups in Chapter~\ref{module_method}, where we introduce the proof method we call the  ``module method''. Then, using the module method, we prove that the alternating group, the $2$-transitive Mathieu groups and all  $4$-transitive groups have the strict EKR property. Further, we show that the projective special linear group $\PSL(2,q)$ has the EKR property and show that, provided some matrix related to this group has full rank, $\PSL(2,q)$, in fact, has the strict EKR property. 

In Chapter~\ref{single_CC} we turn our attention to characterizing the maximum independent sets of some Cayley graphs on the symmetric group with respect to a single conjugacy class of derangements. To do this, we first generalize the module method discussed in Chapter~\ref{module_method}. We then prove a sequence of interesting results  regarding some connections between the algebraic properties of the underlying group and the graph theoretical properties of the corresponding Cayley graph. Then using the generalized module method, we prove that the alternating group has the strict EKR property.

The concluding chapter is Chapter~\ref{future} where we provide a list of open questions and conjectures we have come up with during this research work.
\\\\
{\bf General notation}

In this thesis, all the graphs are assumed to be simple and finite and all  the groups are assumed to be finite. We will denote any cyclic group of size $r$ with $\mathbb{Z}_r$. 
 For the set $\{1,\dots,n\}$, we will use the notation $[n]$. For any subset $\Omega\subseteq [n]$, the symmetric group\index{symmetric group} on $\Omega$ will be denoted by $\sym(\Omega)$. In particular, if $\Omega=[n]$, then the notation $\sym(n)$ is used for $\sym(\Omega)$.  Any subgroup $G$ of $\sym(n)$ is called a \txtsl{permutation group} of degree $n$. The alternating group\index{alternating group} on $[n]$ is denoted by $\alt(n)$. Furthermore, the identity element of any group $G$ is denoted by $\id_G$ and if $G$ is clear from the context, we simply write $\id$.
Let $S$ be a subset of a set $T$. Then we denote the characteristic vector of $S$ in $T$ by $v_S$ where there is no confusion about $T$.

%% file: chap-graph_theory.tex
\chapter{Graph Theory}\label{graph theory}
In this chapter we provide the reader with some  concepts and facts from graph theory which will be needed throughout this  thesis. We refer the reader to the books \cite{MR2159259} and  \cite{MR600654} for definitions and the basic facts for  graph theory and algebra, respectively.

For any graph $X$, a non-empty subset $S$ of the vertex set $V(X)$  is called \textsl{independent}\index{independent set of vertices} (or a  \textsl{coclique}) if no pair of its elements are adjacent. The maximum size of an independent set in $X$ is called the \txtsl{independence number} of $X$ and is denoted by $\alpha(X)$. Any independent set of maximum size is simply called a \textsl{maximum independent set}. The concept of independent sets is very old and well-studied in  graph theory as well as computer science and other related fields of discrete mathematics. We refer the reader to the books \cite{MR2159259}, \cite{MR1829620}  for  discussions related to the independence number.  The concept of independent sets is very essential in this thesis  as the characterization of maximum independent sets of vertices in some Cayley graphs is the core concept in Chapters~\ref{EKR_perm_groups} through \ref{single_CC}.

This chapter includes four sections. In the first section we introduce two graph products which will be useful in later chapters.
Section~\ref{matrix_theory} provides some well-known facts from matrix theory which will be employed in the thesis, especially in  Section~\ref{spectral_graph_theory} where we present  a brief introduction to algebraic graph theory and some basic concepts related to  linear algebraic aspects of graphs. Finally in Section~\ref{bounds_on_indy} the famous ratio bound for the  independent sets of regular graphs as well as some new bounds will be proved.
Note that if two vertices $x$ and $y$ are adjacent in a graph, then we write $x\sim y$.

\section{Graph products}\label{graph_basics}
In this section we define two products on graphs. Let $X$ and $Y$ be two graphs. The \textsl{direct product}\index{direct product of graphs} of $X$ and $Y$ is the graph $X\times Y$ whose vertex set is
\[
V(X\times Y)=V(X)\times V(Y),
\]
and in which two vertices $(x_1,y_1)$ and $(x_2,y_2)$ are adjacent if $x_1\sim x_2$ in $X$ and $y_1\sim y_2$ in $Y$. 
Note that for any independent set $S$ in $X$, the set $S\times V(Y)$ is an independent set in $X\times Y$; hence, $\alpha(X\times Y)\geq \alpha(X)|V(Y)|$. Similarly $\alpha(X\times Y)\geq \alpha(Y)|V(X)|$. We conclude that  $\alpha(X\times Y)\geq\max\{\alpha(X)|V(Y)|\,,\,\alpha(Y)|V(X)|\}$. This inequality can be strict for general graphs (see \cite{klavzar}), but Tardif \cite{Tardif} asked if the equality holds if both $X$ and $Y$ are vertex-transitive graphs. This question was answered by Zhang in \cite{Zhang2012832}. He proved
\begin{thm} If $X$  and $Y$ are vertex-transitive  graphs, then 
\[
\alpha(X\times Y)=\max\{\alpha(X)|Y|\,,\,\alpha(Y)|X|\}.\qed
\]
\end{thm}
The following can, then, be easily derived.
\begin{cor}\label{alpha_of_direct_prod_of_graphs}
If $X_1,\dots,X_k$ are vertex-transitive graphs, then 
\[
\alpha(X_1\times \cdots\times X_k)=\max_i\{\,\alpha(X_i)\prod_{\substack{j=1,\dots,n\\ j\neq i}}|V(X_j)|\,\}.\qed
\]
\end{cor}

Our second product is the lexicographic product of graphs. Let $X$ and $Y$ be two graphs. Then their \txtsl{lexicographical product} $X[Y]$ is a graph with vertex set $V(X)\times V(Y)$ in which two vertices $(x_1,y_1),(x_2,y_2)$ are adjacent if and only if $x_1\sim x_2$ in $X$ or $x_1=x_2$ and $y_1\sim y_2$ in $Y$. An easy interpretation of $X[Y]$ is as follows: to construct $X[Y]$, replace any vertex of $X$ with a copy of $Y$, and if two vertices $x_1$ and $x_2$ in $X$ are adjacent, then in $X[Y]$ all the vertices which replace $x_1$ will be adjacent to all the vertices which replace $x_2$. For example,
\[
K_n[\overline{K_m}]\cong K_{\underbrace{m,m,\ldots,m}_{n\,\,\text{times}}},
\]
where $K_n$ is the complete graph on $n$ vertices.
Note that if $S_X$ and $S_Y$ are independent sets in $X$ and $Y$, respectively, then  $S_X[S_Y]$ is an independent set of vertices of $X[Y]$. This implies that 
\[
\alpha(X[Y])\geq \alpha(X)\alpha(Y).
\]
In fact, we can say more:
\begin{prop}[{{see \cite{Geller197587}}}]\label{independence-of-lex} Let $X$ and $Y$ be graphs. Then 
\[
\alpha(X[Y])=\alpha(X)\alpha(Y).\qed
\]
\end{prop}
For any $x\in V(X)$, let $Y_x=\{x\}\times Y$. Then it is easy to see that $Y_x\cong Y$.  In order to see what the maximum independent sets in $X[Y]$ look like, for any subset $S$ of the vertices of $X[Y]$, we define the projection of $S$ to $X$ as
\[
\proj_X(S)=\{x\in X\,\,|\,\, (x,y)\in S,\, \text{for some}\, y\in Y\}.
\]
Similarly, for any $x\in V(X)$ we define the projection of $S$ to  $Y_x$ as
\[
\proj_{Y_x}(S)=\{y\in Y\,\,|\,\,(x,y) \in S\}.
\]
We can, then, observe the following.
\begin{prop}\label{max_indy_in_lex} Let $X$ and $Y$ be graphs. If $S$ is an independent set in $X[Y]$ of size $\alpha(X)\alpha(Y )$,  then $\proj_X(S)$ is a maximum independent set  in $X$ and, for any $x\in V(X)$, $\proj_{Y_x}(S)$  is a maximum independent set in $Y_x$.\qed
\end{prop}


\section{Some matrix theory}\label{matrix_theory}
In this section, we recall some facts from matrix theory which we will use throughout the thesis. 
The reader may refer to \cite{MR0276251} or \cite{MR832183} for detailed discussions about these results.

For a matrix $A=[a_{i,j}]$, the \textsl{transpose}\index{transpose of a matrix} $A^\top$ of $A$ is the defined as $A^\top=[a'_{i,j}]$, where $a'_{i,j}=a_{j,i}$, for all $i$ and $j$. An $n\times n$ matrix $A$ is said to be \textsl{symmetric}\index{symmetric matrix} if $A^\top=A$. 
Also, $A$ is said to be \textsl{real orthogonal}\index{real orthogonal matrix}, if $A^{-1}=A^\top$; this is a particular case of \textsl{unitary} matrices, where we assume  the matrices to have real entries.

The \txtsl{trace} of a square matrix $A$, $\Tr(A)$, is defined to be the sum of diagonal entries of $A$. Two matrices $A$ and $B$ are said to be  \textsl{similar}\index{similar matrices} (\txtsl{orthogonally equivalent}) if there exists a non-singular (real orthogonal) matrix $S$, such that $B=S^{-1}AS$. A square matrix $A$ is said to be (\textsl{orthogonally}\index{diagonalizable!orthogonally}) \txtsl{diagonalizable} if it is similar (orthogonally equivalent) to a diagonal matrix $D$. The following fact is well-known (see \cite[Section 4.1]{MR832183} for  a proof).
\begin{thm}\label{diag}
Any symmetric matrix is orthogonally diagonalizable.\qed
\end{thm}

Throughout this thesis, we will denote the identity matrix of size $n$, by $I_n$ or simply $I$ if there is no confusion about the size. Similarly the $n\times n$ matrix all of whose entries are $1$ is denoted by $J_n$ or $J$.

For a square matrix $A$, the \txtsl{characteristic polynomial} of $A$ is defined to be the monic polynomial
\[
\phi(A;x)=\det(xI-A).
\]
For instance, the characteristic polynomials of the zero matrix of size $n$, $I_n$ and $J_n$ are $x^n$, $(x-1)^n$ and $(x-n)x^{n-1}$, respectively. It is not hard to see that $\phi(A,0)$ is the determinant of $A$. 
\begin{thm}[{{see \cite[Section 1.3]{MR832183}}}]\label{similar}
If two matrices $A$ and $B$ are similar, then $A$ and $B$ have the same characteristic polynomial. In particular, they have the same determinants and traces.\qed
\end{thm}

The \textsl{eigenvalues} of $A$ are defined to be the roots $\lambda$ of the \txtsl{characteristic polynomial}. Equivalently, a complex number $\lambda$ is an eigenvalue of $A$, if the determinant of the matrix $\lambda I-A$ is zero. This is, in turn, equivalent to the fact that there is a non-zero vector $v$ in the null space of $\lambda I-A$. If $D=S^{-1}AS$ is a diagonalization  of a matrix $A$, then the diagonal entries of $D$ are the eigenvalues of $A$. Hence, one can deduce from Theorem~\ref{similar} that the trace of a matrix is equal to the sum of the eigenvalues of the matrix. 

If $\lambda$ is an eigenvalue of $A$, then the null space of $\lambda I -A$ and its non-zero elements are called the \textsl{eigenspace}\index{eigenspace} and the \textsl{eigenvectors} of $A$ corresponding to the eigenvalue $\lambda$, respectively. The (algebraic) \textsl{multiplicity}\index{multiplicity of an eigenvalue} of the eigenvalue $\lambda$ of a matrix $A$, is the maximum power of the factor $x-\lambda$ in the characteristic polynomial $\phi(A;x)$ and is denoted by $m(\lambda)$.
The following facts are  well-known  in matrix theory.

\begin{prop}[{{see \cite[Section 4.1]{MR832183}}}]
If a square matrix $A$ is symmetric, then all the eigenvalues of $A$  are real.\qed
\end{prop}

A symmetric matrix $A$ is \txtsl{positive semi-definite} (\txtsl{positive definite}) if all the eigenvalues of $A$ are non-negative (positive). 
\begin{lem}\label{sd_matrices} If a matrix $A$ is positive semi-definite, then for any vector $x$, $x^\top Ax\geq 0$. The equality holds if and only if $Ax=0$.
\end{lem}
\begin{proof}
By  Theorem~\ref{diag}, there is a real orthogonal matrix $S$ such that $A=S^\top DS$, where $D$ is a diagonal matrix with diagonal entries $\lambda_1,\dots,\lambda_n$, i.e. the eigenvalues of $A$. Hence we can write 
\[
x^\top Ax= x^\top S^\top DSx =  (Sx)^\top D(Sx) =  y^\top D y = \sum_{i=1}^n \lambda_i y_i^2\geq 0,
\]
where $y=Sx$. This proves the first part of the lemma.
For the second part, let $\sqrt{D}$ be the diagonal matrix with the diagonal elements $\sqrt{\lambda_1},\dots,\sqrt{\lambda_n}$. Then $A=B^\top B$, where $B=\sqrt{D} S$. Now if $x^\top A x=0$, then $x^\top B^\top Bx=0$. This implies that $(Bx)^\top (Bx)=0$; that is $||Bx||=0$, hence $Bx=0$. This yields $Ax=B^\top Bx=0$. The converse is trivial.
\end{proof}

We, next, introduce the ``tensor product'' of matrices  which is one of the well-known matrix operations and has many applications. Let $A$ and $B$ be $m\times n$ and $p\times q$ matrices, respectively. Then their \textsl{tensor product}\index{tensor product of matrices} $A\otimes B$ is defined to be the matrix
\[
\begin{bmatrix}
a_{11}B & \cdots & a_{1n} B\\
\vdots & \ddots & \vdots \\
a_{m1}B & \cdots & a_{mn} B\\
\end{bmatrix}.
\]
In other words, $A\otimes B$ is obtained from $A$ by replacing any entry $a_{ij}$ by the matrix $a_{ij} B$. Hence $A\otimes B$ is an $mp\times nq$ matrix. In particular, if $x$ and $y$ are column vectors of lengths $m$ and $n$, respectively, then $x\otimes y$ is a column vector of length $mn$. It is clear that $A\otimes B$ needs not to be equal to $B\otimes A$. It is easy to see that $(A\otimes B)^{\top}=A^{\top}\otimes B^{\top}$. Also the proof of the following is straight-forward.
\begin{lem}\label{tensor_and_regular_products} Let $A,B,C$ and $D$ be $m\times n$, $r\times s$, $n\times p$ and $s\times t$ matrices, respectively. Then
\[
(A\otimes B) (C\otimes D)= (AC\otimes BD).\qed
\]
\end{lem}
Using this, one can see that $(A\otimes B)^{-1}=A^{-1}\otimes B^{-1}$. In addition to these nice properties, the tensor product enjoys so many other properties. In the following proposition, we will see how the eigenvalues of $A\otimes B$ can be written in terms  of those of $A$ and $B$.
\begin{prop}\label{evalues_of_tensor} Let $A$ and $B$ be $m\times m$ and $n\times n$ matrices, respectively. Then the eigenvalues of $A\otimes B$ are $\lambda_i\mu_j$, where $\lambda_1,\dots,\lambda_m$ and $\mu_1,\dots,\mu_n$ are all the eigenvalues of $A$ and $B$, respectively.
\end{prop}
\begin{proof} Let $\lambda$ and $\mu$ be eigenvalues of $A$ and $B$, with eigenvectors $v$ and $w$, respectively. Then $Av=\lambda v$ and $Bw=\mu w$. Therefore, according to Lemma~\ref{tensor_and_regular_products}, we have
\[
(A\otimes B)(v\otimes w)=Av\otimes Bw= \lambda v \otimes \mu w=\lambda\mu (v\otimes w).\qedhere
\]
\end{proof}
 
As a nice application of the tensor product, we point out that the adjacency matrix of the graph $X\times Y$ is the tensor product $A(X)\otimes A(Y)$; hence using Proposition~\ref{evalues_of_tensor}, one can obtain the spectrum of $X\times Y$ using those of $X$ and $Y$.

\section{Spectral graph theory}\label{spectral_graph_theory}

This section is a brief introduction to an important part of algebraic graph theory, namely  spectral graph theory. In this context, algebraic and linear algebraic tools are used to establish graph theoretical properties for graphs. For more details on spectral graph theory, the reader may refer to \cite{MR1271140} or  \cite{MR1829620}.  The starting point is the concept of  \txtsl{adjacency matrix}.

Assume that $X$ is a graph with vertex set $V(X)=\{v_1,\dots,v_n\}$. Then the adjacency matrix of $X$, which is denoted by $A(X)$, is the square 01-matrix of size $n$, whose entry $(i,j)$ is 1 if and only if $v_i$ is adjacent to $v_j$. Note that there are, also, some other important matrices associated to a graph (for example, the incidence matrix or the Laplacian matrix) which provide further connections with linear algebra and matrix theory. It follows immediately from the definition of the adjacency matrix of a graph that it is real symmetric and hence is diagonalizable and, since the graph $X$ has no loops, the trace of $A(X)$ is zero. Note also that if $\alpha\in \sym(V(X))$ is a permutation on the vertices of $X$, then the adjacency matrix of $X$ based on the labeling $V(X)=\{v_{\alpha(1)},\dots,v_{\alpha(n)}\}$, is $P^{-1}A(X)P$, where $P$ is the permutation matrix corresponding to $\alpha$. This implies that changing the order of the vertices of $X$ will result in distinct but similar matrices to the adjacency matrix. We deduce that the order on the vertices of $X$ is not important in this context.

The nullity and the rank of a graph $X$ is defined to be the nullity and the rank of $A(X)$, respectively. Further, the \textsl{eigenvalues}\index{eigenvalue of a graph} of the graph $X$ are defined to be the eigenvalues of the matrix $A(X)$ and for an eigenvalue $\lambda$ of the graph $X$, the \textsl{eigenspace}\index{eigenspace of a graph}  and the \textsl{eigenvectors}\index{eigenvector of a graph} of $X$ corresponding to $\lambda$ are defined to be the eigenspace and the eigenvectors of $A(X)$ corresponding to $\lambda$, respectively. Note these are well-defined since changing the order of the vertices of $X$ makes similar matrices. Because  the eigenvalues of a graph $X$ with $n$ vertices are real, we usually order them as $\lambda_n\geq \lambda_{n-1}\geq \cdots\geq \lambda_1$. The least eigenvalue, $\lambda_1$ (which is often denoted by $\tau$), plays an important role in the characterization of independent sets of regular graphs (see Section~\ref{EKR_for_psl}, Section~\ref{sporadic} and Section~\ref{EKR_cyclic_perms}). Since the trace of $A(X)$ is zero, one can observe that the least eigenvalue of any non-empty graph is negative.

 The \txtsl{spectrum} of a graph $X$ is the following array:
\[
\Spec(X) =
\left( {\begin{array}{cccc}
 \lambda_1 & \lambda_2 & \ldots & \lambda_s   \\
 m(\lambda_1) & m(\lambda_2) & \ldots & m(\lambda_s)\\
 \end{array} } \right),
\]
where $\lambda_s>\cdots>\lambda_1$ are the distinct eigenvalues of $X$.
For example, the spectra of the complete graph $K_n$ and the complete bipartite graph $K_{m,n}$ are as follows:
\begin{equation}\label{spec_of_K_n}
\Spec(K_n) =
\left( {\begin{array}{cccc}
 n-1 & -1   \\
 1 & n-1\\
 \end{array} } \right),
\end{equation}
\[
\Spec(K_{m,n}) =
\left( {\begin{array}{cccc}
 \sqrt{mn}& 0 & -\sqrt{mn}   \\
 1 & m+n-2& 1\\
 \end{array} } \right).
\]

For any pair of graphs $X$ and $Y$ with disjoint vertex sets  and disjoint edge sets, the \textsl{disjoint union}\index{disjoint union of graphs} of $X$ and $Y$, denoted by $X\cupdot Y$, is defined to be the graph whose set of vertices is the union of the vertex sets of $X$ and $Y$ and whose edge set is the union of the edge sets of $X$ and $Y$. It is not difficult to see that the spectrum of $X\cupdot Y$ is simply the union of those of $X$ and $Y$.


It is obvious that isomorphic graphs have the same spectrum but the converse is not true in general; for example, the star $K_{1,4}$ and the union of a $4$-cycle and  a single vertex, $C_4\cup K_1$, have the same spectrum. The problem of finding the graphs which are determined by their spectrum is one of the most interesting topics in spectral graph theory. One of the most surprising results in this field was shown by Schwenk in~\cite{MR0384582} which states that ``almost'' no tree is identified by its spectrum. However, the spectrum of a graph contains important information about the graph and, to some extent, describes some  graph theoretical parameters of the graph. In what follows we observe some of the applications of the spectrum.
The following result gives an upper bound for the diameter of  a graph $G$, $\diam(G)$, using $\Spec(G)$. 
\begin{thm}\label{distinct_evals} 
Let $G$ be a connected graph. Then $G$ has at least $\diam(G)+1$ distinct eigenvalues.\qed
\end{thm} 
See \cite[Corollary 2.7]{MR1271140} for a proof. Using this, we observe the following fact which will be useful in Section~\ref{EKR_property}.
\begin{prop}\label{graphs_with_2_evalues} Let $G$ be a graph with exactly two distinct eigenvalues. Then
\[
G\cong \underbrace{K_n\cupdot\cdots\cupdot K_n}_{r \text{ times}},
\]
for some $r\geq 1$ and $n\geq 2$.
\end{prop}
\begin{proof} If there are vertices $u,v$ and $w$ in $G$ such that $u\sim v$ and $u\sim w$ but $v$ is not adjacent to $w$, then there will be an induced path of length two in $G$; that is, $\diam(G)\geq 2$ which is a contradiction, according to Theorem~\ref{distinct_evals}. This implies that all of the connected components of $G$ are complete graphs. To complete the proof, it suffices to note that if the sizes of these cliques are not the same, then according to (\ref{spec_of_K_n}), $G$ will have more than two distinct eigenvalues.
\end{proof}

Throughout the text, we denote by $\mathbf{1}$ the column vector all of whose entries are 1; that is, $\mathbf{1}=(1,1,\ldots,1)^\top$, where its length is clear from the context. If $X$ is a $k$-regular graph, then an easy calculation shows that  $A(X)\mathbf{1}=k\mathbf{1}$  which implies that $k$ is an eigenvalue of $X$ and that $\mathbf{1}$ is an eigenvector  of $A$ corresponding to $k$. We can say more:
\begin{thm}[{{see \cite[Chapter 3]{MR1271140}}}]\label{multiplicity}
If $X$ is a $k$-regular graph, then the multiplicity of $k$ as an eigenvalue of $X$ is equal to the number of connected components of $X$. Furthermore, for any eigenvalue $\lambda$ of $X$, we have $|\lambda|\leq k$.\qed
\end{thm}

We say the spectrum of a graph is \textsl{symmetric about the origin} if for any $\lambda$ in the spectrum, $-\lambda$ is also in the spectrum. For example, the spectrum of $K_{m,n}$ is symmetric about the origin. This is, indeed, true for any bipartite graph. In fact, there is an even stronger result.
\begin{thm}[{{see \cite[Theorem 8.6.9]{MR1367739}}}]\label{symmetric_about_origin} A graph  is bipartite if and only if its spectrum is symmetric about the origin.\qed
\end{thm}

Next we define the ``morphisms'' of the ``category'' of graphs, i.e. the maps between graphs which  preserve some of the structures of graphs. Let  $X$ and $Y$ be graphs. A \txtsl{graph homomorphism} (or simply a \textsl{homomorphism}) $\phi: X\to Y$ is a map
\[
\phi:V(X)\to V(Y),
\]
such that if  $u$ and $v$ are adjacent in $X$, then $\phi(u)$ and $\phi(v)$ are adjacent
in $Y$. The notions \textsl{monomorphism}, \textsl{epimorphism}, \textsl{isomorphism} and \textsl{automorphism} are, then, defined as usual.
It is easy to observe that two graphs $X$ and $Y$ are isomorphic if and only if there is a permutation matrix $P$ (i.e. a square 01-matrix whose every row and every column has exactly one 1), such that $A(Y)=P^\top A(X)P$. In particular, $A(X)$ and $A(Y)$ are similar and, therefore, have the same spectrum. Note that changing the order of the vertices of $X$ is, indeed, an isomorphism on $X$.

\section{Bounds on the independence number}\label{bounds_on_indy}

For any vertex-transitive graph $X$, we define the \txtsl{fractional chromatic  number} of $X$ to be 
\[
\chi^*(X)=\frac{|V(X)|}{\alpha(X)}.
\]
See \cite{MR1829620} for a general definition of the fractional chromatic number. The following inequality has been proved in \cite{MR1829620}.
\begin{prop}\label{fractional_chrom_ineq} If there is a homomorphism   from  a graph $X$ to a graph $Y$, then $\chi^*(X)\leq \chi^*(Y)$.\qed
\end{prop}
This means that if $X$ and $Y$ are vertex-transitive graphs, then the existence of a homomorphism $f:X\longrightarrow Y$, implies that 
\[
\frac{|V(X)|}{\alpha(X)}\leq \frac{|V(Y)|}{\alpha(Y)},
\]
which provides a bound for $\alpha(X)$ or $\alpha(Y)$ provided that the other one is given.

Now we turn our attention to some discussions about the least eigenvalues of graphs and their applications. Recall that an independent set  in a graph is a set of vertices in which no pair of the vertices are adjacent. A \txtsl{clique} in a graph is a set of vertices in which every pair of vertices are adjacent. We first present the well-known \txtsl{ratio bound for independent sets}, which gives an upper bound for the size of independent sets in a regular graph. The main tool for this theorem is the least eigenvalue of the graph.  It is, therefore, of great importance to know what the least eigenvalue is or, to be able to approximate it.  

The ratio bound for independent sets is due to Delsarte who used a linear programming argument to prove the ratio bound in association schemes (see Section 3.2 of \cite{Newman}). There is another proof based on equitable partitions and interlacing which is due to Haemers (see Section 9.6 of \cite{MR1829620}). We will present a proof which  appears in \cite{Newman} and is based on
positive semi-definite matrices.

\begin{thm}[ratio bound for independent sets]\label{ratio2}
Let $X$ be a  $k$-regular graph on $n$ vertices with $\tau$ the least eigenvalue of $X$. For any independent set $S$ we have
\[
|S|\leq \frac{n}{1-\frac{k}{\tau}}.
\]
Furthermore, the equality holds if and only if
\[
A(X)\left(v_S-\frac{|S|}{n}\mathbf{1}\right)=\tau\left(v_S-\frac{|S|}{n}\mathbf{1}\right).
\]
\end{thm}
\begin{proof}
Let $A=A(X)$ and for convenience assume $z=v_S$. Since the matrix $A-\tau I$ is positive semi-definite, according to Lemma~\ref{sd_matrices}, for any vector $y$, we have
\[
y^\top(A-\tau I)y\geq 0,
\]
and equality holds if and only if $y$ is an eigenvector associated to $\tau$. Thus for the vector $y=z-\frac{|S|}{n}\mathbf{1}$ we must have
\begin{equation}\label{ratio2_non_neg}
\left(z-\frac{|S|}{n}\mathbf{1}\right)^\top(A-\tau I)\left(z-\frac{|S|}{n}\mathbf{1}\right)\geq 0.
\end{equation}
Since $S$ is an independent set, it is not hard to see that $z^\top Az=0$ and $z^\top A\mathbf{1}=k|S|$. Therefore, by expanding the terms in (\ref{ratio2_non_neg}), we  have
\[
|S|\leq \frac{n}{1-\frac{k}{\tau}},
\]
which completes the proof of the bound.
For the second statement in the theorem, first assume the theorem holds with equality. Then (\ref{ratio2_non_neg}) holds with equality. Since $A-\tau I$ is positive semi-definite, by Lemma~\ref{sd_matrices}, we have 
\[
(A-\tau I)\left(z-\frac{|S|}{n}\mathbf{1}\right)=0;
\]
hence
\begin{equation}\label{ratio2_evector}
A\left(z-\frac{|S|}{n}\mathbf{1}\right)= \tau \left(z-\frac{|S|}{n}\mathbf{1}\right).
\end{equation}
Note this means that vector $z-\frac{|S|}{n}\mathbf{1}$ is an eigenvector associated to $\tau$. For the converse, assume that (\ref{ratio2_evector}) holds. Let $x$ be a vertex of $X$ which is in $S$. Thus $z_x$, the component of $z$ corresponding to $x$, is 1. Therefore 
\[
\left(\tau\left(z-\frac{|S|}{n}\mathbf{1}\right)\right)_x=\tau \left(1-\frac{|S|}{n}\right).
\]
On the other hand, the component of the vector in the left hand side of (\ref{ratio2_evector}) corresponding to $x$ is
\[
\sum_{w\sim x}\left(z-\frac{|S|}{n}\mathbf{1}\right)_w=\sum_{w\sim x}\frac{-|S|}{n}=-\frac{k|S|}{n};
\]
because $w\sim x$ implies that $z_w=0$. Hence
\[
-\frac{k|S|}{n}=\tau \left(1-\frac{|S|}{n}\right),
\]
which is equivalent to
\[
|S|=\frac{n}{1-\frac{k}{\tau}},
\]
and the proof is complete.
\end{proof}

Note that, the second part of Theorem~\ref{ratio2} states that the characteristic vector of any maximum independent set lies in the direct sum of the $k$-eigenspace  and the $\tau$-eigenspace of $A(X)$. Note also that  Theorem~\ref{ratio2} gives  an upper bound for the least eigenvalue of a $k$-regular graph on $n$ vertices; namely
\[
\tau\leq -\frac{k|S|}{n-|S|},
\]
where $S$ is an independent set in the graph. Thus, in order to find a better bound for $\tau$, one should find an independent set of large size. Therefore, the best bound for $\tau$ using this method is obtained when the independence number of the graph is known:
\[
\tau\leq -\frac{k\cdot\alpha(X)}{n-\alpha(X)}.
\]
For certain graphs, it is also possible to establish a lower bound for the least eigenvalue of the graph, in terms of the size of cliques. The proof of the following theorem was originally done by Mike Newman (but has not been published elsewhere). 

\begin{prop}\label{w}
  Let $X$ be a $k$-regular graph and let $\tau$ be the least   eigenvalue of $X$. Assume that there is a collection $\cc$ of   cliques of $X$ of size $w$, such that every edge of $X$ is contained in a fixed number of elements of $\cc$. Then
\[
\tau\geq -\frac{k}{w-1}.
\]
\end{prop}
\begin{proof}
Assume that every edge of $X$ is contained exactly in $y$ cliques in   $\cc$. Then every vertex of $X$ is contained exactly in  $\frac{k}{w-1}y$ cliques in $\cc$. Define a $01$-matrix $N$ as follows: the   rows of $N$ are indexed by the vertices of $X$ and the columns are indexed by the members of $\cc$; the entry $N_{(x,C)}$ is 1   if and only if the vertex $x$ is in the clique $C$. We will, therefore, have 
\[
NN^\top=\frac{yk}{w-1}I+yA(X),
\]
where $I$ is the identity matrix and $A(X)$ is the adjacency matrix  of $X$. Thus 
\[
\frac{k}{w-1}I+A(X)=\left(\frac{1}{\sqrt{y}}N\right)\left(\frac{1}{\sqrt{y}}N\right)^\top,
\]
which implies that the matrix 
\[
A(X)-\frac{-k}{w-1}I
\]
is positive semi-definite and, hence 
\[
\tau\geq \frac{-k}{w-1}.\qedhere
\]
\end{proof}

We now define a new graph $X_n$, for $n>3$, which we call the \txtsl{pairs graph}. We will make use of this graph in Chapters \ref{module_method} and \ref{single_CC}.  For any $n>3$, the vertices of $X_n$ are all the ordered pairs $(i,j)$, where $i,j\in[n-1]$ and $i\neq j$; the vertices $(i,j)$ and $(k,l)$ are adjacent in $X_n$ if and only if either $\{i,j\}\cap\{k,l\}=\emptyset$, ($i=l$ and $j\neq k$) or ($i\neq l$ and $j=k$). The graph $X_n$ is regular of valency $(n-2)(n-3)$. Note that the vertices of the pairs graph $X_n$ are the pairs from $[n-1]$; so the notation might seem odd, but this is how the graphs arise in Chapters \ref{module_method} and \ref{single_CC}.

\begin{lem}\label{least_eval_of_X_n}
 For any $n>3$, the least eigenvalue of the pairs graph $X_n$ is at least $-(n-3)$.
\end{lem}
\begin{proof}
First note that any cyclic permutation $\alpha=(i_1,\ldots, i_{n-1})$ of $[n-1]$ corresponds to a unique clique of size $n-1$ in $X_n$; namely the clique $C_\alpha$ induced by the vertices $\{(i_1,i_2), (i_2,i_3),\ldots, (i_{n-2},i_{n-1}), (i_{n-1},i_1)\}$. We claim that any edge of $X_n$ is contained in exactly $(n-4)!$ cliques of form $C_\alpha$. Consider the edge $\{(a,b),(c,d)\}$.  If $a\neq d$ and $b=c$, then there are $(n-4)!$ cyclic permutations of form $(a,b,d, -, - ,\cdots, -)$ and this edge is in exactly $(n-4)!$ of the cliques. The case  where $a=d$ and $b \neq c$ is similar. If $\{a,b\}\cap \{c,d\}=\emptyset$ then there are  again $(n-4)!$ cyclic permutation of form $(a,b, -,\cdots,-, c,d, -,\cdots,-)$ (as there are $(n-4)$ ways to assign a position for the  pair $c,d$, and then there are $(n-5)!$ ways to arrange other elements of $\{1,\ldots,n-1\}$ in the remaining spots). Thus the claim is proved. If $\tau$ denotes the least eigenvalue of $X_n$, then we can apply Proposition~\ref{w} to $X_n$ to get
\[
\tau\geq \frac{-k}{w-1}=-\frac{(n-2)(n-3)}{n-2}=-(n-3).\qedhere
\]
\end{proof}

We conclude this section with recalling the well-known \txtsl{clique-coclique bound}; the version we use here was originally proved by Delsarte \cite{MR0384310}. Assume $\mathcal{A}=\{A_0, A_1,\ldots, A_d\}$ is an association scheme on $v$ vertices. Note, then that any pair of the matrices $A_i$ commute; therefore they are ``simultaneously diagonalizable''; that is, they have the same eigenspaces (see \cite[Theorem 1.3.19]{MR832183}). Therefore, we let $\{E_0, E_1,\ldots, E_d\}$ be the set of projections to these common eigenspaces. Note that $E_i$ are idempotents. (For a detailed discussion about association schemes, the reader may refer to
\cite{MR2047311} or \cite{MR882540}.) 

\begin{thm}[clique-coclique bound]\label{clique_coclique_bound} Let  $X$ be the union of some of the graphs in an association scheme $\mathcal{A}$ on $v$ vertices. If $C$ is a clique and $S$ is an independent set in $X$, then
\[
|C||S|\leq v.
\]
If equality holds then
\[
v_C^\top\,E_j\,v_C\,\,v_S^\top\,E_j\,v_S = 0,\quad \text{for all} \quad j > 0.\qed
\]
\end{thm}
Note also that هب the clique-coclique bound holds with equality, then any independent set of maximum size intersects with any clique of maximum size. We refer the reader to \cite{Karen} for a proof of Theorem~\ref{clique_coclique_bound}. We will also make use of the following straight-forward corollary of this result that was also proved in \cite{Karen}.

\begin{cor}\label{clique_vs_coclique}
Let $X$ be a union of graphs in an association scheme such that the clique-coclique bound holds with equality in $X$.  Assume   that $C$ is a maximum clique and $S$ is a maximum independent set in  $X$. Then, for $j > 0$, at most one of the vectors $E_jv_C$ and  $E_jv_S$ is not zero.\qed
\end{cor}

%% file: chap-rep_theory.tex
\chapter{Representation Theory}\label{rep. theory}
This chapter is a brief introduction to representation  theory of finite groups, especially that of the symmetric group. Part of the results of this chapter, indeed, form a basis for the work in the next chapter where the eigenvalues of normal Cayley graphs are described using the representations of the underlying groups. For more details on representation theory of groups, we refer the reader to \cite{MR1153249} or \cite{MR1824028}. 

This chapter includes five sections. The first section contains the basic definitions and facts from representation theory of groups. In Section~\ref{rep_sym} we describe all the irreducible representations of the symmetric group. In Section~\ref{MNRule_and_hook_length_formula}, we recall  the ``Murnaghan-Nakayama rule'' and the ``hook-length formula''. Section~\ref{rep_alt} provides some facts about how to derive the irreducible representations of the alternating group from those of the symmetric group. The final section is devoted to the new concept of ``two-layer hooks'', the results of which will be of great importance in Section~\ref{EKR_for_alt}.

\section{Definitions and basic properties}\label{rep_basic}
We start with the definition of a representation. A \txtsl{representation} of a group  $G$ on a finite dimensional complex vector space $V$ is a homomorphism $\mathcal{X}: G\rightarrow GL(V)$ of $G$ to the group of automorphisms of $V$. Often $V$ itself is called the representation. The \txtsl{dimension} of $\mathcal{X}$ is defined to be the dimension of $V$. If there is a representation $\mathcal{X}:G\to GL(V)$, then $V$ has the structure of a $G$-\txtsl{module}.  Note also that, sometimes we may denote the automorphism $\mathcal{X}(g)$ simply by $g$, for every $g\in G$; that is, for a $g\in G$ and a vector $v\in V$, we may write $g\cdot v$ or $gv$ instead of $\mathcal{X}(g)(v)$.

Let $E$ be the identity element of $GL(V)$; then the \textsl{kernel}\index{kernel of a representation} of $\mathcal{X}$, denoted by $\ker (\mathcal{X})$, is defined as
\[
\ker (\mathcal{X}) = \{g\in G\,|\,\mathcal{X}(g)=E\}.
\]
It is not hard to see that this is a normal subgroup of $G$. In fact, representation theory provides a powerful machinery to prove the existence of normal subgroups (see \cite[Sections 15.3, 16.1, 17.9 and 45.1]{huppert1998character} for some examples of this approach). We say $\mathcal{X}$ is \textsl{faithful}\index{faithful representation} if $\ker (\mathcal{X}) =\{\id\}$.

The action of a group on a set is closely related to the concept of group representations. Roughly speaking, if one removes the vector space structure of $V$, in the definition of representation, then there will be an action of $G$ on the set $V$. On the other hand, if the group $G$ acts on the set $S$, then $\mathbb{C}[S]$, is a $G$-module; that is, we have found the representation $\mathcal{X}:G\to GL(V)=GL(\mathbb{C}[S])$ by setting $\mathcal{X}(g)(s)=g\cdot s$, for all $g\in G$ and all $s$ in the basis $S$, and then $\mathbb{C}$-linearly extending it. This representation is referred to as the \textsl{permutation representation}\index{representation!permutation} of the group $G$ with respect to this action. This representation is of dimension $|S|$.

\begin{example}
Consider the trivial action of a group $G$ on a singleton $S=\{s\}$; that is, the action defined by $g\cdot s=s$, for all $g\in G$. Then the 1-dimensional permutation representation of $G$ with respect to this action is called the \textsl{trivial representation}\index{representation!trivial} of $G$ and is denoted by $\ID_G$ or simply $\ID$.
\end{example}

\begin{example}
Consider the (left) multiplication action of $G$ on itself; that is, the action defined by $g\cdot h=gh$, for all $g,h\in G$. Then the $|G|$-dimensional permutation representation of $G$ with respect to this action is called the \textsl{(left-) regular representation}\index{representation!regular}.
\end{example}

\begin{example}
More generally, let $H$ be a subgroup of $G$ of index $m$, and consider the (left) coset action of $G$ on the set of all cosets $S=\{g_1H,\ldots,g_mH\}$; that is, the action defined by $g\cdot(g_iH)=(gg_i)H$. Then the $m$-dimensional permutation representation of $G$ with respect to this action is called \textsl{(left) coset representation}\index{representation!coset} of $G$ with respect to $H$.
\end{example}

\begin{example}
Let $G\leq \sym(n)$ be a permutation group. Consider the defining action of the group $G$ on the set $S=[n]$; that is, the action defined by $\sigma\cdot i=\sigma(i)$, for all $\sigma \in G$  and all $i\in S$. Then the $n$-dimensional permutation representation of $G$ with respect to this action is called the \textsl{defining representation}\index{representation!defining} of $G$.
\end{example}
In this thesis, we will also use the term ``natural action'' instead of ``defining action''. 

Let $G$ be a group and assume $V$ and $W$ are two representations of $G$. A $G$-homomorphism\index{representation homomorphism} (or $G$-map) $\phi$ from the representation $V$ to the representation $W$ is the vector space map (linear map) $\phi: V\to W$ such that
\[
g\cdot\phi(v)=\phi(g\cdot v),\quad\quad\text{for all} \,\, g\in G \,\,\text{and} \,\, v\in V;
\]
that is, the following diagram is commutative:
\begin{displaymath}
    \xymatrix{
        V \ar[r]^\phi \ar[d]_g & W \ar[d]^{g} \\
        V \ar[r]_\phi       & W }
\end{displaymath}
The terms $G$-monomorphism, $G$-epimorphism and $G$-isomorphism are, then, defined in the natural way. If two representations are $G$-isomorphic, we say that they are \textsl{equivalent}\index{equivalent representations}. If $\phi:V\to W$ is a $G$-homomorphism, it is not hard to see that $\ker(\phi)$ and $\im(\phi)$ are $G$-modules. For the following definition we declare that a  subset $S$ of a $G$-module $V$ is said to be \txtsl{invariant} under $G$ if $g\cdot s\in S$ for every $g\in G$ and $s\in S$. A \txtsl{subrepresentation} of a representation (so a $G$-module) $V$ is a vector subspace of $V$ which is invariant under $G$. For example, if $\phi:V\to W$ is a $G$-homomorphism, then $\ker(\phi)$ is a subrepresentation of $V$ and Im$(\phi)$ is a subrepresentation of $W$.

A representation $V$ of a group $G$ is \textsl{irreducible}\index{representation!irreducible} if $V$ has no proper nonzero subrepresentations.  Clearly the trivial representation of any group is irreducible since it has dimension $1$. If $n>1$, then the defining representation of $\sym(n)$ is not irreducible since the 1-dimensional subspace $W$ of $\,\mathbb{C}\{1,\ldots,n\}$ generated by the element $1+\cdots+n$ is $\sym(n)$-invariant.

There are various ways to construct new representations of a group using given representations. For instance, if $V$ and $W$ are representations of $G$, then the \txtsl{direct sum} $V\oplus W$ is a representation via
\[
g(v+w)=gv+gw,\quad\quad \text{for all}\,\, g\in G\,\,\text{and all} \,\, v\in V,\, w\in W,
\]
and the \textsl{tensor product}\index{tensor product of representations} $V\otimes W$ is another representation via
\[
g(v\otimes w)=gv\otimes gw,\quad\quad \text{for all}\,\, g\in G\,\,\text{and all} \,\, v\in V,\, w\in W.
\]
In particular, for a given representation $V$, the summation $\bigoplus_{i=1}^n V$ and the tensor power $V^{\otimes n}$ are also representations for each integer $n\geq 1$.

The irreducible representations, having no smaller representations inside them, roughly speaking, turn out to be the atomic objects in the category of representations of a group. This tempts one to prove that for a given group, all the representations can be written as a direct sum of the irreducible representations. This is, indeed, the content of the \txtsl{complete reducibility  theorem} (also called \textsl{semisimplicity theorem}) due to Maschke. To prove Maschke's theorem, we use the following lemma.
\begin{lem}\label{complement}
If $W$ is a subrepresentation of a representation $V$ for the group $G$, then there is a subrepresentation $W'$ of $V$ such that $V=W\oplus W'$.
\end{lem}
\begin{proof}
Since $V$ is finite dimensional, there is an inner product on $V$. Fix an inner product on $V$ and suppose $U$ is the complementary subspace to $W$ under this product; that is, $U=W^\perp$. Assume, also, that $\pi_0:V\to W$ is the projection given by the direct decomposition $V=W\oplus U$. Define the linear map $\pi: V\to W$ by
\[
\pi(v)=\frac{1}{|G|}\sum_{g\in G}g\cdot \pi_0(g^{-1}\cdot v),\quad\quad\text{for all}\,\, v\in V.
\]
It is easy to see that $\pi$ is a $G$-epimorphism. On the other hand, since $W$ is $G$-invariant, for all $w\in W$, we have
\[
\pi(w)=\frac{1}{|G|}\sum_{g\in G}g\cdot \pi_0(g^{-1}\cdot w)=\frac{1}{|G|}\sum_{g\in G}g\cdot g^{-1}\cdot w=\frac{1}{|G|}\sum_{g\in G}w=w.
\]
Therefore, $\pi$ is the identity on $W$ and satisfies $\pi^2=\pi$; that is, $\pi$ is a projection of $V$ on $W$.
Let $W'=\ker(\pi)$. We show that $V=W\oplus W'$. First, assume $w\in W\cap W'$. Then $w=\pi(w)$ and $\pi(w)=0$; which shows that $W\cap W'=0$. Furthermore, if $v\in V$, then one can write $v= \pi(v)+ \left(v-\pi(v)\right)$, where $\pi(v)\in W$ and since $\pi(v-\pi(v))=\pi(v)-\pi(\pi(v))=\pi(v)-\pi(v)=0$, we have that $\left(v-\pi(v)\right)\in W'$. The conclusion is that $V=W\oplus W'$.
\end{proof}

\begin{thm}[Maschke's Theorem]\label{maschke}
Any representation of a group is a direct sum of irreducible representations.
\end{thm}
\begin{proof}
Let $V$ be a representation. We prove the theorem by induction on  $d$, the dimension of $V$. If $d=1$, then $V$ itself is irreducible and we are done. Now assume $d>1$. If $V$ is irreducible, then there is nothing to prove. Suppose, therefore, that $V$ is not irreducible, and let $W$ be a proper nontrivial subrepresentation of $V$. Using Lemma~\ref{complement}, there is a subrepresentation $W'$ of $V$ such that $V=W\oplus W'$. Since the dimensions of $W$ and $W'$ are less than $d$, the induction hypothesis applies for them and, therefore, the proof is complete.
\end{proof}

Let $\mathcal{X}^{reg}$ be the regular representation of a group $G$. By Maschke's Theorem, we can write 
\begin{equation}\label{all_irrs}
\mathcal{X}^{reg}=\bigoplus_{i} m_i V_i,
\end{equation}
where $V_i$ are all pairwise inequivalent irreducible representations and for any $i$, $m_i$ is the number of repetitions of  the representation $V_i$ (we will see in Example~\ref{some_characters} that $m_i$ is in fact $\dim V_i$). In fact, $V_i$ are all the irreducible representations of $G$; that is, all the irreducible representations of $G$ occur in the decomposition (\ref{all_irrs}); see \cite[Section 1.10]{MR1824028}. Furthermore 
\begin{thm}\label{sum_of_dim_of_irrs} Let $\{V_i\}$ be the set of all irreducible representations of a group $G$. Then
\[
\sum_i(\dim V_i)^2=|G|.\qed
\]
\end{thm}
Throughout this thesis, we denote the set of all irreducible representations of a group $G$ by $\irr(G)$.

The following theorem, which is known as Schur's Lemma for representations, states that the ring of all $G$-homomorphisms between irreducible representations is, in fact, a division ring.

\begin{thm}[{{Schur's Lemma}}]\label{schur}
Let $V$ and $W$ be irreducible representations of $G$ and let $\phi: V\to W$ be a $G$-homomorphism. Then
\\
{\bf (a)} $\phi$ is either zero or an isomorphism.
\\
{\bf (b)} If $V\cong W$, then $\phi=\lambda I$, for some $\lambda\in\mathbb{C}$, where $I$ is the identity map.
\end{thm}
\begin{proof}
To prove part (a), it is enough to note that $\ker(\phi)$ and Im$(\phi)$ are subrepresentations of $V$ and $W$, respectively. For part (b), since $\mathbb{C}$ is algebraically closed, the operator $\phi$ (or, equivalently the matrix corresponding to $\phi$ with respect to a basis of $V$) must have an eigenvalue $\lambda$ in $\mathbb{C}$. This implies that the map $\phi-\lambda I$ has a nonzero kernel. Therefore, by part (a), this map must be zero. Thus $\phi=\lambda I$.
\end{proof}

\begin{example}
In this example we show that, if $G$ is an abelian group, then all the irreducible representations of $G$ are 1-dimensional. To prove this, let $V$ be an irreducible representation of $G$. For each $g\in G$, consider the operator $g: V\to V$. Since $G$ is abelian, for all $h\in G$ we will have
$$g(h(v))=gh(v)=hg(v)=h(g(v)),\quad\quad\text{for all}\,\,v\in V,$$
so that $g$ is a $G$-isomorphism. Therefore, by Theorem~\ref{schur}, the map $g$ acts on $V$ by a scalar multiplication. This implies that all the subspaces of $V$ are $G$-invariant. But the assumption is that $V$ is irreducible; thus $V$ must be 1-dimensional.
\end{example}

Now we define one of the most important concepts of representation theory, namely the  characters.
Let $\mathcal{X}: G\to GL(V)$ be a representation of the group $G$. Then the \txtsl{character} $\chi^V$ (or simply $\chi$) of $\mathcal{X}$ is the complex-valued function on $G$ which maps every $g\in G$ to the trace of a matrix representation of $\mathcal{X}(g)$.

The characters of irreducible representations are, usually, called \textsl{irreducible characters}\index{character!irreducible}. Note that for any representation $V$ we have $\chi^V(\id)=\dim V$.
Note, also, that by the definition, $\chi(hgh^{-1})=\chi(g)$, for all $g,h\in G$. We deduce that the characters are \txtsl{class functions}; that is, they are constant on the conjugacy classes of groups. Furthermore, if two representations of a group $G$ are $G$-isomorphic, then their characters are the same. Surprisingly, the converse is also true. In order to prove this result, we first recall the notion of inner product of the characters. Let $\chi$ and $\psi$ be two characters of a group $G$. We define their \textsl{inner product}\index{inner product of characters} as 
\begin{equation}\label{inner_prod}
\left\langle \chi , \psi\right\rangle = \frac{1}{|G|} \sum_{g\in G} \chi(g) \psi(g^{-1}).
\end{equation}
The following fact shows that the irreducible characters are orthogonal to each other with respect to this inner product.
\begin{thm}[{{see \cite[Section 2.2]{MR1153249}}}]\label{orth_chars} Let $\chi$ and $\psi$ be all of the irreducible characters of a group. Then 
\[
\left\langle \chi , \psi\right\rangle =\begin{cases}
1, & \text{if}\quad \chi=\psi;\\
0, & \text{if} \quad\chi\neq\psi.\qed
\end{cases}\]
\end{thm}
The following result states that the characters contain all the information of the representations. 
\begin{thm}\label{char-classification}
Any representation is determined up to $G$-isomorphism by its character.
\end{thm}
\begin{proof} Let $V$ and $W$ be two representations of a group with the same character $\chi$. Without loss of generality we may assume $V$ and $W$ have the following decomposition to the irreducible representations:
\[
V= \bigoplus_{i=1}^k m_i V_i\quad \text{and}\quad W= \bigoplus_{i=1}^k n_i V_i,
\]
where $V_i$ are irreducible representations. Assume $\chi_i$ are the characters of $V_i$; hence 
\begin{equation}\label{2_char_decomposition}
\chi= m_1 \chi_1+\cdots+m_k\chi_k =  n_1 \chi_1+\cdots+n_k\chi_k.
\end{equation}
According to Theorem~\ref{orth_chars}, if we multiply both sides of the second equality on (\ref{2_char_decomposition}) by $\chi_i$ under the inner product (\ref{inner_prod}), we will get $m_i=n_i$, for all $i\in \{1,\ldots,k\}$. This means that $V$ and $W$ both have the same irreducible representations with the same multiplicities. By Theorem~\ref{orth_chars}, this completes the proof.
\end{proof}
Another useful consequence of Theorem~\ref{orth_chars} is the following which can be easily seen.
\begin{cor}\label{inner_prod_one_then_irr}
A character $\chi$ of a group is irreducible if and only if $\left\langle \chi , \chi\right\rangle=1$.\qed
\end{cor}
Because of this fact, in this thesis, we may sometimes abuse the notation and denote the set of all irreducible characters of a groups $G$ by $\irr(G)$.
The following theorem, which also uses the orthogonality of irreducible characters, is another important application of characters.

\begin{thm}[{{see \cite[Section 2.2]{MR1153249}}}]\label{char-conjugacy} The number of irreducible representations of a group is   the number of conjugacy classes of the group.\qed
\end{thm}

\begin{example}\label{some_characters}
Let $\chi$ be the character of the permutation representation of $G$ with respect to an action on a set $S$. It is not hard to see that $\chi(g)$ is the number of elements of $S$ fixed by $g$, for any $g\in G$.
In particular the character of the trivial representation is always 1, and the character $\chi^{reg}$ of the regular representation of $G$ assumes the following values:
\[
\chi^{reg}(g)=\begin{cases} |G|, & \, \text{if }\, g=\id;\\
0, & \,\text{otherwise.}
\end{cases}
\]
On the other hand, according to (\ref{all_irrs}), we have $\chi^{reg}=m_1\chi_1+\cdots+m_k \chi_k$, where $k$ is the number of non-equivalent irreducible representations of $G$ and $\chi_i$ are the characters of $V_i$. We have  
\[
m_i=\left\langle \chi^{reg} , \chi_i\right\rangle=\frac{1}{|G|} \sum_{g\in G} \chi^{reg}(g) \chi_i(g^{-1})=\frac{1}{|G|} \chi^{reg}(\id) \chi_i(\id)=\chi_i(\id);
\]
thus we have proved that $m_i=\dim V_i$.
\end{example}

We now introduce the standard representation which plays a very important role in this thesis. 
Let $G\leq \sym(n)$ be a permutation group and let $\chi$ be the character of its defining representation  on the set $S=[n]$. Then the \textsl{standard representation}\index{representation!standard} $V$ of $G$ is the representation whose character $\chi_V$ is defined as follows
\begin{equation}\label{char_of_standard}
\chi_V(g)=\chi(g)-1,\quad\text{for any}\,\,g\in G.
\end{equation}
In  other words, the character value of $V$ at a permutation $g\in G$ is the number of points fixed by $g$ minus one. Since the dimension of the defining representation of $G$ is $n$, one can deduce the following from (\ref{char_of_standard}).
\begin{lem}\label{dim_of_standard_char} The standard representation of any permutation group of degree $n$ has dimension $n-1$.\qed
\end{lem}

For the next result, we need to recall the well-known Burnside's lemma. 
\begin{thm}[{{Burnside's lemma}}]\label{burnside} Let $G$ be a group acting on a set $X$. Then the number of orbits of this action is equal to 
\[
\frac{1}{|G|}\sum_{g\in G}|\fix(g)|,
\]
where $\fix(g)$ is the set of elements of $X$ fixed by $g$.\qed
\end{thm}
For a proof of this fact the reader may refer to \cite[Section 2.2]{MR1829620}.
 
We recall that a permutation group $G\leq \sym(n)$  is called $k$-transitive (for $k\geq 1$) if for any pair of ordered $k$-sets $x=(x_1,\dots,x_k)$ and $y=(y_1,\ldots,y_k)$ from $[n]$, there is an element $\sigma\in G$  such that $x^\sigma=y$. In particular, $G$ is 1-transitive (or simply, transitive) if for any $x,y\in [n]$ there is a $\sigma\in G$ with $x^\sigma=y$. For more discussions on transitivity and related concepts the reader may refer to \cite{cameron1999permutation}.

\begin{prop}\label{standard_is_irr}
If $G\leq\sym(n)$ is a $2$-transitive permutation group,  then the standard representation is irreducible.
\end{prop}
\begin{proof}
The group $G$ acts transitively on the set of ordered pairs $(x,y)$, where $x,y\in [n]$; hence this action has exactly two orbits, namely 
\[
O_1=\{(x,x)\,|\, x\in[n]\}\quad\text{and}\quad O_2=\{(x,y)\,|\, x,y\in [n], x\neq y\}.
\]
For any $g\in G$, let $\Fix(g)$ be the set of fixed points of $g$ under this action. It is easily seen that $|\Fix(g)|=|\fix(g)|^2$, for any $g\in G$,  where $\fix(g)$ is the set of fixed points of $g$ in $[n]$ under the natural action of $G$ on $[n]$. Thus according to Burnside's lemma, we have
\begin{equation}\label{Fix_and_fix}
2=\frac{1}{|G|}\sum_{g\in G}|\Fix(g)|= \frac{1}{|G|}\sum_{g\in G}|\fix(g)|^2.
\end{equation}
On the other hand, if $\chi$ is the character of the defining representation of $G$ on $[n]$, we have
\begin{align*}
\left\langle \chi , \chi\right\rangle & =\frac{1}{|G|}\sum_{g\in G} \chi(g)\chi(g^{-1})= \frac{1}{|G|}\sum_{g\in G} |\fix(g)|^2;
\end{align*}
therefore, using (\ref{Fix_and_fix}), $\left\langle \chi , \chi\right\rangle=2$. But $\chi=\chi_V+1$, which implies that
\[
\left\langle \chi_V , \chi_V\right\rangle+1=2,
\]
that is $\left\langle \chi_V , \chi_V\right\rangle=1$. Thus, the proposition follows from Corollary~\ref{inner_prod_one_then_irr}.
\end{proof}

Let $H$ be a subgroup of a group $G$. Then it is a natural question to ask if there are any relationships between the representations of $G$ and those of $H$. The answer for this question is nested in the concepts of restricted and induced representations. We close this section by recalling  their definitions. 
 
Let $\mathcal{X}$ be a representation of $G$ with  character $\chi$. Then the \textsl{restriction}\index{representation!restricted} of $\mathcal{X}$ to $H$, denoted by $\mathcal{X}\downarrow_H^G$, is the  representation of $H$ given by  
\[
\mathcal{X}\downarrow_H^G(h)=\mathcal{X}(h), \quad \text{for all }\, h\in H.
\]
Then the character of $\mathcal{X}\downarrow_H^G$, denoted by $\chi\downarrow_H^G$, is simply given by
\[
\chi\downarrow_H^G(h)=\chi(h),\quad \text{for all }\, h\in H.
\]
On the other hand, if $\mathcal{Y}$ is a representation of $H$ with the character $\psi$, then the corresponding \textsl{induced representation}\index{representation!induced} of $G$, denoted by $\mathcal{Y}\uparrow_H^G$, is the representation of $G$ whose character is given by
\begin{equation}\label{char_of_induced} 
\psi\uparrow_H^G(g)=\frac{1}{|H|}\sum_{x\in G} \psi(x^{-1} g x), \quad \text{for any} \, g\in G,
\end{equation}
where $\psi(g)$ is assumed to be zero if $g\notin H$. The reader may refer to \cite[Section 3.3]{MR1153249} or \cite[Section 1.12]{MR1824028} for alternative definitions and more detailed discussions.


\section{Representations of the symmetric group}\label{rep_sym}

In this section we investigate all the irreducible representations of the group $\sym(n)$. For  more detailed discussions, the reader may refer to \cite[Chapter 4]{MR1153249} or \cite[Chapter 2]{MR1824028}.
Note that, by Theorem~\ref{char-conjugacy}, the number of irreducible representations of a group is  equal to the number of conjugacy classes of the group. It is well-known that there is a one-to-one correspondence between the conjugacy classes of $\sym(n)$ and the \textsl{partitions}\index{partition} of $n$. A partition $\lambda$ of $n$, denoted by $\lambda \vdash n$, is a weakly decreasing sequence of positive integers $\lambda=[\lambda_1,\ldots,\lambda_k]$ such that $1\leq k \leq n$ and $\,\lambda_1+\cdots+\lambda_k=n$. All the elements of a conjugacy class of $\sym(n)$ have a fixed cycle structure which can be described by  a unique partition of $n$: a conjugacy class $C$ corresponds to  the partition $\lambda=[\lambda_1,\ldots,\lambda_k]$ if and only if every permutation $\pi\in C$ has a cycle decomposition of form $\pi=\pi_1\pi_2\cdots\pi_k$, where $\pi_1,\ldots,\pi_k$ are mutually disjoint and $\pi_i$ is a cyclic permutation of length $\lambda_i$.  For example, the conjugacy class of $\sym(9)$ containing  the element $(3\,\,9)(4\,\,1\,\,6\,\,7)(2\,\,8)$ corresponds to the partition $\lambda=[4,2,2,1]$.
In this section we present a method to associate an appropriate irreducible representation to a partition.

To a partition $\lambda=[\lambda_1,\ldots,\lambda_k]$ of $n$, we associate a \txtsl{Young diagram} (or Ferrers diagram or Young frame),  which is an array of $n$ boxes having $k$ left-justified rows with row $i$ containing $\lambda_i$ boxes, for $1\leq i \leq k$. This diagram is said to be of \txtsl{shape} $\lambda$.  As an example, the Young diagram of the partition  $\lambda=[5,3,3,2,1,1]$ of 15 is as follows:

\begin{figure}[H]
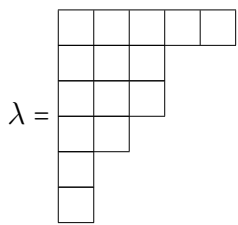

\[
\lambda=
\yng(5,3,3,2,1,1)
\]
\caption{Young diagram for $\lambda=[5,3,3,2,1,1]$}
\end{figure}
The \textsl{transpose}\index{transpose of a partition} partition $\hat{\lambda}$ of a partition $\lambda$ is defined by interchanging the rows and columns in the Young diagram of $\lambda$. For example, if $\lambda$ is as above, then its transpose will be $\hat{\lambda}=[6,4,3,1,1]$.

\begin{figure}[H]
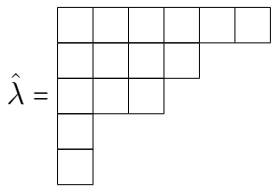

\[
\hat{\lambda}=\yng(6,4,3,1,1)
\]
\caption{Transpose of partition $\lambda=[5,3,3,2,1,1]$}
\end{figure}
Note that the transpose partition is also known as the conjugate partition, and is also denoted by $\lambda^*$. A partition $\lambda$ is said to be \textsl{symmetric}\index{partition!symmetric} if $\lambda=\hat{\lambda}$.

We often use the ``multiplicity'' notation for writing partitions; that is, for example, we may write $[5,3^2,2,1^2]$ instead of $[5,3,3,2,1,1]$.

We now  provide the tools to answer the main question of this section; that is, what are the irreducible representations of $\sym(n)$?
Given a partition $\lambda\vdash n$, a \txtsl{Young tableau}  $t$ of shape $\lambda$, is the Young diagram of $\lambda$ with its boxes  filled with the numbers $1,2,\ldots,n$ with some arrangement. Thus different  arrangements of the numbers $1,2,\ldots,n$ in the boxes of the Young diagram of $\lambda$ will result in a different Young tableaux. Two tableaux $t_1$ and $t_2$ of shape $\lambda$ are \txtsl{row equivalent}, denoted $t_1\sim t_2$, if the corresponding rows of the two tableaux contain the same elements (possibly in different orders). A \txtsl{tabloid} of shape $\lambda$, denoted by $\{t\}$ will, then, be defined as
\[
\{t\}=\{t_1\,:\, t_1\sim t\},
\]
where $t$ is a tableau of shape~$\lambda$.

Assume that  a tableau $t$ has columns $C_1,\ldots, C_l$. Define the subgroup $C_t$ of $\sym(n)$ as follows:
\[
C_t=\sym(C_1)\times\cdots\times \sym(C_l),
\]
where $\sym(C_i)$ is the group of all permutations on the set $C_i$. The subgroup $C_t$ is referred to as the \txtsl{column-stabilizer} of $t$.
Then define the element $\kappa_t$ in the group algebra $\mathbb{C}[\sym(n)]$ as follows:
\[
\kappa_t=\sum_{\pi\in C_t}\sgn(\pi)\pi,
\]
where sgn$(\pi)$ is the sign of the permutation $\pi$ (i.e. sgn$(\pi)$ is 1 if $\pi$ is even and $-1$ otherwise).
Furthermore, if we define the action of a permutation $\pi$ on a tabloid $\{t\}$ as $\{\pi(t)\}$, then the \txtsl{polytabloid} associated to $t$ is defined as
\[
\mathbf{e_t}=\kappa_t\{t\},
\]
which is a member of the group algebra on $\mathbb{C}$ generated by all tabloids $\{t\}$ of shape~$\lambda$.

To illustrate these concepts, consider the partition $\lambda=[3,2]\vdash 5$, and
\Yboxdim13pt
\[
t=\young(412,35)
\]
which is a tableau of shape $\lambda$. We have
\Yboxdim8pt
\[
\{t\}=\left\{ \tiny\young(124,35)\,,\, \young(124,53)\,,\,\young(142,35)\,,\, \young(142,53)\,,\, \young(214,35)\,,\, \young(214,53)\,,\,\young(241,35)\, ,\,\young(241,53)\,,\, \young(412,35)\, ,\, \young(412,53)\,,\, \young(421,35)\,,\, \young(421,53)\right\},	
\]
\Yboxdim13pt
and
\[
C_t=\sym(\{3,4\})\times \sym(\{1,5\})\times \sym(\{2\}).
\]
In addition,
\[
\kappa_t=\id-(3\,\,4)-(1\,\,5)+(3\,\,4)(1\,\,5),
\]
and, hence,
\[
\mathbf{e_t}=\left\{\,\small\young(412,35)\,\right\}-\left\{\,\small\young(312,45)\,\right\}-\left\{\, \small\young(452,31) \,\right\}+\left\{\,\small \young(352,41)  \, \right\}.
\]

Finally, for any partition $\lambda\vdash n$, we define the corresponding \textsl{Specht module}\index{module!Specht}, $S^\lambda$, as the algebra on $\mathbb{C}$ generated by all the polytabloids $\mathbf{e_t}$, where $t$ is of shape $\lambda$; that is,
\[
S^\lambda=\mathbb{C}\{\mathbf{e_t}\,\,|\,\, t\,\,\text{is of shape }\,\lambda\}.
\]
Then for any $\pi \in\sym(n)$, $\pi$ defines an endomorphism on $S^{\lambda}$ by $\pi(\mathbf{e_t})=\mathbf{e_{\pi(t)}}$.
The following important theorem is the answer for our main question (a proof of which appears in \cite[Section 2.4]{MR1824028}).
\begin{thm}
For any $\lambda\vdash n$, the corresponding Specht module $S^\lambda$ is an irreducible representation of $\sym(n)$. Furthermore, if $\lambda,\mu\vdash n$ and $\lambda\neq \mu$, then $S^\lambda\ncong S^\mu$.\qed
\end{thm}

Therefore, using Theorem~\ref{char-conjugacy} and the fact that there is a one-to-one correspondence between the conjugacy classes of $\sym(n)$ and the partitions $\lambda\vdash n$, we conclude that  the Specht modules are all the irreducible representations of $\sym(n)$. In other words
\begin{cor}\label{list_of_irrs_of_sym}
The representations $S^\lambda$ for $\lambda\vdash n$ form a complete list of irreducible representations of $\sym(n)$ over $\mathbb{C}$.\qed
\end{cor}
In this thesis, for any partition $\lambda\vdash n$, we denote the character of the representation $S^{\lambda}$ of $\sym(n)$ by $\chi^{\lambda}$.
\begin{example}
Consider the partition $\lambda=[n]\vdash n$. The irreducible representation of $\sym(n)$ associated to $\lambda$ (i.e. the Specht module corresponding to $\lambda$) is the trivial representation as the only tabloid of shape $\lambda$ is
\[
\{t\}=\left\{\, \young(12\cdots n)\, \right\};
\]
thus $\mathbf{e_t}=\{\,\young(12\cdots n)\,\}$ and so $S^{\lambda}=\mathbb{C}\left\{\,\{\young(12\cdots n)\}\,\right\}$ and the action of $\sym(n)$ on this space will be the trivial action, $\pi \mathbf{e_t}=\mathbf{e_t}$.
\end{example}
\begin{example} Consider the partition $\lambda=[1^n]\vdash n$. One can see that the irreducible representation of $\sym(n)$ associated to $\lambda$ is
\[
S^{\lambda}=\mathbb{C}\{\mathbf{e_{t_0}}\},
\]
where
\newcommand{\tvd}{{\tiny \vdots}}
\[
t_0=\young(1,2,\,,n)\,,
\]
with the action $\pi \mathbf{e_{t_0}}=$sgn$(\pi)\mathbf{e_{t_0}}$, for any $\pi\in \sym(n)$. To see this, note that for any $\pi\in \sym(n)$, 
\[
\pi \mathbf{e_{t_0}} = \sum_{\sigma\in\sym(n)} (\sgn \sigma)\pi \sigma\{t_0\}= \sgn(\pi^{-1}) \sum_{\tau\in \sym(n)}(\sgn\tau)\tau\{t_0\}=\sgn (\pi) \mathbf{e_{t_0}}.
\]
 This representation is usually called the \textsl{sign}\index{representation!sign} or \textsl{alternating representation}\index{representation!alternating}. In this representation, every even permutation is mapped to the identity  automorphism  ${\bf Id}:\mathbb{C}\{\mathbf{e_{t_0}}\}\to  \mathbb{C}\{\mathbf{e_{t_0}}\}$ and every odd permutation is mapped to the automorphism ${\bf -Id}:\mathbb{C}\{\mathbf{e_{t_0}}\}\to  \mathbb{C}\{\mathbf{e_{t_0}}\}$, where $({\bf -Id})(\mathbf{e_{t_0}})=-\mathbf{e_{t_0}}$.
\end{example}

\begin{example}\label{standard rep} Consider the partition $\lambda=[n-1,1]\vdash n$. One can see that the irreducible representation of $\sym(n)$ associated to $\lambda$ is the $(n-1)$-dimensional space generated by  
\[
\{\alpha_2-\alpha_1,\alpha_{3}-\alpha_1,\ldots,\alpha_n-\alpha_1\},
\]
where $\alpha_k$ corresponds to the tabloid
\[
\{t_k\}=\left\{\,\young(1\cdots{k-1}{k+1}\cdots n,k)\,\right\}\,
\]
This representation is, in fact,  the standard representation of $\sym(n)$; that is, it can be shown that $\chi^{\lambda}(\pi)$ is the number of fixed points of $\pi$ minus 1, for all $\pi\in \sym(n)$.
\end{example}

\section{Two formulas}\label{MNRule_and_hook_length_formula}
In this section, we recall two well-known  formulas from representation theory of the symmetric group which will be very useful in  future chapters. The first one is the \txtsl{Murnaghan-Nakayama rule}.  Before stating this result, we need to introduce some notation.

Partitions $\lambda\vdash n$ of the form $\lambda=[\lambda_1,1^{n-\lambda_1}]$ and $[\lambda_1,2,1^{n-\lambda_1-2}]$,  for $\lambda_1>1$, are called \textsl{hooks}\index{hook} and \textsl{near-hooks}\index{hook!near}, respectively.
The $(i,j)$-block\index{block} in a Young diagram is the block in the $i$-th row (from the top) and the $j$-th column (from the left). If a Young diagram contains an $(i,j)$-block but not a $(i+1,j+1)$-block, then the $(i,j)$-block is part of what is called the \txtsl{boundary} of the Young diagram. A \textsl{skew hook}\index{hook!skew}  of $\lambda$ is an edge-wise connected part (meaning that all blocks are either side by side or one below the other)  of the boundary blocks with the property that removing them leaves a smaller proper Young diagram. Figure~\ref{skew} shows all the skew hooks of length $4$ in $\lambda=[5,4,4,2,1,1]$:

\begin{figure}[!ht]
\centering
\includegraphics[width=0.28\textwidth]{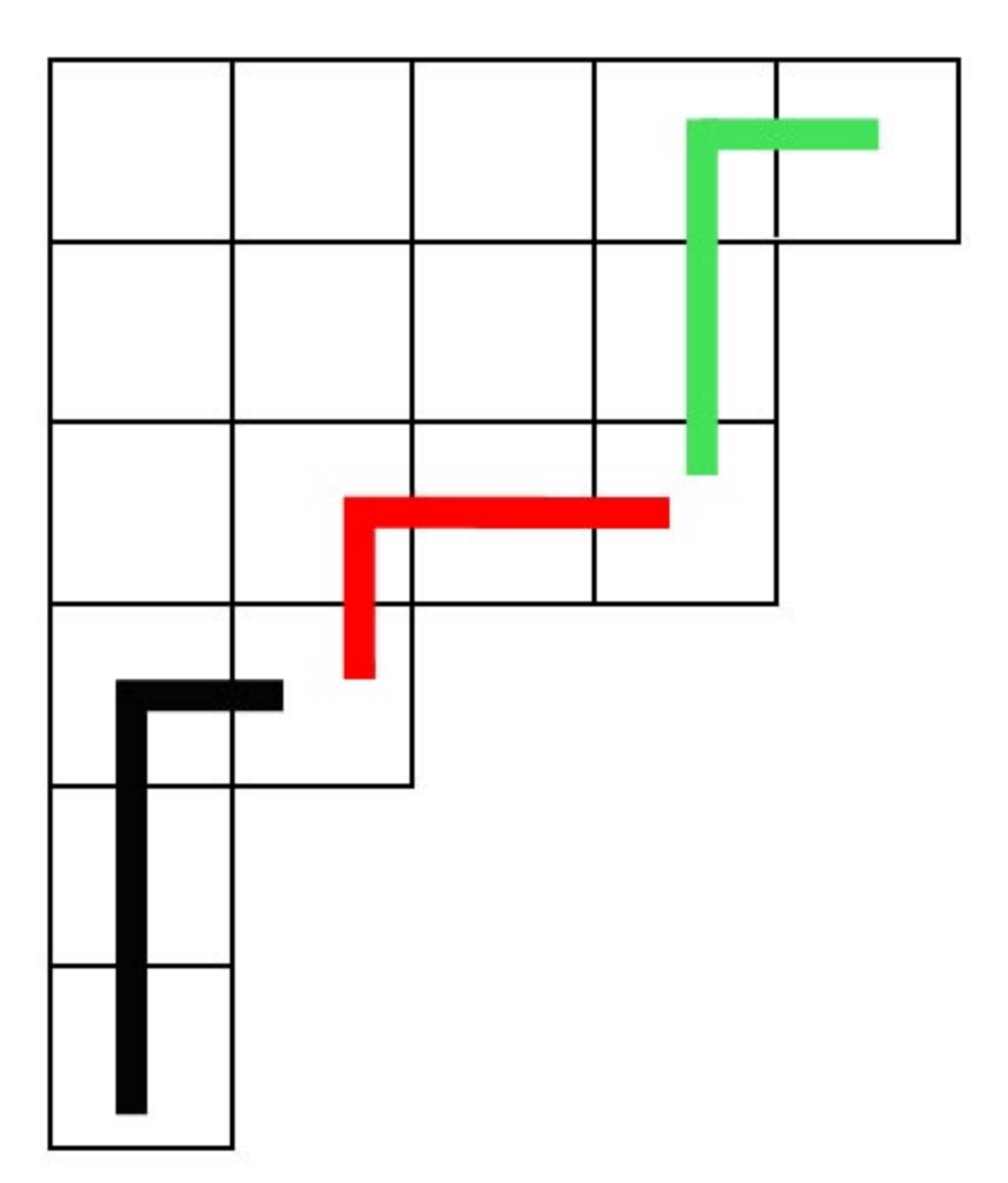}
\\
\caption{All skew hooks of length 4 in $\lambda=[5,4,4,2,1,1]$}
\label{skew}
\end{figure}

\begin{thm}\label{nakayama}(Murnaghan-Nakayama rule) If $\lambda \vdash n$ and $\sigma \in \sym(n)$ can be written as a product of an $m$-cycle and  a disjoint permutation $h\in \sym(n-m)$, then
\[
\chi^{\lambda}(\sigma)=\sum_{\mu} (-1)^{r(\mu)}\chi^{\mu}(h),
\]
where the sum is over all partitions $\mu$ of $n-m$ that are obtained from $\lambda$ by removing a skew hook of length $m$, and
$r(\mu)$ is one less than the number of rows of the removed skew hook.\qed
\end{thm}
For a proof of this theorem, the reader may refer to \cite[Theorem 4.10.2]{MR1824028}.  For general partitions $\lambda$ and permutations $\sigma$ of general cycle structure, often the Murnaghan-Nakayama rule is not so practical; however for some specific cases, one can easily derive the character value using this rule. The following, for example, are two easy cases; they will be used in Section~\ref{EKR_for_alt}.
\begin{cor}\label{cor:appofMN}
Let $\sigma$ is an $n$-cycle in $\sym(n)$. Then for any partition $\lambda\vdash n$, we have 
\[
\chi^\lambda(\sigma) =
 \begin{cases}
 (-1)^{n-\lambda_1},   & \textrm{if $\lambda=[\lambda_1,1^{n-\lambda_1}]$};\\
  0,                 &\textrm{otherwise.} \qed
 \end{cases}
\]
\end{cor}
\begin{cor}\label{cor:appofMNtwo}
Let $n$ be even and  $\sigma$ be the product of two disjoint $n/2$-cycles in $\sym(n)$. Then 
\[
\chi^\lambda(\sigma) \in \{0,\pm 1,\pm 2\},
\]
for any partition $\lambda\vdash n$.
\end{cor}
\begin{proof}
According to Corollary~\ref{cor:appofMN}, $\chi^{\mu}(h)$ in the Murnaghan-Nakayama rule is either $0$, $1$ or $-1$. Hence if $|\chi^\lambda(\sigma)|>2$, then $\lambda$ must have more than two skew hooks of length $n/2$ which is not possible.
\end{proof}

Let $\lambda \vdash n$. For any box of the Young diagram of $\lambda$, the corresponding \txtsl{hook length} is one plus the number of boxes horizontally to the right and vertically below the box.  Define $\hl(\lambda)$ to be the product of
all hook lengths of $\lambda$. We, next, state the \txtsl{hook length formula} which gives a way  to evaluate the dimension of the representation of $\sym(n)$ corresponding to $\lambda$, in terms of $\hl(\lambda)$. Its proof  relies on the Frobenius formula (see \cite[Section 4.1]{MR1153249}).

\begin{thm}[Hook length formula]\label{hl_formula}
If $\lambda\vdash n$, then the dimension of the irreducible representation of $\sym(n)$ corresponding to $\lambda$ is $n!/\hl(\lambda)$.\qed
\end{thm}
To illustrate this formula, consider the partition $\lambda=[5,3,3,2,1]\vdash 14$ below. The hook length of each box in the Young diagram  for $\lambda$ is written in the box.

\begin{figure}[ht!]
\[
\young(97521,642,531,31,1)
\]
\caption{Hook lengths $\lambda=[5,3,3,2,1]$}
\end{figure}
Therefore,
\[
\chi^{\lambda}(\id)=\frac{14!}{9\times7\times5\times 2\times6\times4\times2\times5\times3\times3}=64064.
\]
\smallskip

The following are easy consequences of the hook length formula.
\begin{cor}\label{dim_oftranspos} For any partition $\lambda$ of $n$, we have $\chi^{\lambda}(\id)=\chi^{\hat{\lambda}}(\id)$.\qed
\end{cor}
\begin{cor}\label{dim_non-identities} Let $\lambda\vdash n$. Then $\chi^{\lambda}(\id)=1$ if and only if $\lambda=[n]$ or $[1^n]$.\qed
\end{cor}

\section{Representations of the alternating group}\label{rep_alt}

This section includes some notes about the irreducible representations (or characters) of the alternating group $\alt(n)$. In general, if $H$ is subgroup of a group $G$ of index $2$, then one can obtain $\irr(H)$ using $\irr(G)$. In our particular case (i.e. $G=\sym(n)$ and $H=\alt(n)$), the method is explained below. Our main reference for this part is \cite[Section 5.1]{MR1153249}. We start with the following theorem which is a particular case of Proposition 5.1 in \cite{MR1153249}.  

\begin{thm}\label{reps_of_alt} Let $\lambda$ be a partition of $n$ and let $W$ and $\widehat{W}$ be the restrictions of $S^\lambda$ and $S^{\hat{\lambda}}$ to $\alt(n)$, respectively. Then
\begin{enumerate}[(a)]
\item if $\lambda$ is not symmetric, then $W$ is an irreducible representation of $\alt(n)$ and is isomorphic to $\widehat{W}$; and
\item if $\lambda$ is symmetric, then $W=W'\oplus W''$, where $W'$ and $W''$ are irreducible but not isomorphic representations of $\alt(n)$.
\end{enumerate}
All the irreducible representations of $\alt(n)$ arise uniquely in this way.\qed
\end{thm}
Throughout this thesis, we will use the notation of Theorem~\ref{reps_of_alt}.

For any conjugacy class $c$ of $\alt(n)$, either $c$ is also a conjugacy class in $\sym(n)$ or $c\cup c'$ is a conjugacy class in $\sym(n)$, where $c'=tct^{-1}$, for some $t\notin \alt(n)$.  The second type of conjugacy classes are said to be \textsl{split}\index{split conjugacy class}. A conjugacy class $c$ of $\alt(n)$ is split if and only if all the cycles in the cycle  decomposition of an element of $c$ have odd length and no two cycles have the same length.

Suppose $c$ is a conjugacy class of $\sym(n)$ that is not a conjugacy class in $\alt(n)$. Assume that the decomposition of an element of $c$ contains cycles of odd lengths $q_1>q_2>\cdots>q_r$. Then we say $c$ \textsl{corresponds} to the symmetric partition
$\lambda=[\lambda_1,\lambda_2,\ldots]$ of $n$ if $q_1=2\lambda_1-1$, $q_2=2\lambda_2-3$, $q_3=2\lambda_3-5,\ldots$. This is a
correspondence between a split conjugacy classes of $\alt(n)$ and the symmetric partitions  of $n$. 
\begin{example}
Consider the conjugacy class of $\sym(23)$ containing the element 
\[
(1\quad 2\quad\cdots \quad 11)\, (12\quad 13\quad\cdots\quad 20)\,(21\quad 22\quad 23).
\]
In this case, $q_1=11$, $q_2=9$ and $q_3=3$. Thus $\lambda_1=6$, $\lambda_2=6$ and $\lambda_3=4$. Since $\lambda$ must be symmetric, this implies that $\lambda=[6,6,4,3,2,2]$. See Figure~\ref{correspondence}.

\begin{figure}[ht!]
\centering
\[ c:\underbrace{(\quad\cdots\quad)}_{11} \underbrace{(\quad\cdots\quad)}_{9} 
  \underbrace{(\quad\cdots\quad)}_{3}\quad\longrightarrow \quad \yng(6,6,4,3,2,2)\]
\caption{Split conjugacy classes and symmetric partitions}
\label{correspondence}
\end{figure}
\end{example}
Using this correspondence, we can give equations for the irreducible characters of $\alt(n)$ in terms of the characters of $\sym(n)$.  This result is also proved in \cite[Section 5.1]{MR1153249}.

\begin{thm}\label{char_values_of_alt}
Let $\lambda$ be a partition of $n$ and let $\chi^\lambda$ be the character of $S^\lambda$ for $\sym(n)$. Assume $c$ is a non-split conjugacy   class  of $\alt(n)$ and $c'\cup c''$ is a pair of split conjugacy classes in  $\alt(n)$. Let $\sigma\in c$, $\sigma'\in c'$, $\sigma''\in c''$ and $\bar{\sigma}\in c'\cup c''$. 
\begin{enumerate}[(a)]
\item If $\lambda$ is not symmetric, let $\chi_\lambda$ be the character of $\alt(n)$ corresponding to $W$, then
\[ 
\chi_\lambda(\sigma)=\chi^\lambda(\sigma)\quad\text{and}\quad\chi_\lambda(\sigma')=\chi_\lambda(\sigma'')=\chi^\lambda(\bar{\sigma}).
\]
\item If $\lambda$ is symmetric, let $\chi_\lambda'$ and $\chi_\lambda''$ be the characters of $\alt(n)$ corresponding to $W'$ and $W''$, respectively, then
\[
\chi_\lambda'(\sigma)=\chi_\lambda''(\sigma)=\frac{1}{2}\chi^\lambda(\sigma),
\] 
and
\begin{enumerate}[(i)]
\item if $c'\cup c''$ does not correspond to $\lambda$ then 
\[
\chi_\lambda'(\sigma')= \chi_\lambda'(\sigma'')= \chi_\lambda''(\sigma')=\chi_\lambda''(\sigma'')=\frac{1}{2}\chi^\lambda(\bar{\sigma}).
\]
\item if $c'\cup c''$ corresponds to $\lambda$, then
\[ 
\chi_\lambda'(\sigma')= \chi_\lambda''(\sigma'')=x\quad\text{and}\quad \chi_\lambda'(\sigma'')=\chi_\lambda''(\sigma')=y.
\]
\end{enumerate}
The values of $x$ and $y$ are
\[
\frac{1}{2}\left[(-1)^m\pm\sqrt{(-1)^m q_1\dotsm q_r}\right],
\]
where $m=\frac{n-r}{2}$ and the cycle decomposition of an element of $c'\cup c''$ has cycles of odd lengths $q_1,\ldots,q_r$.\qed
\end{enumerate}
\end{thm}
We point out that  part (b) of the theorem  contains two cases: the case where  the conjugacy class is non-split (the conjugacy class is $c$) and the case where the conjugacy class is split (the conjugacy class is  $c'\cup c''$). The latter case, in turn, has two cases: the case where the split conjugacy class does not correspond to $\lambda$ and the case where it does correspond to $\lambda$. 
We will, also, use the notation of Theorem~\ref{char_values_of_alt} throughout the thesis and hence want to emphasize that for
representations of $\sym(n)$ we use $\lambda$ as a superscript and for representations of $\alt(n)$, the $\lambda$ is a subscript.

\section{Two-layer hooks}\label{sec:twolayerhooks}
In this section we  define a new type of partition, namely ``two-layer hooks'' and study some of their properties. This machinery will be useful in Section~\ref{EKR_for_alt}. Assume $\lambda=[\lambda_1,\ldots,\lambda_k]$ is a partition of $n$ such that $k\geq 3$, $\lambda_2+\hat{\lambda}_2\geq 5$, $\lambda_3\leq 2$ and $\lambda_1-\lambda_2=\hat{\lambda}_1-\hat{\lambda}_2>0$. Then we say $\lambda$ is a \textsl{two-layer hook}\index{hook!two-layer}.  In fact, a two-layer hook is a partition whose Young diagram is obtained by ``appropriately gluing'' two hooks of lengths greater than $1$.  See Figure~\ref{two_layer_hook} for some examples of two-layer hooks.

\begin{figure}[H]
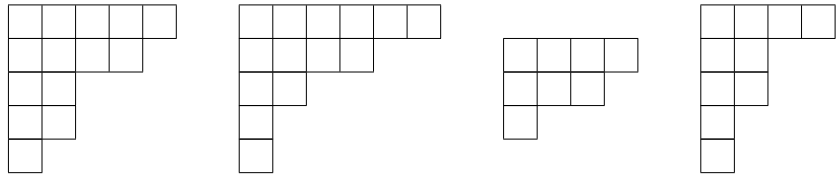

\[\yng(5,4,2,2,1)\quad\quad \yng(6,4,2,1,1)\quad\quad \yng(4,3,1)\quad\quad \yng(4,2,2,1,1) \]
\caption{Two-layer hooks}
\label{two_layer_hook}
\end{figure}
Note that if $\lambda \vdash n$ is a two-layer hook, then $\hat{\lambda}$ is also a two-layer hook and that by the definition,
\[
n=\lambda_1+\lambda_2+\hat{\lambda}_1+\hat{\lambda}_2-4=2(\lambda_1+\hat{\lambda}_2)-4,
\]
which implies that $n$ must be an even number greater than or equal to $8$. Note also that a near hook is not a two-layer hook. 

\begin{lem}\label{char_of_two_layer_hook}
Let $\lambda$ be a partition of $n$ and let $\sigma$ be a  permutation in $\sym(n)$ that is the product of two disjoint $n/2$-cycles. If $\chi^\lambda(\sigma)=-2$, then $\lambda$ is either a two-layer hook or a symmetric near hook.
\end{lem}
\begin{proof}
According to the Murnaghan-Nakayama rule and  Corollary~\ref{cor:appofMN}, $\lambda$ should have two  skew-hooks of length $n/2$ and deleting each of them should leave a hook of length $n/2$. If we denote $\lambda =[\lambda_1,\ldots,\lambda_k]$, then this obviously implies that $k>1$.

If $k=2$, then $\lambda$ must be the  partition $[\frac{n}{2},\frac{n}{2}]$ (since if $\lambda=[\lambda_1,\lambda_2]$, where $\lambda_1>\lambda_2$, then $\lambda$ will not have two skew-hooks of  length $n/2$); in this case we can calculate the character value at $\sigma$ to be $2$.  Thus $k\geq 3$.  

If $\lambda_3>2$, then the partition $\lambda'$ obtained from $\lambda$ by deleting any skew-hook will have $\lambda'_2\geq 2$   which implies that $\lambda'$ is not a hook. Thus $\lambda_3\leq 2$.

Let $\lambda_1-\lambda_2=s$ and $\hat{\lambda}_1-\hat{\lambda}_2=t$. Assume $\mu$ and $\nu$ are the two skew hooks of $\lambda$ of length $n/2$. Since they have length $n/2$, we may assume that $\mu$ contains the last box of the first row and $\nu$ contains the last box of the first column. The lengths  of $\mu$ and $\nu$ being both equal to $n/2$ implies that 
\[
(s+1) +(\lambda_2-1) +(\hat{\lambda}_2-1)-1= (t+1)+(\hat{\lambda}_2-1)+(\lambda_2-1)-1,
\]
which yields $s=t$.

If $s=t=0$ then $\lambda=[\lambda_1,\lambda_1, 2, \dots 2]$. If we  denote the number of rows in $\lambda$ by $k$, then according to the Murnaghan-Nakayama rule,
\begin{align*}
\chi^\lambda(\sigma) &= (-1)^{r(\mu)} (-1)^{r(\lambda\backslash \mu)}\,+\, (-1)^{r(\nu)} (-1)^{r(\lambda\backslash \nu)}\\
&= (-1)^{k} (-1)^{k}\,+\, (-1)^{k-1} (-1)^{k-1}\\ 
&=2. 
\end{align*}

Finally, note that if $\lambda_2+\hat{\lambda}_2<5$, then either $\lambda_2+\hat{\lambda}_2=2$ or $4$. In the former case, $\lambda$ is a hook and obviously it cannot have two skew-hooks of length $n/2$. In the latter case, $\lambda$ must be a near hook. If it is not symmetric then it cannot have two skew-hooks.

These imply that if $\lambda$ is neither a symmetric near hook nor a two layer hook, then $\chi^\lambda(\sigma)\neq -2$; this completes the proof. 
\end{proof}

The following lemma provides a lower bound on the dimension of a symmetric near hook.
\begin{lem}\label{dimensions_of_near_hooks} 
If a symmetric partition $\lambda$ of $n\geq 8$ is a near hook, then  $\chi^\lambda(\id)> 2n-2$.
\end{lem}
\begin{proof}
Since $\lambda$ is a symmetric near hook we know that $\lambda=[n/2,2,1^{\frac{n}{2}-2}]$ and we can calculate the hook lengths directly 
\begin{align*}
\hl(\lambda)&=(n-1)\left(\frac{n}{2}\,\right)^2\left[\left(\frac{n}{2}-2\right)!\right]^2\\[.2cm]
&\leq (n-1)\,\frac{n^2}{4}\,(n-4)!=\frac{n(n-1)}{2(n-2)(n-3)}\,\,\frac{n(n-2)!}{2}\\[.2cm]
&<\frac{n(n-2)!}{2}
\end{align*}
since $n\geq 8$. Putting this bound into the hook length formula (Theorem~\ref{hl_formula}) gives the lemma.
\end{proof}

Next we prove that the same lower bound holds for the dimension of a two-layer hook. 
\begin{lem}\label{dimensions_of_2_layer_hooks}
If a partition $\lambda$ is a two-layer hook, then $\chi^\lambda(\id)> 2n-2$.
\end{lem}
\begin{proof}
Let $\lambda=[\lambda_1,\lambda_2,\ldots,\lambda_k]$. According to the hook length formula, it suffices to show that $\hl(\lambda)<n(n-2)!/2$.  We proceed by induction on $n$. It is easy to see the lemma is true for $n=8$. Let $n\geq 10$ and without loss of generality assume that $\lambda_1>\hat{\lambda}_1$. This implies $\lambda_2\geq 3$. We compute
\begin{align*}
\hl(\lambda)=&(\lambda_1+\hat{\lambda}_1-1)(\lambda_1+\hat{\lambda}_2-2) (\lambda_2+\hat{\lambda}_1-2) (\lambda_2+\hat{\lambda}_2-3)\\ &\cdot\frac{(\lambda_1-1)!}{s+1}\frac{(\hat{\lambda}_1-1)!}{s+1}(\lambda_2-2)! (\hat{\lambda}_2-2)!
\end{align*}
where $s$ is as in Lemma~\ref{char_of_two_layer_hook}. One can re-write this as
\begin{align}
\hl(\lambda)
=& \frac{\lambda_1+\hat{\lambda}_1\!-\!1}{\lambda_1+\hat{\lambda}_1\!-\!2}\,\cdot\,\frac{\lambda_1+\hat{\lambda}_2\!-\!2}{\lambda_1+\hat{\lambda}_2\!-\!3} \,\cdot\,\frac{\lambda_2+\hat{\lambda}_1\!-\!2}{\lambda_2+\hat{\lambda}_1\!-\!3}\,\cdot\, \frac{\lambda_2+\hat{\lambda}_2\!-\!3}{\lambda_2+\hat{\lambda}_2\!-\!4} 
 \cdot(\lambda_1\!-\!1)(\lambda_2\!-\!2) \nonumber \\
& \left[\vphantom{\frac12}(\lambda_1+\hat{\lambda}_1-2)(\lambda_1+\hat{\lambda}_2-3)(\lambda_2+\hat{\lambda}_1-3)(\lambda_2+\hat{\lambda}_2-4)\right.\nonumber \\
& \quad \left.\cdot   \frac{(\lambda_1-2)!}{s+1}\frac{(\hat{\lambda}_1-1)!}{s+1}(\lambda_2-3)! (\hat{\lambda}_2-2)!\right]\nonumber \\
=& \frac{\lambda_1+\hat{\lambda}_1-1}{\lambda_1+\hat{\lambda}_1-2}\,\cdot\,\frac{\lambda_1+\hat{\lambda}_2-2}{\lambda_1+\hat{\lambda}_2-3} \,\cdot\,\frac{\lambda_2+\hat{\lambda}_1-2}{\lambda_2+\hat{\lambda}_1-3}\,\cdot\, \frac{\lambda_2+\hat{\lambda}_2-3}{\lambda_2+\hat{\lambda}_2-4} \nonumber \\ 
& \cdot(\lambda_1-1)(\lambda_2-2)  \hl(\tilde{\lambda}),\label{tilde_lambda}
\end{align}
where $\tilde{\lambda}=[\lambda_1-1,\lambda_2-1,\lambda_3,\ldots,\lambda_k]$ is the partition whose Young diagram is obtained from that of $\lambda$ by removing the last boxes of the first and the second rows. We can simplify (\ref{tilde_lambda}) as

\begin{align}
\hl(\lambda)=& \left(1+\frac{1}{\lambda_1+\hat{\lambda}_1-2}\right)\,\left(1+\frac{1}{\lambda_1+\hat{\lambda}_2-3} \right) \,\left(1+\frac{1}{\lambda_2+\hat{\lambda}_1-3}\right)\nonumber\\
 &  \left(1+\frac{1}{\lambda_2+\hat{\lambda}_2-4}\right)(\lambda_1-1)(\lambda_2-2)\cdot \hl(\tilde{\lambda}).\label{hl_lambda}
\end{align}

The partition $\tilde{\lambda}$ is either a near hook or a two-layer hook. In the first case, because $\lambda$ is a two-layer hook, we have
\[
\tilde{\lambda}_1-2=\tilde{\lambda}_1-\tilde{\lambda}_2=\lambda_1-\lambda_2=\hat{\lambda}_1-\hat{\lambda}_2=\hat{\lambda}_1-2=\widehat{\tilde{\lambda}}_1-2,
\]
that is, the sizes of the first row and the first column of $\tilde{\lambda}$ are equal which implies that $\tilde{\lambda}$ is
symmetric and, thus, according to Lemma~\ref{dimensions_of_near_hooks},
\[
\hl(\tilde{\lambda})<\frac{(n-2)(n-4)!}{2}.
\]
If $\tilde{\lambda}$ is a two layer hook, then the same bound holds by the induction hypothesis.

We now observe the following facts:
\begin{enumerate}
\item $\lambda_1+\hat{\lambda}_1-2> n/2$; thus
\[
1+\frac{1}{\lambda_1+\hat{\lambda}_1-2}\,<\,\frac{n+2}{n}.
\]
\item By Lemma~\ref{char_of_two_layer_hook} and the definition of a two-layer hook, $\lambda_1-\lambda_2=\hat{\lambda}_1-\hat{\lambda}_2$; hence $\lambda_1+\hat{\lambda}_2=\lambda_2+\hat{\lambda}_1$. On the other hand
\[
\lambda_1+\hat{\lambda}_2+\lambda_2+\hat{\lambda}_1-5=n-1,
\]
hence
\[
1+\frac{1}{\lambda_1+\hat{\lambda}_2-3}=\frac{n}{n-2},
\quad
1+\frac{1}{\lambda_2+\hat{\lambda}_1-3}=\frac{n}{n-2}.
\]
\item Since $\lambda_2+\hat{\lambda}_2\geq 5$, we have
\[
1+\frac{1}{\lambda_2+\hat{\lambda}_2-4}\,<\,2.
\]
\item Since $\lambda_1+\lambda_2\leq n-1$, we have
\[
(\lambda_1-1)(\lambda_2-2)\leq \frac{(n-4)^2}{4}.
\]
\end{enumerate} 

All of these facts together with (\ref{hl_lambda}) yield 
\begin{align*}
\hl(\lambda)\,&<\, \frac{n+2}{n}\, \frac{n}{n-2}\, \frac{n}{n-2}\,\, 2\, \frac{(n-4)^2}{4}\, \frac{(n-2)(n-4)!}{2}\\[.3cm]
&=\frac{n(n+2)(n-4)^2(n-4)!}{4(n-2)}\\[.3cm]
&=\frac{(n+2)(n-4)^2}{2(n-2)^2(n-3)}\,\,\frac{n(n-2)! }{2}\\[.3cm]
& <\,\frac{n(n-2)! }{2}. \qedhere
\end{align*}
\end{proof}

%% file: chap-evals_of_Cayley.tex
\chapter{Cayley Graphs}\label{evalues of Cayley}
This chapter is devoted to describing the eigenvalues of normal Cayley graphs using the irreducible  representations of the underlying groups. Some of the facts proved in this chapter will be used in the next chapters. First in Section~\ref{cayley} we introduce the Cayley graphs and point out some of their basic properties. Then in Section~\ref{evals_of_normal} we explain the proof of the well-known Theorem~\ref{Diaconis} which is due to Diaconis and Shahshahani \cite{MR626813}. This is a beautiful connection between the character theory of groups and the spectral graph theory and has attracted the attention of many researchers in the field of algebraic combinatorics. Using this machinery, then, in the other section of this chapter, we will try to establish relationships between the spectra of Cayley graphs on groups and those of the related quotient groups.

\section{Basic facts}\label{cayley}
Let $G$ be a group and $S$ be a subset of $G$ which is closed under inversion and does not contain the identity. The \txtsl{Cayley
graph} of $G$ with respect to $S$, denoted by $\Gamma(G;S)$, is the graph whose vertices are the elements of $G$ and
two vertices $g$ and $h$ are adjacent if $gh^{-1}\in S$. If $S$ is closed under conjugation, then $\Gamma(G;S)$ is said to be a {normal Cayley graph}\index{Cayley graph!normal}. The set $S$ is called the \txtsl{connection set}.
\begin{example}
The Cayley graph $\Gamma(\mathbb{Z}_n;\{\pm1\})$ is isomorphic to the cycle $C_n$, for $n\geq 3$. More generally, if $k\in \mathbb{Z}_n$, then $\Gamma(\mathbb{Z}_n;\{\pm k\})$ is isomorphic to $C_n$ if and only gcd$(n,k)=1$, for $n\geq 3$. The Cayley graph $\Gamma(\sym(3);\{(1\,\,2\,\,3),(1\,\,3\,\,2)\})$ is isomorphic to the  disjoint union of two 3-cycles. The general cases of this example and their structures will be studied in Chapter~\ref{single_CC}.
\end{example}

Clearly $\Gamma(G;S)$ is an $|S|$-regular graph. In fact, it is easy to see that Cayley graphs are vertex transitive; however, not every vertex-transitive graph can be considered as a Cayley graph.  It has been shown \cite{MR1397877} that the Petersen graph\index{Petersen graph} is the smallest vertex-transitive graph which is not a Cayley graph.

\begin{prop}\label{index}
Consider the Cayley graph $\Gamma(G;S)$. Let $H=\langle S \rangle$ be the subgroup of $G$ generated by $S$, and assume $i=[G:H]$, the index of $H$ in $G$. Then $\Gamma(H;S)$  is connected and
\[
\Gamma(G;S)\cong \underbrace{\Gamma(H;S)\,\,\bigcupdot\cdots\bigcupdot\,\, \Gamma(H;S)}_{i\,\,\text{times}}.
\]
\end{prop}
\begin{proof}
The first part is trivial. For the second part, it is enough to note that the right multiplication by a $g\in G\backslash H$ is a graph  isomorphism between $\Gamma(H;S)$ and the component of $\Gamma(G;S)$ which has the vertex $g$.
\end{proof}
Furthermore, the following fact follows from the definition of the Cayley graphs.
\begin{prop}\label{alpha_of_Cayley_subgraphs}
Let $T$ and $S$ be non-empty subsets of a group $G$ which are closed under inversion such that $\id\notin T\subseteq S$. Then 
\[
\alpha(\Gamma(G;T))\geq \alpha(\Gamma(G;S)).\qed
\]
\end{prop}

\section{Eigenvalues of Cayley graphs}\label{evals_of_normal}

In this section we prove the following important theorem which states how the eigenvalues of normal Cayley graphs are related to the irreducible representations of the underlying groups. This theorem is, indeed, the most fundamental theorem   of this thesis.

\begin{thm}\label{Diaconis} The eigenvalues of a normal Cayley graph $\Gamma(G;S)$ are given by
\[
\eta_{\chi}=\frac{1}{\chi(\id)}\sum_{s\in S}\chi(s),
\]
where $\chi$ ranges over all irreducible characters of $G$. Moreover, the multiplicity of $\eta_{\chi}$ is $\chi(\id)^2$.
\end{thm}

The proof is due to Diaconis and Shahshahani \cite{MR626813}. We provide a proof which is essentially a modification of the method
applied in \cite{MR626813} to our special case. First we introduce the following notation. 
Let $G$ be a group and $P:G\rightarrow\mathbb{C}$ be a function with complex values. For each representation $\xx$ of $G$, define 
\[
\xx(P)=\sum_{x\in G}P(x)\xx(x).
\]
\begin{lem}\label{sh_lemma}
Let $\xx$ be an irreducible representation of the group $G$ with character $\chi$, and $P:G\rightarrow \mathbb{C}$ be constant on conjugacy classes. For the $i$-th conjugacy class $C_i$, let $P_i$ and  $\chi_i$ be the values of $P$ and $\chi$, respectively, on $C_i$ and  $n_i$ be the size of $C_i$. Then
\[
\xx(P)=kI,
\]
where $I$ is the identity matrix and the constant $k$ is as follows
\[
k=\frac{1}{\chi(\id)}\sum_{i}P_in_i\chi_i.
\]
\end{lem}
\begin{proof}
Assume $V$ is the vector space corresponding to the representation $\xx$ and let $C_1,\ldots,C_r$ be the conjugacy classes of $G$ and, for each $i=1,\dots,r$, define the matrix $M_i$ as follows
\[
M_i=\sum_{x\in C_i}\xx(x).
\]
Then
\[
\xx(P)=\sum_{x\in G} P(x)\xx(x)=\sum_{i=1}^rP_iM_i.
\]
For any $x\in G$, we have $\xx(x)M_i\xx(x^{-1})=M_i$; because
\begin{align*}
\xx(x)M_i\xx(x^{-1})&=\sum_{y\in C_i}\xx(x)\xx(y)\xx(x^{-1})=\sum_{y\in C_i}\xx(xyx^{-1})\\
&=\sum_{y\in C_i}\xx(y)=M_i.
\end{align*}
This means that the operator induced by the matrices $M_i$  are $G$-homomorphisms from $V$ to itself. Therefore, by Schur's lemma (Theorem~\ref{schur}), $M_i=k_iI$, for some $k_i\in \mathbb{R}$, $i=1,\ldots,r$. Taking traces, then, we have both
\[
\Tr(M_i)=k_i \,\dim(\xx)\quad \text{and}\quad \Tr(M_i)=n_i\chi_i,
\]
which implies
\[
k_i=\frac{1}{\chi(\id)}n_i\chi_i.
\]
Thus
\[
M_i=\frac{1}{\chi(\id)}n_i\chi_i I.
\]
We conclude that
\[
\xx(P)=\left(\frac{1}{\chi(\id)}\sum_iP_in_i\chi_i \right)I,
\]
and the proof is complete.
\end{proof}
\begin{cor}\label{sh_cor}
If $S\subseteq G$ is closed under conjugation, then for any irreducible representation $\xx$ with character $\chi$, we have
\[
\sum_{s\in S}\xx(s)=\left(\frac{1}{\chi(\id)}\sum_{s\in S}\chi(s) \right) I_d,
\]
where $d=\chi(\id)$ is the dimension of $\xx$.
\end{cor}
\begin{proof}
Let $P=\delta_S$, the characteristic function of $S$. Since $S$ is stable under conjugation, $P$ is constant on conjugacy classes. Thus, using Lemma~\ref{sh_lemma}, we have
\begin{align*}
\xx(P)&=\sum_{s\in S}\xx(s)=\frac{1}{\chi(\id)}\sum_{i}P_in_i\chi_i I\\[.2cm]
&=\left(\frac{1}{\chi(\id)}\sum_{i:\, C_i\subseteq S}n_i\chi_i\right)\,I=\left(\frac{1}{\chi(\id)}\sum_{s\in S}\chi(s)\right)\,I,
\end{align*}
which completes the proof.
\end{proof}

Now we are ready to prove Theorem~\ref{Diaconis}.
\begin{proof} (Theorem~\ref{Diaconis}.)
Consider the group algebra $\mathbb{C}[G]$ with the basis $\{e_g\,|\, g\in G \}$ whose multiplication is defined as the $\mathbb{C}$-linear extension of the multiplication
\[
e_g\cdot e_h=e_{gh},\quad\quad\text{for each}\,\, g,h\in G.
\]
Define the linear transformation $T:\mathbb{C}[G]\rightarrow\mathbb{C}[G]$ by
\[
T(x)=\left(\sum_{s\in S}e_s\right)x.
\]
If we let $Q$ be the matrix associated to the transformation $T$ with respect to the basis $\{e_g\,|\, g\in G\}$, then $Q$ will be
the adjacency matrix of $\Gamma(G;S)$.

On the other hand, assume that $\xx:G\rightarrow GL(\mathbb{C}[G])$ is the left regular representation of $G$ (thus $\dim\xx=|G|$) and let $\chi$ be the character of $\xx$. Define the matrix  $\xx(A)$ to be
\[
\xx(A)=\sum_{s\in S}\xx(s).
\]
Then it is not hard to see that the action of $\xx(A)$ on $\mathbb{C}[G]$ is identical to the action of $Q$ on $\mathbb{C}[G]$.  Therefore, in order to find the eigenvalues of the Cayley graph $\Gamma(G;S)$, it suffices to find the eigenvalues of $\xx(A)$.

By Maschke's theorem (Theorem~\ref{maschke}), we have
\[
V=\mathbb{C}[G]=\bigoplus_{\rho}V_\rho,
\]
where $\rho$ ranges over all irreducible representations of $G$ and for each $\rho$,
\[
V_\rho=\underbrace{W_\rho\oplus\cdots\oplus W_\rho}_{\dim(\rho)\,\,\text{times}},
\]
where $W_\rho$ is the vector  space corresponding to the irreducible representation $\rho$. Let $d_\rho=\dim(\rho)=\chi_\rho(\id)$. We have 
\begin{align*}
\xx(A)&=\sum_{s\in S}\xx(s)\\[.2cm]
&=\sum_{s\in S}\,\,\bigoplus_{\rho\in\irr(G)}\left(\bigoplus_{i=1}^{d_\rho}\rho(s)\right)\\[.2cm]
&=\bigoplus_{\rho\in\irr(G)}\,\,\bigoplus_{i=1}^{d_\rho}\left(\sum_{s\in S} \rho(s)\right).
\end{align*}
Using Corollary~\ref{sh_cor}, therefore, we have
\[
\xx(A)=\bigoplus_{\rho\in \irr(G)}\bigoplus_{i=1}^{d_\rho}\left[\left(\frac{1}{\chi_\rho(\id)}\sum_{s\in S} \chi_\rho(s)\right)I_{d_\rho} \right].
\]
Thus
\[
\xx(A)=\begin{bmatrix}
 \frac{1}{\chi_{\rho_1}(\id)}\sum_{s} \chi_{\rho_1}(s)\, I_{d_{\rho_1}} & & & & & &&\\
  & \hspace{-3cm}\ddots & &&&&&\\
 & & \hspace{-3cm}\frac{1}{\chi_{\rho_1}(\id)}\sum_{s} \chi_{\rho_1}(s) \, I_{d_{\rho_1}} & & &&& \\
 &&\hspace{-1cm} \frac{1}{\chi_{\rho_2}(\id)}\sum_{s} \chi_{\rho_2}(s)\, I_{d_{\rho_2}} && &&&\\
 &&& \hspace{-3cm}\ddots & &&&\\
&&&& \hspace{-3cm}\frac{1}{\chi_{\rho_2}(\id)}\sum_{s} \chi_{\rho_2}(s) \, I_{d_{\rho_2}}&&&\\
&&&&\ddots&&&\\
&&&&&\hspace{-1cm}\frac{1}{\chi_{\rho_t}(\id)}\sum_{s} \chi_{\rho_t}(s)\, I_{d_{\rho_t}}  &&\\
  &&&&&& \hspace{-3cm}\ddots &\\
 & &&&&&& \hspace{-3cm}\frac{1}{\chi_{\rho_t}(\id)}\sum_{s} \chi_{\rho_t}(s)\, I_{d_{\rho_t}}
\end{bmatrix},
\]
where $t$ is the number of distinct irreducible representations of $G$. Therefore the proof is complete.
\end{proof}


\section{Cayley graphs of quotient groups}\label{quotient}
In this section, for a given group $G$, we investigate the connections between Cayley graphs of quotient groups $G/N$, where
$N$ is a normal subgroup of $G$, and Cayley graphs of $G$. For a subset $S$ of $G$, let $S/N=\{sN\,:\, s\in S\}$.

\begin{prop}\label{quotient is cayley}
If $\Gamma(G;S)$ is a (normal) Cayley graph and $N\lhd G$ such that $N\cap S=\emptyset$, then  $\Gamma(G/N; S/N)$ will be a  (normal) Cayley graph.
\end{prop}
\begin{proof}
First $N\cap S=\emptyset$ implies that $S/N$ does not include the identity, and since $S$ is closed under inversion, $S/N$ is also closed under inversion. Thus $\Gamma(G/N; S/N)$ is a Cayley graph. Second, if $\Gamma(G;S)$ is normal, then for every  $x\in G$ and $s\in S$, we have that $x^{-1}sx\in S$. Therefore $(xN)^{-1}(sN)(xN)\in S/N$ for all $xN\in G/N$  and $sN\in S/N$. This shows that $\Gamma(G/N;S/N)$ is a normal Cayley graph as well. 
\end{proof}

The following lemma is an easy consequence of Proposition~\ref{quotient is cayley}.
\begin{lem}
The canonical group epimorphism $\pi:G\to G/N$ induces a graph epimorphism
\[
\tilde{\pi}:\Gamma(G;S)\to\Gamma(G/N;S/N)
\]
for every $N\lhd G$ with $N\cap S=\emptyset$. Moreover, $\tilde{\pi}$ preserves the degrees of vertices if and only if $N\cap \{st^{-1}\,:\, s,t\in S, s\neq t\}=\emptyset$. \qed
\end{lem}

Now we investigate how the eigenvalues of $\Gamma(G;S)$ and $\Gamma(G/N;S/N)$ are related. Note that if $\xx:G\to GL(V)$ is a representation of $G$ of dimension $n$, then since $GL(V)\cong M_n(\mathbb{C})$, the representation $\xx$ extends $\mathbb{C}$-linearly to an algebra homomorphism $\xx:\mathbb{C}[G]\to M_n(\mathbb{C})$ via
\[
\xx(\sum_{i=1}^kc_ig_i)=\sum_{i=1}^kc_i\xx(g_i),
\]
where $g_i\in G$, $c_i\in \mathbb{C}$ and $k\in \mathbb{N}$.

\begin{lem}\label{irr}
Let $G$ be a group. Then a  representation $\xx:G\to M_n(\mathbb{C})$ is irreducible if and only if $\xx(\mathbb{C}[G])=M_n(\mathbb{C})$.
\end{lem}
\begin{proof}
If $\xx$ is not irreducible, then 
\[
\xx(g)=\left[ \begin{array}{cc}
A(g) & 0  \\
0 & B(g) \end{array} \right],
\] 
for every $g\in G$. Moreover, $A(g)$ is a square matrix of  size $n'\in \{1,\ldots, n-1\}$. This would, then, imply that for all $\alpha\in \mathbb{C}[G]$,
\[
\xx(\alpha)=\left[ \begin{array}{cc}
A'(\alpha) & 0  \\
0 & B'(\alpha) \end{array} \right],
\] 
where  $A'(\alpha)$ is a square matrix of size $n'$. But this means that $\xx(\mathbb{C}[G])\neq M_n(\mathbb{C})$.

For the converse, consider $\mathbb{C}[G]$ as the vector space for the  left-regular representation of $G$. Since $\mathbb{C}[G]$ is
semisimple, we will have the decomposition 
\[
\mathbb{C}[G]=\bigoplus_{\rho} \,\, \mathbb{C}[G]e_{\rho},
\]
where $\rho$ ranges over all irreducible representations of $G$ and $e_{\rho}$ are idempotents of $\mathbb{C}[G]$ corresponding to the irreducible representations $\rho$ (see \cite{MR2270898} and \cite{MR0450380} for more details). Note that all the irreducible
representations of $G$ appear in this decomposition. Now applying $\xx$ and assuming that $\xx$ is irreducible, we get 
\[
\xx(\mathbb{C}[G])\cong\mathbb{C}[G]e_{\xx}.
\]
On the other hand, since $\im(\xx)\subseteq M_n(\mathbb{C})$, we have
\[
\mathbb{C}[G]e_{\xx}\cong M_n(\mathbb{C}),
\]
and, therefore, $\xx(\mathbb{C}[G])=M_n(\mathbb{C})$.
\end{proof}

The following lemma states that irreducible representations of $G/N$ are, indeed, a subset of the irreducible representations of $G$.

\begin{lem}\label{induces}
If $N \lhd G$, then every irreducible representation of $G/N$ produces an irreducible representation for $G$.
\end{lem}
\begin{proof}
Let $\xx:G/N\to M_n(\mathbb{C})$ be an irreducible representation of $G/N$. Consider the canonical projection $\pi:G\to G/N$. The group homomorphism $\hat{\xx}=\xx \circ \pi$ is, therefore, a representation of $G$. In other words for any $g\in G$,
\[
\hat{\xx}(g)=\xx(gN).
\]
Since $\xx $ is irreducible, by Lemma~\ref{irr}, $\xx(\mathbb{C}[G/N])=M_n(\mathbb{C})$, that is, $\xx$ is onto. Thus $\hat{\xx}$ is onto, which again using the Lemma~\ref{irr} implies that $\hat{\xx}$ is irreducible.
\end{proof}
Now we can prove the following.
\begin{thm}\label{main1}
 Let $\Gamma(G;S)$ be a normal Cayley graph for which $S$ is contained in a conjugacy class of $G$ and $N \lhd G$ for which $N\cap S=\emptyset$. Assume $\lambda$ is an eigenvalue of the Cayley  graph $\Gamma(G/N;S/N)$. Then $c\lambda$ is an eigenvalue of $\Gamma(G;S)$, where $c=\frac{|S|}{|S/N|}$.
\end{thm}
\begin{proof}
Let $\chi$ be the character of $G/N$ corresponding to the irreducible representation $\xx$ of $G/N$. Using the notation of  Lemma~\ref{induces}, $\hat{\xx}$ is an irreducible representation of $G$ with the character $\hat{\chi}$. By Theorem~\ref{Diaconis}, then, $\eta_{\hat{\chi}}$ will be an eigenvalue of $\Gamma(G;S)$. Furthermore 
\[
\eta_{\hat{\chi}}=\frac{1}{\hat{\chi}(\id)}\sum_{s\in S}\hat{\chi}(s)=\frac{1}{\chi(N)}\sum_{s\in S}\chi(sN).
\]
But note that since $S$ is closed under conjugation and is contained in a single conjugacy class, $S$ is a conjugacy class of $G$. With  the same reasoning $S/N$ is a conjugacy class of $G/N$. Thus $\chi$ has a fixed value on $S/N$. Therefore
\[
\eta_{\hat{\chi}}=\frac{|S|}{\chi(N)}\chi(s_0N),
\]
where $s_0$ is an arbitrary element of $S$. On the other hand
\[
\eta_{\chi}=\frac{1}{\chi(N)}\sum_{sN\in S/N}\chi(sN)=\frac{|S/N|}{\chi(N)}\chi(s_0N);
\]
thus $\eta_{\hat{\chi}}=c\eta_{\chi}$.
\end{proof}

\begin{thm}\label{long thm}
Let $\Gamma(G;S)$ be a normal Cayley graph and $N \lhd G$ for which $N\cap S=\emptyset$ and $N\cap \{st^{-1}\,:\, s,t\in S, s\neq t\}=\emptyset$. Then the spectrum of $\Gamma(G/N;S/N)$  is a subset of the spectrum of $\Gamma(G;S)$.
\end{thm}
\begin{proof}
Using the notation of the proof of Theorem~\ref{main1}, we have
\[
\eta_{\hat{\chi}}=\frac{1}{\hat{\chi}(\id)}\sum_{s\in S}\hat{\chi}(s)=\frac{1}{\chi(N)}\sum_{s\in S}\chi(sN).
\]
Now, note that the condition  $N\cap \{st^{-1}\,:\, s,t\in S, s\neq t\}=\emptyset$ implies that $sN\neq tN$ for $s,t\in S$ and $s\neq t$. Therefore
\[
\eta_{\hat{\chi}}=\frac{1}{\chi(N)}\sum_{sN\in S/N}\chi(sN)=\eta_{\chi},
\]
which completes the proof.
\end{proof}

\begin{example}
Let $n>2$; Then the cycle $C_{2n}\cong \Gamma(\mathbb{Z}_{2n}; \{\pm1\})$. Now the (normal) subgroup $N=\{0,n\}$ satisfies the conditions of Corollary~\ref{long thm}. Thus the spectrum of $\Gamma(\mathbb{Z}_{2n}/N;\{\pm1\}/N)$ is involved in the spectrum of $\Gamma(\mathbb{Z}_{2n}; \{\pm1\})$. But
\[
\Gamma(\mathbb{Z}_{2n}/N;\{\pm1\}/N)\cong\Gamma(\mathbb{Z}_n;\{\pm1\})\cong C_n.
\]
Thus we have shown that the spectrum of $C_n$ is a subset of the spectrum of $C_{2n}$, for~$n>2$.
\end{example}

%% file: chap-EKR_perm_groups.tex
\chapter{EKR for Permutation Groups}\label{EKR_perm_groups}

Recall from Chapter~\ref{introduction} that the  Erd\H{o}s-Ko-Rado theorem gives bounds for the sizes of intersecting set systems and characterizes the systems that achieve the bound. Recall, also, that many similar theorems have been proved for other mathematical ``objects'' with a relevant concept of ``intersection''. In this chapter, in which our most fundamental research work  starts, we assume these objects to be the permutations and define the intersection as follows: two permutations intersect if a point has the same image under both permutations. In this chapter we establish versions of the Erd\H{o}s-Ko-Rado theorem  for this situation. Note that in the first part of the original Erd\H{o}s-Ko-Rado theorem for $k$-subsets of an $n$-set, the bound $\binom{n-1}{k-1}$ is, indeed, the size of a ``trivially intersecting'' family of $k$-subsets; that is, the family off all $k$-sets that contain a common single point. Its equivalent in the permutations category, is the ``cosets of the point-stabilizers''. Note also, that the second part of the Erd\H{o}s-Ko-Rado theorem states that the only maximum-sized families are the trivially intersecting ones; therefore, the equivalent  statement in   our situation is that the only sets of permutations of the maximum size are cosets of the point-stabilizers. If a permutation group satisfies  in the first condition we will say it has the ``EKR property''. If it also satisfies the second condition, then we say it has the ``strict EKR property''.

In Section~\ref{EKR_property} we present precise definitions for the EKR and strict EKR properties. In Section~\ref{EKR_induced} we discuss how these properties can be induced to a group by its subgroups. In Section~\ref{EKR_some_families}, we try to prove the EKR  or the strict EKR property holds for some well-known  families of groups. The chapter is concluded by Section~\ref{EKR_products} where  the EKR problem is studied for certain  group products. The reader is expected to be familiar with the basic concepts from the permutation groups; see for instance \cite{cameron1999permutation}.


\section{Erd\H{o}s-Ko-Rado property}\label{EKR_property}

Let $G\leq \sym(n)$ be a permutation group with the natural action on the set $[n]$. 
Throughout this thesis, we denote the set of all derangement elements of $G$ by $\dd_G$.
Two permutations $\pi,\sigma\in G$ are said to \textsl{intersect}\index{intersecting permutations} if $\pi\sigma^{-1}$ has a fixed point in $[n]$. In other words, $\sigma$ and $\pi$ do not intersect if $\pi\sigma^{-1}\in \dd_G$. A subset $S\subseteq G$ is, then, called \textsl{intersecting}\index{intersecting set of permutations} if any pair of its elements intersect. Clearly, the stabilizer of a point is an intersecting set in $G$ (as is any coset of the stabilizer of a point).  

We say the group $G$ has the \txtsl{EKR property}, if the size of any intersecting subset of $G$  is bounded above by the size of the largest point-stabilizer in $G$. Further, $G$ is said to have the \txtsl{strict EKR property} if the only maximum intersecting subsets of $G$ are the cosets of the point-stabilizers. It is clear from the definition that if a group has the strict EKR property  then it will have the EKR property.

Obviously the first group to consider is the symmetric group. In 1977 Frankl and Deza \cite{Frankl1977352} proved $\sym(n)$ has the EKR property and conjectured that it had the strict EKR property. In 2003, Cameron and Ku \cite{MR2009400} proved this conjecture. That is, they proved that
\begin{thm}\label{EKR_for_sym}
For any $n\geq 2$, $\sym(n)$ has the strict EKR property. \qed
\end{thm}
This result caught the attention of several researchers, indeed, the result was proved with vastly different methods in \cite{Karen,MR2061391} and \cite{MR2419214}. In Chapter~\ref{module_method} we will explain the method used in \cite{Karen}. Further, researchers have also worked on finding other subgroups of $\sym(n)$ that have the strict EKR property. For example in \cite{KuW07} it is shown that $\alt(n)$ has the strict EKR property, provided that $n\geq 5$ and that all the ``Young subgroups'', except a few of them, have the strict EKR property.

It is a  natural question to ask if every permutation group has the EKR property. The answer is no; for instance, it is shown in Section~\ref{sporadic}, that the Mathieu group $M_{20}$ does not have the EKR property. In Section~\ref{sporadic} we also give an example of groups with EKR property which fails to have the strict EKR property.  It is therefore an important problem to classify the groups which have the EKR or the strict EKR property. We address this general question in Chapter~\ref{future}.

We point out that the group action is essential for the concepts of EKR and strict EKR properties. In other words, a group can have the EKR or the strict EKR property under some action on a set while it fails to have this property under another action. As an example, there is a $2$-transitive subgroup of $\sym(5)$ which is isomorphic to $G=\zz_5\rtimes \zz_4$ and does not have the strict EKR property (see Table~\ref{small_perm_groups}). On the other hand, $\sym(5)$ is the stabilizer of $6$ in $\sym(6)$; that is,  $\sym(5)=(\sym(6))_6$. Thus $G$ can be considered as a subgroup of  $\sym(6)$. Then, since all the elements of $G$ fix $6$, under the natural action of $\sym(6)$ on $\{1,\dots,6\}$, the whole set $G$ is intersecting; that is, the only maximum intersecting set in $G$ under this action is $G$, which is the stabilizer of $6$. This means that $G$ trivially  has the strict EKR property under the natural action of $\sym(6)$ on $\{1,\dots,6\}$. This is the reason that in this thesis, we always consider the ``permutation groups'' (i.e. the subgroups of $\sym(n)$ with their natural action on $[n]$) rather than ``groups''.


The Cayley graph $\Gamma(G,\dd_G)$ is called the \txtsl{derangement graph} of $G$ and is denoted by $\Gamma_G$.   
Note that two permutations in $G$ are  intersecting if and only if their corresponding vertices are not adjacent in $\Gamma_G$. Therefore, the problem of classifying the maximum intersecting subsets of $G$ is equivalent to characterizing the maximum independent sets of vertices in $\Gamma_G$.  According to Section~\ref{cayley}, $\Gamma_G$ is vertex transitive. 
Let $CC(G)$ denote the set of all derangement conjugacy classes of $G$. The following is a consequence of Theorem~\ref{Diaconis}.
\begin{cor}\label{evals_of_der_graph} The eigenvalues of the derangement graph $\Gamma_{\sym(n)}$ are given by
\[
\eta_{\lambda}=\sum_{c\in CC(G)}\frac{|c|}{\chi^{\lambda}(\id)}\chi^{\lambda}(\sigma), 
\]
where  $\lambda$ ranges over all the partition of $n$, and $\sigma\in c$. Moreover, the multiplicity of $\eta_\lambda$ is $\chi_{\lambda}(\id)^2$.\qed
\end{cor}

\section{EKR property induced by subgroups}\label{EKR_induced}
In this section we prove that if a $2$-transitive group $G\leq \sym(n)$ has the strict EKR property, then any permutation group of degree $n$ containing $G$ also has the strict EKR property. This shows that the ``minimal'' $2$-transitive permutation groups are  ``core'' objects for studying the strict EKR property.

We first show how  transitive groups can inherit EKR the  property from their subgroups. These two facts were first pointed out by Pablo Spiga. The proof of the first result that we provide here is due to Chris Godsil.

\begin{thm}\label{H_EKR_then_G_EKR}
Let $G$ be a transitive subgroup of $\sym(n)$ and let $H$ be a transitive subgroup of $G$. If $H$ has the EKR property, then $G$ has the EKR property. 
\end{thm}
\begin{proof} The group $H$ has the EKR property and is transitive, so the size of the maximum independent set is $|H|/n$. Further, according to Section~\ref{cayley} the graph $\Gamma_H$ is vertex transitive so its fractional chromatic number is $n$.
The embedding $\Gamma_H \hookrightarrow \Gamma_G$ is a homomorphism, so according to Proposition~\ref{fractional_chrom_ineq}, the fractional chromatic number of $\Gamma_G$ is at least the fractional chromatic number of $\Gamma_H$. The graph $\Gamma_G$ is also vertex transitive, so 
\[
n \leq \frac{|G|}{\alpha(\Gamma_G)}
\]
where $\alpha(\Gamma_G)$ is the size of a maximum independent set. Thus $\alpha(\Gamma_G) \leq \frac{|G|}{n}$, and since $G$ is
transitive, the stabilizer of a point achieves this bound.
\end{proof}


\begin{thm}\label{2transitive_subgroup} Let $G$ be a $2$-transitive subgroup of $\sym(n)$ and let $H$ be a $2$-transitive subgroup of $G$. If $H$ has the strict EKR property, then $G$ has the strict EKR property.
\end{thm}
\begin{proof}
Since $H$ has the strict EKR property, it also has the EKR property and by Theorem~\ref{H_EKR_then_G_EKR}, $G$ also has the EKR property. Assume  that $S$ is an independent set in $\Gamma_G$ of size $|G|/n$ that contains the identity; we will prove that $S$ is the stabilizer of a point. Let $\{x_1=\id,\dots,x_{[G:H]}\}$ be a left transversal of $H$ in $G$ and set $S_i = S \cap x_iH$.  Then for each $i$ the set $x_i^{-1}S_i$ is an independent set in $\Gamma_H$ with size $|H|/n$. Since $H$ has the strict EKR property each $x_i^{-1}S_i$ is the coset of a stabilizer of a point. Since $x_1 =\id$, the identity is in $S_1$ which means that $S_1$ is the stabilizer of a point and we can assume that $S_1 = H_\alpha$ for some $\alpha \in [n]$. We need to show that every   permutation in $S$ also fixes the point $\alpha$. Assume that there is a $\pi \in S$ that does not fix $\alpha$.  Since $S$ is   intersecting, for every $\sigma \in S_1$ the permutation $\sigma\pi^{-1}$ fixes some element (but not $\alpha$ and not $\pi(\alpha)$), from this it follows that
\[
H_\alpha \pi^{-1}  = \bigcup_{\stackrel{\beta \neq \alpha}{\beta \neq \pi(\alpha)}} ( G_\beta \cap H_\alpha \pi^{-1} ).
\]
Assume that $\sigma \pi^{-1} \in G_\beta \cap H_\alpha \pi^{-1}$, then $\beta^{\sigma \pi^{-1}} = \beta$ and $\alpha^{\sigma} = \alpha$.  The permutation $\sigma \pi^{-1}$ must map $(\alpha, \beta)$ to $(\alpha^{\pi^{-1}}, \beta)$.  Since the group $H$ is $2$-transitive there are exactly $|H|/n(n-1)$ such permutations and we have that
\[
| G_\beta \cap H_\alpha \pi^{-1}| = \frac{|H|}{n(n-1)}.
\]
From this we have that the size of $H_\alpha \pi^{-1}$ is
\[
\sum_{\stackrel{\beta \neq \alpha}{\beta \neq \pi(\alpha)}}\frac{|H|}{n(n-1)} = (n-2) \frac{|H|}{n(n-1)},
\]
but since this is strictly less that $\frac{|H|}{n}$,  which  is a contradiction.
\end{proof}


\section{EKR for some families of groups}\label{EKR_some_families}

In this section we show that the EKR and the strict EKR property holds for some important families of groups. The first  groups we consider are  cyclic groups.
\begin{thm}\label{cyclic_groups}
For any permutation $\sigma\in \sym(n)$, the cyclic group $G$ generated by $\sigma$ has the strict EKR property.
\end{thm}
\begin{proof}
Let $\sigma=\sigma_1\sigma_2\cdots\sigma_k$, where $\sigma_i$ are disjoint cycles. Assume that $\sigma_i$ has order $r_i$, and that $1\leq r_1\leq\cdots\leq r_k$. Note that the subgroup $H\leq G$ generated by $\sigma_1$ is of order $r_1$ and  the graph $\Gamma_H$ is isomorphic to the complete graph $K_{r_1}$.
On the other hand, it is not hard to see that the set $C=\{\sigma, \sigma^2,\ldots,\sigma^{r_1}\}$ induces a clique of size $r_1$ in the graph $\Gamma_G$. Therefore, there is a graph embedding $\Gamma_H \hookrightarrow \Gamma_G$. Since $\Gamma_H$ and $\Gamma_G$ are vertex transitive, according to Proposition~\ref{fractional_chrom_ineq}, we have
\[
\frac{|V(\Gamma_H)|}{\alpha(\Gamma_H)}=\chi^*(H)\leq \chi^*(G)= \frac{|V(\Gamma_G)|}{\alpha(\Gamma_G)};
\]
thus
\[
\frac{r_1}{1}\leq \frac{|G|}{\alpha(\Gamma_G)};
\]
that is,
\[
\alpha(\Gamma_G)\leq \frac{|G|}{r_1}.
\]
Note, in addition, that if $\sigma_1=(a_1,\ldots,a_{r})$, then the stabilizer of $a_1$ in $G$ is
\[
G_{a_1}=\{\sigma^{r_1}, \sigma^{2r_1},\ldots, \sigma^{|G|} \};
\]
therefore
\[
\alpha(\Gamma_G)=|G_{a_1}|=\frac{|G|}{r_1}.
\]
It is clear that this is the largest size of a point-stabilizer in $G$. This proves that $G$ has the EKR property.

To show the second part, first note that the clique $C$ and the independent set $G_{a_1}$ together show that the clique-coclique bound (Theorem~\ref{clique_coclique_bound}) for $\Gamma_G$ holds with equality. Hence any maximum independent set must intersect with any maximum clique in $\Gamma_G$. Let $S$ be any maximum independent set and without loss of generality assume $\id\in S$. We  show that $S=G_{a_1}$. For any $i\geq 0$, set 
\[
C_i=\{\sigma^{i},\sigma^{i+1},\dots,\sigma^{i+r_1-1}\}.
\]
Note that $C_i$ are cliques in $\Gamma_G$ of maximum size (i.e. of size $r_1$) and that $C_1=C$. Furthermore,  for any $0\leq t\leq |G|/r_1$, we have $C_{tr_1}\backslash C_{tr_1+1}=\{\sigma^{tr_1}\}$ and $C_{tr_1+1}\backslash C_{tr_1}=\{\sigma^{(t+1)r_1}\}$. Since for any $0\leq t\leq |G|/r_1$, the independent set $S$ intersects with each of the cliques  $C_{tr_1+1} $ and $ C_{tr_1}$ in exactly one point, and since $\sigma^0\in S$, we conclude that $\sigma^0,\sigma^{r_1}, \sigma^{2r_1},\ldots, \sigma^{|G|}\in S$; that is, $S=G_{a_1}$.
\end{proof}
Note that,  if $\sigma$ has a fixed point in $[n]$, then $\Gamma_{G}$ is the empty graph on $|G|$ vertices. Also in the case where $\sigma$ is an $n$-cycle, then $\Gamma_{G}$ is the complete graph on $n$ vertices. In both of these cases, the strict EKR holds trivially.

\begin{prop}
Any permutation group of degree $n$ with an $n$-cycle  has the EKR property.
\end{prop}
\begin{proof}
Let $G$ be such group and let $\sigma\in G$ be an $n$-cycle. Then the  subgroup $H$ generated by $\sigma$ is transitive and, according to Theorem~\ref{cyclic_groups}, it has the EKR property. Therefore, according to Theorem~\ref{H_EKR_then_G_EKR}, $G$ also has the EKR property.
\end{proof}

Recall that for any $n\geq 3$, the \txtsl{dihedral group} of degree $n$, denoted by $D_n$, is the group of symmetries of a regular $n$-gon, including both rotations and reflections. Note that $D_n\leq \sym(n)$ is a permutation group acting on $[n]$. 
\begin{prop}\label{dihedral}
All the dihedral groups have the strict EKR property.
\end{prop}
\begin{proof}
Assume $D_n$ is generated by the permutations $\sigma$, the rotation, and $\pi$, the reflection through the antipodal points of the $n$-gon. Then $\sigma$ is of order $n$, $\pi$ is of order 2 and $\pi\sigma=\sigma^{-1}\pi$ (see \cite[Theorem I.6.13]{MR600654}). Since $\{\id,\pi\}$ is an intersecting set, we have $\alpha(\Gamma_{D_n})\geq 2$. To prove the proposition, we show that any maximum independent set in $\Gamma_{D_n}$ is a coset of a point-stabilizer. Assume $S$ is a maximum independent set in $\Gamma_{D_n}$ and, without loss of generality, assume $\id\in S$. Clearly, $\sigma^i\notin S$, for any $1\leq i<n$. If $\sigma^i\pi, \sigma^j\pi\in S$, for some $1\leq j<i<n$, then their division,
\[
\sigma^i\pi(\sigma^j\pi)^{-1}=\sigma^{i-j}
\]
must have a fixed-point, which is a contradiction. Therefore,  $\alpha(\Gamma_{D_n})=2$ and $S=\{e,\sigma^i\pi\}$, for some $1\leq i<n$. Note, finally, that since no pair $\sigma^i\pi, \sigma^j\pi$ have any common fixed-point, $S$ is indeed the stabilizer of any of the points fixed by $\sigma^i\pi$. This completes the proof.
\end{proof}
Note that the dihedral group $D_n$ can be written as $D_n=\mathbb{Z}_n \mathbb{Z}_2$, where $\mathbb{Z}_n\triangleleft D_n$ corresponds to the subgroup generated by the rotations and $\mathbb{Z}_2$ is the subgroup generated by a reflection. More precisely, $D_n=\mathbb{Z}_n \rtimes\mathbb{Z}_2$ when the non-identity element of $\mathbb{Z}_2$ acts on $\mathbb{Z}_n$ by inversion. When $n$ is odd, this is a particular case of a Frobenius group. A transitive permutation group $G\leq \sym(n)$ is called a \txtsl{Frobenius group} if no non-trivial element fixes more than one point and some non-trivial element fixes a point.
An alternative definition  is as follows: a group $G\leq \sym(n)$ is a Frobenius group if it has  a non-trivial proper subgroup $H$ with the condition that $H\cap H^g=\{\id\}$, for all $g\in G\backslash H$, where $H^g=g^{-1}Hg$. This subgroup is called a \txtsl{Frobenius complement}. 
Define the \txtsl{Frobenius kernel} $K$ of $G$ to be 
\[
K=\left(G\backslash \bigcup_{g\in G} H^g\right)\cup\{\id\}.
\]
In fact, the non-identity elements of $K$ are  all  the derangement elements of $G$. There is a significant result due to Frobenius which states that $K$ is a normal subgroup of $G$. The proof mainly relies on  character theory and is one of the earliest major applications of this theory. Moreover, he showed that $G=K\rtimes H$. The reader may refer to \cite[Theorem 5.9]{ledermann1987introduction} for a proof.  Frobenius groups and their properties have been widely studied; we refer the interested readers to  \cite[Section 10.2]{berkovich1999characters}, \cite[Section 3.4]{dixon2012permutation} and \cite{Flavell2000367} for further details. In the rest of this section we study the EKR problem for the Frobenius groups.

First note that if  $G=KH\leq \sym(n)$ is a Frobenius group with  kernel $K$, then $|K|=n$ and $|H|$ must divide $n-1$; see \cite[Section 3.4]{dixon2012permutation} for proofs. This implies that the Frobenius groups are relatively small transitive subgroups of $\sym(n)$. In particular, if $n-1$ is prime, then $|G|=n(n-1)$. We also observe the following.
\begin{lem}\label{size_of_G_x_frobenius}
If $G=KH\leq \sym(n)$ is a Frobenius group with  kernel $K$, then $|G_x|=|H|$, for any $x\in[n]$.
\end{lem}
\begin{proof}
Let $x\in[n]$. By the ``orbit-stabilizer'' theorem we have
\[
|x^G|=[G:G_x]=\frac{|G|}{|G_x|}=\frac{|K||H|}{|G_x|}=\frac{n|H|}{|G_x|},
\]
where $x^G$ is the orbit of $x$ under the action of $G$ on $[n]$. Since this action is transitive,  $x^G=[n]$; therefore the lemma follows.
\end{proof}
In order to find the maximum intersecting subsets of a Frobenius group, we first describe their derangement graphs. We will make use of the following classical result (see \cite[Theorem 18.7]{huppert1998character}). 

\begin{thm}\label{irrs_of_frobenius}
Let $G=KH$ be a Frobenius group with the kernel $K$. Then the irreducible representations of $G$ are the following two types:
\begin{enumerate}[(a)]
\item Any irreducible representation $\Psi$ of $H$ gives an irreducible representation of $G$ using the quotient map $H\cong G/K$. These give the irreducible representations of $G$ with $K$ in their kernel. 
\item If $\Xi$ is any non-trivial irreducible representation of $K$, then the corresponding induced representation of $G$ is also irreducible. These give the irreducible representations of $G$ with $K$ not in their kernel. \qed
\end{enumerate}
\end{thm}
Now we can describe the derangement graphs of the Frobenius groups.
\begin{thm}\label{gamma_of_frobenius} Let $G=KH\leq \sym(n)$ be a Frobenius group with the kernel $K$. Then $\Gamma_G$ is the disjoint union of $|H|$ copies of the complete graph on $n$ vertices.
\end{thm}
\begin{proof}
According to Theorem~\ref{Diaconis}, the eigenvalues  of $\Gamma_G$ are given by 
\[
\eta_{\chi}=\frac{1}{\chi(\id)}\sum_{\sigma\in \dd_G} \chi(\sigma),
\]
where $\chi$ runs through the set of all irreducible characters of $G$. First assume $\chi$ is the character of an irreducible representation  of $G$ of type (a) in Theorem~\ref{irrs_of_frobenius}. Then we have
\[
\eta_{\chi}=\frac{1}{\chi(\id)}\sum_{\sigma\in \dd_G} \chi(\sigma)=\frac{1}{\chi(\id)}\sum_{\sigma\in \dd_G} \chi(\id)=|\dd_G|=|K|-1=n-1.
\]
According to Theorems \ref{Diaconis} and \ref{sum_of_dim_of_irrs}, the multiplicity of $\eta_{\chi}=n-1$ is 
\[
\sum_{\Psi\in\irr(H)}(\dim \Psi)^2=|H|.
\]
Furthermore, assume $\Xi$ is an irreducible representation  of $G$ of type (b) in Theorem~\ref{irrs_of_frobenius}, whose character is $\xi$ and let $\chi$ be the character of the corresponding induced representation of $G$.  If $\sigma\in G\backslash K$, then $\sigma\in H^g$, for some $g\in G$. Thus $x^{-1}\sigma x\in H^{gx}$, for any $x\in G$; hence $x^{-1}\sigma x \notin K$. According to (\ref{char_of_induced}), the formula for the character of an induced representation, this implies that $\chi(\sigma)=0$. On the other hand, let $\chi_{\ID}$ be the trivial character of $G$. Since $\chi\neq \chi_{\ID}$, according to Theorem~\ref{orth_chars}, the inner product of $\chi$ and $\chi_{\ID}$ is zero. Hence
\[
0=\left\langle \chi , \chi_{\ID}\right\rangle = \frac{1}{|G|} \sum_{\sigma\in G} \chi(\sigma) \chi_{\ID}(\sigma^{-1})=\frac{1}{|G|} \sum_{\sigma\in G} \chi(\sigma)=\frac{1}{|G|}\left(\chi(\id)+\sum_{\sigma\in \dd_G} \chi(\sigma)\right);
\]
hence 
\[
\sum_{\sigma\in \dd_G} \chi(\sigma)=-\chi(\id).
\] 
This yields
\[
\eta_{\chi}=\frac{1}{\chi(\id)}\sum_{\sigma\in \dd_G} \chi(\sigma)=\frac{-\chi(\id)}{\chi(\id)}=-1.
\]
 We have, therefore, shown that 
\[
\Spec(\Gamma_G)=\left( {\begin{array}{cc}
 n-1 & -1 \\
 |H| & |H|(n-1)\\
 \end{array} } \right).
\]
Now the theorem follows from Proposition~\ref{graphs_with_2_evalues}.
\end{proof}
We point out that this proof may not be the easiest method to describe the graph $\Gamma_G$; however it is a nice example of an application of character theoretical facts in graph theory. We will see other similar applications in Chapters~\ref{module_method} and \ref{single_CC}.  Now we establish the EKR property for the Frobenius groups.
\begin{thm}\label{EKR_for_frobenius} Let $G=KH\leq\sym(n)$  be a Frobenius group with kernel $K$. Then $G$ has the EKR property. Furthermore, $G$ has the strict EKR property if and only if $|H|=2$.
\end{thm}
\begin{proof}
Using Theorem~\ref{gamma_of_frobenius}, the independence number of $\Gamma_G$ is $|H|$. This along with Lemma~\ref{size_of_G_x_frobenius} shows that $G$ has the EKR property. For the second part of the theorem, first note that if $|H|=2$ and $S$ is an intersecting subset of $G$ of size two, then $S$ is, trivially, a point stabilizer.
To show the converse, we note that the cliques of $\Gamma_G$ are induced by the $|H|$ cosets of $K$ in $G$. Now suppose $|H|>2$ and let $S$ be a maximum intersecting subset of $G$ which is a coset of a point-stabilizer in $G$. Without loss of generality we may assume $S=\{\id=s_1, s_2,\dots, s_{|H|}\}$ and, hence, $S=G_x$, for some $x\in [n]$. Since $S$ is independent in $\Gamma_G$, no two elements of $S$ are in the same coset of $K$ in $G$. Note that, the only fixed point of any non-identity element of $S$ is $x$. Let $s_3$ be in the coset $gK$. If all the elements of $gK$ fix $x$, then all the elements of $K$ will have fixed points, which is a contradiction. Hence there is an $s'_3\in gK$ which does not fix $x$. Now the maximum intersecting set $S'=\left(S\backslash \{s_3\}\right)\cup\{s'_3\}$ is not a point-stabilizer.
\end{proof}
Note that one can show the second part of Theorem~\ref{EKR_for_frobenius} by a counting argument as follows. There are $n^2$ cosets of point-stabilizers in $G$. Since $\Gamma_G$ is the union of $|H|$ copies of the complete graph on $n$ vertices, the total number of maximum independent sets is $n^{|H|}$. Therefore, in order for all the maximum independent sets of $\Gamma_G$ to be cosets of point-stabilizers, the necessary and sufficient condition is $|H|=2$.

We conclude the section with noting that Theorem~\ref{EKR_for_frobenius} provides an alternative proof for the fact that the dihedral group $D_n$ has the strict EKR property, when $n\geq 3$ and odd.


\section{EKR for some group products}\label{EKR_products}

Now we turn our attention to the products of groups and investigate how a product of some groups can have the EKR or the strict EKR property when the initial groups do so. We consider three types of group products, namely the so-called external and internal direct products and the wreath product. 

Given any sequence of permutation groups $G_1\leq \sym(n_1),\,\dots\,, G_k\leq \sym(n_k)$, their \txtsl{external direct product} is defined to be the group $G_1\times \cdots\times G_k$, whose elements are $(g_1,\ldots,g_k)$, where $g_i\in G_i$, for $1\leq i\leq k$, and the binary operation is, simply, the ``component-wise'' multiplication. This group has  a natural action on the set $\Omega=[n_1]\times\cdots\times [n_k]$ induced by the natural actions of $G_i$ on $[n_i]$; that is, for any tuple $(x_1,\dots, x_k)\in \Omega$ and any element $(g_1,\dots,g_k)\in G_1\times\cdots\times G_k$, we have
\[
(x_1,\dots, x_k)^{(g_1,\dots,g_k)}:=(x_1^{g_1},\dots, x_k^{g_k}).
\]
Let $G=G_1\times\cdots\times G_k$. Then the derangement graph $\Gamma_G$ of $G$  is the graph with vertex set $G$ in which two vertices $(g_1, \dots,g_k)$ and $(h_1,\dots,h_k)$ are adjacent if and only if $g_ih_i^{-1}$ is a derangement, for some  $1\leq i\leq k$. Recall that if $X$ and $Y$ are graphs, then $X\times Y$ is their direct product (see Section~\ref{graph_basics}) and that $\overline{X}$ is the complement graph of $X$. We observe the following.
\begin{lem} Let the group $G=G_1\times\cdots\times G_k$ be the external direct product of the groups $G_1,\dots, G_k$. Then
\[
\Gamma_G=\overline{\overline{\Gamma_{G_1}}\times\cdots\times\overline{\Gamma_{G_k}}}.
\]
\end{lem}
\begin{proof}
By the definition of the external direct product, the vertices $(g_1, \dots,g_k)$ and $(h_1,\dots,h_k)$ of $\overline{\Gamma_G}$  are adjacent if and only if $g_ih_i^{-1}$ has a fixed point, for any $1\leq i\leq k$. This is equivalent to the case where $g_i$ is adjacent to $h_i$ in $\overline{\Gamma_{G_i}}$, for any $1\leq i\leq k$. This occurs if and only if $(g_1, \dots,g_k)$ and $(h_1,\dots,h_k)$ are adjacent in $\overline{\Gamma_{G_1}}\times\cdots\times\overline{\Gamma_{G_k}}$. This completes the proof.
\end{proof}
In the next lemma, we evaluate the independence number of $\Gamma_G$.

\begin{lem}\label{independence_of_external_prod} With the notation above, we have
\[
\alpha(\Gamma_G)=\alpha(\Gamma_{G_1})\cdot\cdots\cdot \alpha(\Gamma_{G_k}).
\]
\end{lem}
\begin{proof}
Let $S_i$  be a maximum independent set in $\Gamma_{G_i}$, for any $1\leq i\leq k$. Then the set $S=S_1\times\cdots\times S_k$ is an independent set in $\Gamma_G$, proving that $\alpha(\Gamma_G)\geq \alpha(\Gamma_{G_1})\cdot\cdots\cdot \alpha(\Gamma_{G_k})$. On the other hand, let $p_i:G\rightarrow G_i$ be the projection of $G$ onto the component $G_i$, for  $1\leq i \leq k$ and let $S$ be a maximum independent set in $\Gamma_G$. Then for any $1\leq i\leq k$, the set $p_i(S)$ is an independent set in $G_i$; hence $|p_i(S)|\leq \alpha(\Gamma_{G_i})$. Since $S\subseteq p_1(S)\times\cdots\times p_k(S)$ the lemma follows.
\end{proof}

\begin{thm}\label{EKR_for_external_prod} With the notation above, if all the $G_i$ have the (strict) EKR property, then $G$ has the (strict) EKR property.
\end{thm}
\begin{proof}
First note that the stabilizer of any point $(x_1,\dots,x_k)\in \Omega$ in $G$ is 
\[
(G_1)_{x_1}\times\cdots\times (G_k)_{x_k}.
\]
 On the other hand, if all the groups $G_i$  have the EKR property, according to Lemma~\ref{independence_of_external_prod}, the maximum size of an independent set in $\Gamma_G$ will be equal to
\[
|(G_1)_{x_1}|\cdot\cdots\cdot |(G_k)_{x_k}|,
\]
for some $(x_1,\dots,x_k)\in \Omega$; this proves that $G$ has the EKR property. Furthermore, assume all the $G_i$ have the strict EKR property and let $S$ be a maximum independent set in $\Gamma_{G}$. This implies that $p_i(S)$ is a maximum independent set in $\Gamma_{G_i}$, for each $1\leq i\leq k$; hence $p_i(S)=(G_i)_{x_i}$, for some $x_i\in [n_i]$. Therefore, $S=G_{(x_1,\dots, x_k)}$.
\end{proof}

The next product is the internal direct product. Assume $\Omega_1,\dots,\Omega_k$ are pair-wise disjoint non-empty subsets of $[n]$,  and consider the sequence $G_1\leq \sym(\Omega_1),\,\dots\,, G_k\leq \sym(\Omega_k)$. Then their \txtsl{internal direct product} is defined to be the group $G_1\cdot G_2\cdot \cdots\cdot G_k$, whose elements are $g_1g_2\cdots  g_k$, where $g_i\in G_i$, for $1\leq i\leq k$ and the binary operation is defined as follows: for the elements $g_1g_2\cdot \cdots \cdot g_k$ and $h_1h_2\cdot \cdots \cdot h_k$ in $G_1\leq \sym(\Omega_1),\,\dots\, G_k\leq \sym(\Omega_k)$, 
\begin{equation}\label{binary_operation_internal_prod}
g_1g_2\cdots g_k\,\,\cdot \,\, h_1h_2\cdots h_k := (g_1h_1)(g_2h_2) \cdots (g_kh_k).
\end{equation}
Note that since the $\Omega_i$ don't intersect, any permutation in $G_i$ commutes with any permutation in $G_j$, for any $1 \leq i\neq j\leq k$; hence the multiplication (\ref{binary_operation_internal_prod}) is well-defined. 
This group also has  a natural action on the set $\Omega=\Omega_1\cup\cdots\cup \Omega_k$ induced by the natural actions of $G_i$ on $\Omega_i$; that is, for any $x\in \Omega$ and any element $g_1 g_2\cdots g_k\in G_1\cdot G_2\cdot\cdots\cdot G_k$, we have
\[
x^{g_1 g_2\cdots g_k}:= x^{g_i},\quad \text{where } x\in \Omega_i.
\]
Let $G=G_1\cdot G_2\cdot\cdots\cdot G_k$. Then the derangement graph of $G$  is the graph $\Gamma_G$ with vertex set $G$ in which two vertices $g_1g_2\cdot \cdots \cdot g_k$ and $h_1h_2\cdot \cdots \cdot h_k$ are adjacent if and only if $g_ih_i^{-1}$ is a derangement, for all  $1\leq i\leq k$.  In other words, $\Gamma_G$ is the direct product of $\Gamma_{G_1},\dots,\Gamma_{G_k}$; that is
\begin{equation}\label{gamma_of_internal_prod}
\Gamma_G=\Gamma_{G_1}\times\cdots\times\Gamma_{G_k}.
\end{equation}
Hence according to Corollary~\ref{alpha_of_direct_prod_of_graphs}, we get the independence number of $G$.
\begin{lem}\label{independence_of_internal_prod} With the notation above, we have
\[
\alpha(\Gamma_G)=\max_j\{\,\alpha(\Gamma_{G_j})\prod_{\substack{i=1,\dots,n\\ i\neq j}}|G_i|\,\}.\qed
\]
\end{lem}

\begin{thm}\label{EKR_for_internal_prod} With the notation above, if all the $G_i$ have the EKR property, then $G$ also has the EKR property.
\end{thm}
\begin{proof}
For any $x\in \Omega$, the stabilizer of $x$ in $G$ is $G_1\cdot\cdots\cdot G_{j-1}\cdot (G_j)_x\cdot G_{j+1}\cdot \cdots\cdot G_k$, where $x\in \Omega_j$. Hence 
\[
|G_x|=|(G_j)_x|\prod_{\substack{i=1,\dots,n\\ i\neq j}}|G_i|.
\]
Therefore, using Lemma~\ref{independence_of_internal_prod}, if all the $G_i$ have the EKR property, then $G$ also has the EKR property.
\end{proof}
Using Theorem~\ref{cyclic_groups}, one can observe the following.
\begin{cor} 
For any sequence $r_1,\dots,r_k$ of positive integers, the internal direct product $\mathbb{Z}_{r_1}\cdot\mathbb{Z}_{r_2}\cdot\cdots\cdot \mathbb{Z}_{r_k}$ has the EKR property.\qed
\end{cor}

Let $\lambda=[\lambda_1,\dots,\lambda_k]$ be a partition of $n$ (see Section~\ref{rep_sym}). Define a set partition of $[n]$ by 
$[n]=\Omega_1\cup\cdots\cup\Omega_k$, where $\Omega_i=\{\lambda_1+ \cdots+\lambda_{i-1}+1,\dots,\lambda_1+\cdots+\lambda_i\}$. Then the internal direct product $\sym(\Omega_1)\cdot \sym(\Omega_2)\cdot\cdots\cdot \sym(\Omega_k)$  is called the \txtsl{Young subgroup} of $\sym(n)$ corresponding to $\lambda$ and is denoted by $\sym(\lambda)$. An easy consequence  of Theorem~\ref{EKR_for_internal_prod} and Theorem~\ref{EKR_for_sym} is the following.
\begin{cor} Any Young subgroup has the EKR property.\qed
\end{cor}
It is not difficult to see that $\Gamma_{\sym(n)}$ is connected if and only if $n\neq 3$,  $\Gamma_{\sym(3)}$ is the disjoint union of two complete graphs $K_3$ and $\Gamma_{\sym([2,2,2])}$ is disconnected. From this we can deduce that if $\lambda=[3,2,\dots,2], [3,3]$ or $[2,2,2]$, then $\Gamma_{\sym(\lambda)}$  will be disconnected and one can find maximum independent sets which do not correspond to cosets of  point-stabilizers. More generally, if $\lambda$ is any partition of $n$ which ``ends'' with one of these three cases, then $\sym(\lambda)$ fails to have the strict EKR property. In \cite{KuW07} the authors have shown that these are the only Young subgroups which don't have the strict EKR property. In other words, they have proved the following.
\begin{thm}\label{strict_EKR_for_Young}
Let $\lambda=[\lambda_1,\dots,\lambda_k]$ be a partition of $n$ with all parts larger than one. Then $\sym(\lambda)$ has the strict EKR property unless one of the following hold
\begin{enumerate}[(a)]
\item $\lambda_j=3$ and $\lambda_{j+1}=\cdots= \lambda_k= 2$, for some $1\leq j<k$;
\item $\lambda_k = \lambda_{k-1} = 3$;
\item $\lambda_k = \lambda_{k-1} = \lambda_{k-2} = 2$. \qed
\end{enumerate}
\end{thm}

Finally we introduce the wreath product and probe whether  it has either the EKR or the strict EKR property. 
Let $G\leq \sym(m)$ and $H\leq \sym(n)$. Then the \txtsl{wreath product} of $G$ and $H$, denoted by $G\wr H$ is the group whose set of elements is
\[
(\underbrace{G\times \cdots\times G}_{n\,\,\text{times}})\times H,
\]
and the binary operation is defined as follows:
\[
(g_1,\dots,g_n,h)\cdot (g'_1,\dots,g'_n,h'):= (g_1 g'_{h(1)},\dots,g_n g'_{h(n)}\,,\,hh').
\]
It is a straight-forward exercise to show that this, indeed, defines a group. In particular, note that the identity element of $G\wr H$ is $(\id_G,\dots,\id_G,\id_H)$ and for any $(g_1,\dots,g_n,h)\in G\wr H$, 
\[
(g_1,\dots,g_n,h)^{-1} = (g_{h^{-1}(1)}^{-1},\dots,g_{h^{-1}(n)}^{-1},h^{-1}).
\]
Note also that the size of $G\wr H$ is $|G|^n|H|$. 
We point out that  $G\wr H$ is in fact the ``semi-direct product''
\[
(\underbrace{G\times \cdots\times G}_{n\,\,\text{times}})\rtimes H,
\]
when the action of $H$ on $G\times \cdots\times G$ is defined as simply permuting the positions of copies of $G$ (see \cite[Section 2.5]{dixon2012permutation} for a more detailed discussion on semi-direct products). It is not hard to see that this group is the stabilizer of a partition of the set $[nm]$ into $n$ parts each of size $m$. 

Now assume $\Omega=[m]\times[n]$. Then we observe the following.
\begin{lem}\label{wreath_action}
 The group $G\wr H$ acts on $\Omega$ in the following fashion: 
\begin{equation}\label{wreath_action_eq}
(x,j)^{(g_1,\dots,g_n,h)}:=(x^{g_j},j^h)=(g_j(x), h(j)),
\end{equation}
for any $(x,j)\in \Omega$ and $(g_1,\dots,g_n,h)\in G\wr H$. 
\end{lem}
\begin{proof} It is obvious that the pair on the right hand side of (\ref{wreath_action_eq}) is in $\Omega$. We also see that for any $(x,j)\in \Omega$,
\[
(x,j)^{\id_{G\wr H}}=(x,j)^{(\id_G,\dots,\id_G,\id_H)}=(\id_G(x),\id_H(j))=(x,j).
\]
Finally assume $(g_1,\dots,g_n,h) , (g'_1,\dots,g'_n,h') \in G\wr H$. Then 
\begin{align*}
(x,j)^{(g_1,\dots,g_n,h) \cdot (g'_1,\dots,g'_n,h')}&=(x,j)^{(g_1 g'_{h(1)},\dots,g_n g'_{h(n)}\,,\,hh')}\\[.2cm]
& = \left( x^{g_j g'_{h(j)}},j^{hh'}\right)= \left(g'_{h(j)}(g_j(x))\,,\,h'(h(j))\right);
\end{align*}
on the other hand
\begin{align*}
\left((x,j)^{(g_1,\dots,g_n,h)}\right)^{ (g'_1,\dots,g'_n,h')}&=(x^{g_j},j^h)^{ (g'_1,\dots,g'_n,h')}\\[.2cm]
&=\left((x^{g_j})^{g'_{h(j)}} , (j^h)^{h'} \right)= \left(g'_{h(j)}(g_j(x)) , h'(h(j)) \right).\qedhere
\end{align*}
\end{proof}
Note this implies that if $(g_1,\dots,g_n,h)$ has a fixed point $(x,j)$, then $h(j)=j$ and $g_j(x)=x$. Thus, it is not difficult to verify the following.
\begin{lem}\label{wreath_stabilizer}
For any pair $(x,j)\in \Omega$, the stabilizer of $(x,j)$ in $G\wr H$ is 
\[
\left(G\times\cdots \times\underset{j\text{th position}}{(G)_x}\times\cdots\times G\right)\times H_j.\qed
\]
\end{lem}

\begin{thm}\label{EKR_for_wreath}
If $G\leq \sym(m)$ and $H\leq \sym(n)$ have the EKR property, then $G\wr H$ also has the EKR property.
\end{thm}
\begin{proof}
For convenience we let $W:=G\wr H$ and $P=G\times\cdots\times G$. Note that by the definition of the wreath product, $P$ is in fact the internal direct product of  $G_1,\dots, G_n$, where $G_i\cong G$ and $G_i\leq \sym([m]\times \{i\})$, for any $1\leq i\leq n$. Hence according to (\ref{gamma_of_internal_prod}), we have
\[
\Gamma_P=\underbrace{\Gamma_G\times \cdots\times \Gamma_G}_{n\,\,\text{times}}.
\]
Consider the lexicographic product $\Gamma=\Gamma_H[\Gamma_P]$. Define the map $f: \Gamma\to \Gamma_W$ by
\[
f(h,(g_1,\dots,g_n))=(g_1,\dots,g_n,h).
\]
We claim that $f$ is a homomorphism.  To prove this, assume $(h,(g_1,\dots,g_n))$ and $(h',(g'_1,\dots,g'_n))$ are adjacent in $\Gamma$. We should show that
\begin{equation}\label{f-is-hom}
(g'_1,\dots,g'_n,h')\cdot(g_1,\dots,g_n,h)^{-1}=
(g'_1  g_{h'h^{-1}(1)}^{-1},\dots,g'_n g_{h'h^{-1}(n)}^{-1},h'h^{-1})
\end{equation}
has no fixed point. By the definition of the lexicographic product, either $h\sim h'$ in $\Gamma_H$ or $h=h'$ and $(g_1,\dots,g_n)\sim (g'_1,\dots,g'_n)$ in $G$. In the first case, $h'h^{-1}$ has no fixed point. Thus $(g'_1,\dots,g'_n,h')\cdot(g_1,\dots,g_n,h)^{-1}$ cannot have a fixed point. In the latter case, $(g'_1,\dots,g'_n)(g_1,\dots,g_n)^{-1}$ has no fixed point; thus, according to (\ref{f-is-hom}),
\[
(g'_1,\dots,g'_n,h')\cdot(g_1,\dots,g_n,h)^{-1}= (g'_1  g_1^{-1},\dots,g'_n g_n^{-1},\id_H)
\]
cannot have a fixed point. Thus the claim is proved. 

We  can, therefore, apply Proposition~\ref{fractional_chrom_ineq} to get
\[
\frac{|V(\Gamma)|}{\alpha(\Gamma)}\leq \frac{|V(\Gamma_W)|}{\alpha(\Gamma_W)}.
\]
Therefore, using Proposition~\ref{independence-of-lex}, we have
\begin{equation}\label{alphas}
\alpha(\Gamma_W)\leq \alpha(\Gamma_P)\alpha(\Gamma_H).
\end{equation}
But since $G$ has the EKR property, according to Theorem~\ref{EKR_for_internal_prod}, $P$ has the EKR property; this means that there is  a point $x\in [m]$ such that 
\[
\alpha(\Gamma_P)=|P_x|.
\]
Similarly, since $H$ has the EKR property, there exists a $j\in[n]$ such that 
\[
\alpha(\Gamma_H)=|H_j|.
\]
This, along with  Lemma~\ref{wreath_stabilizer}, implies that $\alpha(\Gamma_W)= |W_{(x,j)}|$.
\end{proof}
In the case of symmetric groups, we can say more.
\begin{prop} The group $\sym(m)\wr \sym(n)$ has the EKR property. Furthermore, if  $m\geq 4$, then $\sym(m)\wr \sym(n)$ has the strict EKR property.
\end{prop}
\begin{proof} The first part follows from Theorem~\ref{EKR_for_wreath}. For the second part, as in the proof of Theorem~\ref{EKR_for_wreath}, we let $W=\sym(m)\wr \sym(n)$ and 
\[
P=\sym([m]\times\{1\})\times\cdots\times \sym([m]\times\{n\}).
\]
Let $S$ be an intersecting subset of $W$ of maximum size, i.e. $S$ has the size of a point-stabilizer in $W$. Without loss of generality we assume that $S$ contains the identity element of $W$. Consider the homomorphism $f:\Gamma_{\sym(n)}[\Gamma_P]\to \Gamma_W$ defined in the proof of Theorem~\ref{EKR_for_wreath}. It is obvious that $f$ is an injection; hence there is a copy of $\Gamma_{\sym(n)}[\Gamma_P]$ in $\Gamma_W$. This implies that $S$ is an independent set in $\Gamma_{\sym(n)}[\Gamma_P]$ of size $\alpha(\Gamma_{\sym(n)})\alpha(\Gamma_P)$. Then, according to Proposition~\ref{max_indy_in_lex} and the fact that $\sym(m)$ and $P$ have the strict EKR property (see Theorem~\ref{strict_EKR_for_Young}), we have that the projection of $S$ to $\Gamma_{\sym(n)}$ is  the  stabilizer of a point $j$ in $\sym(n)$, i.e. $S_{j,j}$, and the projection of $S$ in each copy of $\Gamma_P$ is  a point-stabilizer in $P$. Therefore
\[
S=\bigcup_{s\in S_{j,j}} P_{(x_s,j)}, 
\]
where $x_s\in[m]$, for any $s\in S_{j,j}$ and $P_{(x_s,j)}$ is the stabilizer of $(x_s,j)$ in $P$. Now if $x_{s}\neq x_{t}$, for some $s,t\in S_{j,j}$, then since $m\geq 4$, there will be an element in $P_{(x_{s},j)}$ which is adjacent to some element in $P_{(x_{t},j)}$ in the graph $\Gamma_W$, which contradicts the fact that $S$ is independent in $\Gamma_W$. Hence we must have
\[
S=\bigcup_{s\in S_{j,j}} P_{(x,j)}=\left(\sym(m)\times\cdots \times\underset{j\text{th position}}{(\sym(m))_x}\times\cdots\times \sym(m)\right)\times S_{j,j},
\]
for some $x\in [m]$. Now the proposition follows from Lemma~\ref{wreath_stabilizer}.
\end{proof}

%% file: chap-module_method.tex
\chapter{Module Method}\label{module_method}

As we mentioned in Chapter~\ref{introduction}, the approach in \cite{Karen} to the EKR problem of the symmetric group was vastly different from the one applied in \cite{MR2009400} where the theorem was first proved. Further, this new proof uses information from  the irreducible representations of the symmetric group. This algebraic proof opened a new way to approach the EKR problem for permutation groups; for example Meagher and Spiga in \cite{MeagherS11} used a similar method to solve the EKR problem for the projective general linear group $\PGL(2,q)$. They also questioned if one can apply this method for the projective special linear group $\PSL(2,q)$. In this chapter we will state this approach as a theorem, called the ``module method'', and will show how this will be useful in proving EKR theorems for permutation groups. 
Then using the module method, we will establish the strict EKR property  for the alternating group in Section~\ref{EKR_for_alt}. In Section~\ref{EKR_for_psl} we will approach the EKR theorem for $\PSL(2,q)$ and will show how the module method proves the strict EKR property for this group, provided  that Conjecture~\ref{M_fullrank_PSL} is true. In Section~\ref{sporadic}, we will apply this method to show that some of the sporadic permutation groups have the strict EKR property.


\section{Introduction of the module method}\label{module_method_intro}

Throughout this chapter we assume $G\leq \sym(n)$ to be a $2$-transitive permutation group, unless otherwise declared.  Recall from Chapter~\ref{EKR_perm_groups} that the problem of characterizing the maximum intersecting subsets of $G$ is equivalent to characterizing the maximum independent sets of the  graph $\Gamma_G$; hence in what follows we will use the graph interpretation of the problem rather than the original problem.  In order to explain the module method, first we define the canonical independent sets of $\Gamma_G$. For any $i,j\in [n]$, we define the \textsl{canonical independent sets}\index{independent set!canonical} $S_{i,j}$ as 
\begin{equation}\label{canonical}
S_{i,j}=\{\pi\in G\,\,|\,\, \pi(i)=j\}.
\end{equation}
The subset $S_{i,j}$  of the vertices of $\Gamma_G$ are, indeed, cosets of the point-stabilizers in $G$ under the natural action of $G$ on $[n]$. Obviously, $S_{i,j}$ is an independent set and since $G$ is transitive, $|S_{i,j}|=\frac{|G|}{n}$, for each $i,j\in [n]$. The sets $S_{i,j}$ form a collection of independent sets for $\Gamma_G$. The goal of the module method is to prove that these are the only maximum independent sets. For any $i,j\in [n]$, we denote the characteristic vector of $S_{i,j}$ with $v_{i,j}$. We will make use of the following lemma in the module method.

\begin{lem}\label{basis_for_standard}If for all $i,j\in[n]$, the vector $v_{i,j}$ lies in the direct sum of the standard and the trivial modules of $G$, then  the set
\[
B:=\{v_{i,j}-\frac{1}{n}\mathbf{1}\,|\, i,j\in[n-1]\} 
\]
is a basis for the standard module $V$ of $G$.
\end{lem}
\begin{proof}
 Since  the vectors $v_{i,j}-\frac{1}{n}\mathbf{1}$ are orthogonal to the all ones vector, we have $B\subset V$ and since the dimension of $V$ is equal to $|B|=(n-1)^2$, it suffices to show that $B$ is linearly    independent. Note, also, that since $\mathbf{1}$ is not in the span  of $v_{i,j}$ for $i,j\in[n-1]$, it is enough to prove that the set $\{v_{i,j}\,|\, i,j\in[n-1]\}$ is linearly independent.  Define a matrix $L$ to have the vectors $v_{i,j}$, with $i,j\in[n-1]$, as its columns.  Then the rows of $L$ are indexed by the elements of $G$ and the columns are indexed by the ordered pairs $(i,j)$, where $i,j\in [n-1]$; we will also assume that the ordered pairs are listed in lexicographic order.  It is, then, easy to see that
\[
L^\top L=\frac{(n-1)!}{2}\,I_{(n-1)^2}\,+\, \frac{(n-2)!}{2}\left( A(K_{n-1})\otimes A(K_{n-1})\right),
\]
where $I_{(n-1)^2}$ is the identity matrix of size $(n-1)^2$, $A(K_{n-1})$ is the adjacency matrix of the complete graph
$K_{n-1}$ and $\otimes$ is the tensor product (see Section~\ref{spectral_graph_theory}).  The distinct eigenvalues of $A(K_{n-1})$ are $-1$ and $n-2$; thus according to Proposition~\ref{evalues_of_tensor}, the eigenvalues of $A(K_{n-1})\otimes A(K_{n-1})$ are $-(n-2), 1, (n-2)^2$. This implies that the least eigenvalue of $L^\top L$ is
\[
\frac{(n-1)!}{2}-\frac{(n-2)(n-2)!}{2}>0.
\]
This proves that $L^\top L$ is non-singular and hence full rank. This, in turn, proves that $L$ is full rank and that  $\{v_{i,j}\,|\, i,j\in[n-1]\}$ is linearly independent.
\end{proof}

Define the $|G|\times n^2$ matrix $H$ to be the matrix whose columns are the vectors $v_{i,j}$, for all $i,j\in [n]$. Note that since $H$ has constant row-sums, the vector $\mathbf{1}$ is in the column space of $H$.
We denote by $H_{(i,j)}$ the column of $H$ indexed by the pair $(i,j)$, for any $i,j\in [n]$. Define the matrix $\overline{H}$ to be the matrix obtained from $H$ by deleting all the columns $H_{(i,n)}$ and $H_{(n,j)}$ for any $i,j\in[n-1]$. With a similar method as in the proof of \cite[Proposition 10]{MeagherS11}, we prove the following.
\begin{lem}\label{col_H_bar}
The matrices $H$ and $\overline{H}$ have the same column space.
\end{lem}
\begin{proof}
Obviously, the column space of $\overline{H}$ is a subspace of the column space of $H$; thus we only need to show that the vectors $H_{(i,n)}$ and $H_{(n,j)}$ are in the column space of $\overline{H}$, for any $i,j\in[n-1]$. Since $G$ is $2$-transitive, it suffices to show this for $H_{(1,n)}$. Define the vectors $v$ and $w$ as follows:
\[
v:=\sum_{i\neq 1,n} \sum_{j\neq n} H_{(i,j)}\quad\text{and}\quad w:=(n-3) \sum_{j\neq n} H_{(1,j)}\,+H_{(n,n)}.
\]
The vectors $v$ and $w$ are in the column space of $\overline{H}$. It is easy to see that for any $\pi \in G$, 
\[
v_\pi=
\begin{cases}
n-2,& \quad\text{if}\quad \pi(1)=n;\\
n-2, & \quad \text{if} \quad\pi(n)=n;\\
n-3,& \quad\text{otherwise},
\end{cases}
\quad\quad\quad
w_\pi=
\begin{cases}
0,& \quad\text{if}\quad \pi(1)=n;\\
n-2, &  \quad\text{if}\quad\pi(n)=n;\\
n-3,& \quad\text{otherwise}.
\end{cases}
\]
Thus 
\[
(v-w)_\pi=
\begin{cases}
n-2,& \quad\text{if}\quad \pi(1)=n;\\
0, & \quad \text{if} \quad\pi(n)=n;\\
0,& \quad\text{otherwise},
\end{cases}
\]
which means that $(n-2)H_{(1,n)}=v-w$. This completes the proof.
\end{proof}

If the columns of $\overline{H}$ are arranged so that the first $n$ columns correspond to the pairs $(i,i)$, for $i\in [n]$, and
the rows are arranged so that the first row corresponds to the identity element, and the next $|\dd_G|$ rows correspond
to the elements of $\dd_G$, then $\overline{H}$ has the following block structure:
\[
\begin{bmatrix}
1& 0\\
0 & M \\
B & C\\
\end{bmatrix}.
\]
Note that the rows and columns of $M$ are indexed by the elements of $\dd_G$ and the pairs $(i,j)$ with $i,j\in[n-1]$ and
$i\neq j$, respectively; thus $M$ is a $|\dd_G|\times (n-1)(n-2)$ matrix. Throughout the thesis, we will refer to this matrix simply as ``the matrix $M$ for $G$''.

\begin{prop}\label{condition_d} Let $G\leq \sym(n)$ be $2$-transitive. Then for any $x\in [n]$, there is an element in $G$ which has only $x$ as its fixed point.
\end{prop}
\begin{proof}
Since $G$ is transitive, it suffices to show it for $x=1$. We need to show that the stabilizer of $1$ in $G$, denoted $G_1$, has a derangement in its action on $\{2,\dots,n\}$. Suppose for every element $g\in G_1$, we have  $|\fix(g)|\geq 1$. This means that 
\[
\frac{1}{|G_1|} \sum_{g\in G_1} |\fix(g)|  \geq \frac{(n-1) +|G_1| -1}{|G_1|} = \frac{(n-2 + |G_1|)}{|G_1|},
\]
which is greater than $1$, if $n>2$. Hence by  Burnside's lemma (Theorem~\ref{burnside}), the number of orbits of the action of $G_1$ on $\{2,3,\dots,n\}$ is more than one which is a contradiction since   $G_1$ acts transitively on $\{2,3,\dots,n\}$. Thus there must be a derangement in $G_1$ and we are done.
\end{proof}

Now we are ready to prove the main theorem of this  section. Recall  that for any $2$-transitive group $G$, the standard representation of $G$ is irreducible (Proposition~\ref{standard_is_irr}).
\begin{thm}[Module method]\label{module_method_thm} Let $G\leq \sym(n)$ be 2-transitive and assume the following conditions hold:
\begin{enumerate}[(a)]
\item $G$ has the EKR  property;
\item for any maximum intersecting set $S$ in $G$, the vector $v_S$ lies in the direct sum of the trivial and the standard modules of $G$; and 
\item the matrix $M$ for $G$ has full rank.
\end{enumerate}
Then $G$ has the strict EKR property.
\end{thm}
\begin{proof}  Since $G$ has the EKR property, the maximum size of an intersecting subset of $G$ is $|G|/n$, i.e. the size of a point-stabilizer. Suppose that $S$ is of maximum size. It is enough to show  that $S=S_{i,j}$, for some $i,j\in [n]$.
Without loss of generality, we may assume that $S$ includes the identity element. 
By the assumption (b) and Lemma~\ref{basis_for_standard}, $v_S$ is in the column space of $H$; thus according to  Lemma~\ref{col_H_bar}, $v_S$ belongs to the column space of  $\overline{H}$; therefore
\[
\begin{bmatrix}
1 & 0\\
0 & M \\
B & C\\
\end{bmatrix}\begin{bmatrix} z \\ w  \end{bmatrix}
=v_S
\]
for some vectors $z$ and $w$. Since the identity is in $S$, no elements from $\dd_G$ are in $S$, thus the characteristic vector of
$S$ has the form
\[
v_S= \begin{bmatrix} 1 \\ 0 \\ t  \end{bmatrix}
\]
for some vector $01$-vector $t$. Thus we have $1^\top z=1$, $Mw=0$ and $Bz+Cw=t$. Since $M$ has full rank, $w=0$ and so $Bz=t$. Furthermore, according to Proposition~\ref{condition_d},  for any $x\in [n]$, there is a permutation $g_x\in G$  which has only $x$ as its fixed point; thus by a proper permutation of the rows of $B$, one can write
\[
B=\begin{bmatrix} I_{n} \\[.2cm] B'  \end{bmatrix}\quad 
\text{and}\quad 
Bz=\begin{bmatrix} z \\[.2cm] B'z  \end{bmatrix}. 
\]
Since $Bz$ is equal to the $01$-vector $t$, the vector $z$ must also be a $01$-vector. But, on the other hand, $1^\top z=1$, thus we conclude that exactly one of the entries of $z$ is equal to $1$. This means that $v_S$ is the characteristic vector of the stabilizer of a point.
\end{proof}
We point out that the module condition (b) is the reason we call this method the module method.

\section{EKR for the alternating group}\label{EKR_for_alt}
Recall that it has been proved in \cite{KuW07} that the alternating group $\alt(n)$ has the strict EKR property for $n\geq 5$. 
In this section we will apply the module method to present an entirely new proof for this result.
\begin{thm}\label{main_Alt}
For $n\geq 5$, any intersecting subset of $\alt(n)$ has size at most
\[
\frac{(n-1)!}{2}.
\]
An intersecting subset of $\alt(n)$ achieves this bound if and only if it is a coset of a point-stabilizer.
\end{thm}

For any group  $G$,  since $\dd_{G}$ is a union of conjugacy classes of $G$, the graph  $\Gamma_{G}$ is a union of graphs in the conjugacy class scheme of $G$ (see \cite[Example 2.1 (2)]{MR882540} or \cite[Example 2.4(2)]{delsarte1973algebraic}). Thus the clique-coclique bound (Theorem~\ref{clique_coclique_bound}) applies to $\Gamma_{G}$.  Let $c_1,\dots,c_k$ be the derangement conjugacy classes of $G$. Then the matrices $A_1,\dots, A_k$ in the conjugacy class scheme of $G$ are, in fact, $|G|\times |G|$ matrices such that, for any $1\leq i\leq k$, the entry $(g,h)$ of $A_i$ is $1$ if $hg^{-1}\in c_i$, and $0$ otherwise. Recall from Section~\ref{bounds_on_indy} that these $A_i$ are simultaneously diagonalizable and, hence,  have common eigenspaces. The idempotents of this scheme are $E_\chi$, where $\chi$ runs through the set of all irreducible characters of $G$; the entries of $E_\chi$ are given by
\begin{equation}\label{idempotent}
(E_\chi)_{\pi,\sigma}=\frac{\chi(1)}{|G|}\chi(\pi^{-1}\sigma).
\end{equation}
To show that $E_{\chi}$ are, indeed, the projections to the common eigenspaces, it can be shown that, for any $1\leq i\leq k$, 
\[
A_i E_{\chi} = \frac{|c_i|\chi(\sigma_i)}{\chi(1)} E_\chi, \quad \sigma_i\in c_i.
\]
This also shows that the eigenvalue of $A_i$, corresponding to the eigenspace arising from $\chi$, is $\frac{|c_i|\chi(\sigma_i)}{\chi(1)}$, for any $1\leq i\leq k$. (See~\cite{MR546860} or \cite[Sections 2.2 and 2.7]{MR882540} for a proof of this.) The vector space generated by the columns of $E_\chi$ is called the \textsl{module corresponding to $\chi$}\index{module corresponding to a character} or simply the \textsl{$\chi$-module} of $\Gamma_G$. For any character $\chi$ of $G$ and any subset $X$ of $G$ define
\[
\chi(X)=\sum_{x\in X}\chi(x).
\]
Using Corollary~\ref{clique_vs_coclique} and Equation (\ref{idempotent}) one observes the following.
\begin{cor}\label{at_most_one_non-zero}
 Assume the clique-coclique bound holds with equality for the graph $\Gamma_G$ and let $\chi$ be an irreducible character of $G$ that is  not the trivial character. If there is a clique $C$ of maximum size in $\Gamma_G$ with $\chi(C)\neq 0$, then
\[
E_\chi\,v_S = 0
\]
for any maximum independent set $S$ of $\Gamma_G$.\qed
\end{cor}
In other words, provided that the clique-coclique bound holds with equality, for any module of $\Gamma_G$ (other than the trivial  module) the projection of at most one of the vectors $v_C$ and $v_S$ will be non-zero, where $S$ is any maximum independent set and $C$ is any maximum clique.  In this section we let $G_n=\alt(n)$ for simplicity. In what follows we will show that the conditions of Theorem~\ref{module_method_thm} hold. We will find cliques $C$ such that $E_{\chi} v_C\neq 0$, for all irreducible characters $\chi$ of $G_n$ except the trivial and the standard characters to prove that condition (b) of the module method holds.

\subsection{The standard module}\label{standard_module}

Recall from Section~\ref{rep_sym} that the representation of $\sym(n)$ corresponding to $[n]$ is the trivial representation, the character of which  is equal to $1$ for every permutation. Also if $\lambda= [n-1,1]$, then the irreducible
representation $S^\lambda$ of  $\sym(n)$ is  the standard representation. For $n\geq 5$, $\lambda=[n-1,1]$ is not
symmetric; hence according to Theorem~\ref{char_values_of_alt}, the restriction $V$ of $S^\lambda$ to $G_n$ is also irreducible. We also deduce from Theorem~\ref{char_values_of_alt} that this representation is the standard representation of $G_n$ and $V$ is the standard module of $G_n$. Recall from Section~\ref{rep_basic} that the value of the character of the standard representation on a permutation $\sigma$ is the number of elements of $[n]$ fixed by $\sigma$ minus $1$ and that the dimension of this representation is $n-1$ (see Lemma~\ref{dim_of_standard_char}).

In this subsection we prove that the conditions (a) and (b) of Theorem~\ref{module_method_thm} hold for the alternating group. 
Ku and Wong conjectured \cite{KuW07} that the least eigenvalue of the derangement graph $\Gamma_{\sym(n)}$ of the symmetric group is given by the standard representation (see Theorem~\ref{Diaconis}). This was proved by Renteln in \cite{Renteln2007}. Based on our observations of several examples, we feel that there is a similar situation for the case of the alternating group. In other words, we propose the following.
\begin{conj}\label{least_eval_of_alt} The least eigenvalue of $\Gamma_{G_n}$ is given only by the standard representation of $G_n$.
\end{conj}
Note that, using Theorem~\ref{Diaconis} and Theorem~\ref{char_values_of_alt},  the eigenvalue of $\Gamma_{G_n}$ given by the standard representation of $G_n$ is 
\[
\eta_{[n-1,1]}=\frac{1}{\chi_{[n-1,1]}(\id)}\sum_{\sigma\in \dd_{G_n}}\chi_{[n-1,1]}(\sigma)=\frac{-|\dd_{G_n}|}{n-1}.
\]
Now if Conjecture~\ref{least_eval_of_alt} is true, then according to the first part of Theorem~\ref{ratio2}, the size of any independent set $S$ of $\Gamma_{G_n}$ is bounded above by
\[
\frac{|G_n|}{1-\frac{|\dd_{G_n}|}{\eta_{[n-1,1]}}}=\frac{n!/2}{n}=\frac{(n-1)!}{2}.
\]
Since the size of any point-stabilizer in $G_n$ is $(n-1)!/2$, this means that  $G_n$ has the EKR property; i.e. condition(a) of Theorem~\ref{module_method_thm} holds for $G_n$. Furthermore, the second part of  Theorem~\ref{ratio2} yields that for any maximum independent set $S$ of vertices of $\Gamma_{G_n}$, the characteristic vector $v_S$ of $S$ lies in the direct sum of the trivial and the standard modules of $G_n$; i.e. condition (b) of Theorem~\ref{module_method_thm} holds for $G_n$.

Solving Conjecture~\ref{least_eval_of_alt} does not  seem  easy   even using the fact that the standard representation of $\sym(n)$ gives the least eigenvalue of $\Gamma_{\sym(n)}$. Thus the method we use here is similar to the work done in \cite{Karen}. In other words, we will use the clique-coclique bound (Theorem~\ref{clique_coclique_bound}). To do this, we will show that
the clique-coclique bound for $\Gamma_{G_n}$ holds with equality by defining sufficiently large cliques $C$. Moreover, for each $\lambda$, which is neither $[n]$ nor $[n-1,1]$, we will observe that  $E_{\lambda}v_C \neq 0$. From this, using Corollary~\ref{clique_vs_coclique}, we conclude that $E_{\lambda}v_S = 0$ for any maximum independent set, unless $E_{\lambda}$ is the projection to either the trivial module or the standard module. We will consider two cases, first when $n$ is odd and second when it is even.


Assume $n\geq 5$ is odd. In \cite[Theorem 1.1]{AlspachGSV03} it has been proved that there is a decomposition of the arcs of the complete digraph $K_n^\ast$ on $n$ vertices to $n-1$ directed cycles of length $n$. Each of these cycles corresponds to an $n$-cycle in $G_n$. Since no two such decompositions share an arc in $K_n^\ast$, no two of the corresponding permutations intersect.  Let $C^o$ be the set of these permutations together with the identity element of $G_n$. Then $C^o$ is a clique in $\Gamma_{G_n}$ of size $n$. We can therefore prove the following.
\begin{lem}\label{EKR_alt_odd} If $n\geq 5$ is odd, then $G_n$ has the EKR property.
\end{lem}
\begin{proof} According to the clique-coclique bound, for any independent set $S$ in  $\Gamma_{G_n}$, we have
\[
|S|\leq \frac{|G_n|}{|C^o|}=\frac{n!/2}{n}=\frac{(n-1)!}{2},
\]
which is the size of any point-stabilizer in $G_n$.
\end{proof}
The set of all $n$-cycles from $\sym(n)$ forms a pair of split conjugacy classes $c'_0$ and $c''_0$ in $G_n$. Thus all the non-identity elements of $C^o$ lie in $c'_0\cup c_0''$. Theorem~\ref{char_values_of_alt} gives all the irreducible representations of $G_n$. We will use the notation of this theorem; in particular, the reader should note that the superscript refers to characters of $\sym(n)$ and the subscript to those of $\alt(n)$. In the proof of the next lemma we will use the ``double factorial'' notation; we define  $a!!=a(a-2)(a-4)\cdots 2$ if $a$ is even positive integer and $a!!=a(a-2)(a-4)\cdots 1$, if $a$ is odd.

\begin{lem}\label{n_odd} Let $n\geq 5$ be odd. Then for any  irreducible character $\chi$ of $G_n$, other than the standard character, we have $\chi(C^o)\neq 0$.
\end{lem}
\begin{proof}
 First consider the case where $\chi$ is the character of the restriction of the representation  $S^{\lambda}$, where $\lambda$ is not symmetric, to $G_n$. Let  $\chi^\lambda$ be the character of $S^{\lambda}$ then $\chi = \chi_\lambda$. According to Theorem~\ref{char_values_of_alt}, $\chi_\lambda$ has the same values on $c'_0$ and $c_0''$ and this  value is equal to the value of $\chi^\lambda$ on $c'_0 \cup c''_0$. We compute
\[
\chi_\lambda(C^o)=\sum_{x\in C^o}\chi_\lambda(x)=\chi^\lambda(1)+(n-1)\chi^\lambda(\sigma),
\]
where $\sigma$ is a cyclic permutation of length $n$. Using the corollary of the Murnaghan-Nakayama Rule (Corollary~\ref{cor:appofMN}), we have $\chi^\lambda(\sigma)\in \{0,\pm 1\}$. Therefore, if $\chi_\lambda(C^o)=0$, since $\chi_\lambda(1) >0$, it must be that $\chi^\lambda(\sigma)= -1$ and then $\chi^\lambda(1)=n-1$. The representations corresponding to the partition $[n-1,1]$ and its transpose, $[2,1,\ldots,1]$, are the only representations of $\sym(n)$ of dimension $n-1$ and according to Theorem~\ref{reps_of_alt}, their restrictions to $G_n$ are both isomorphic to the standard representation of $G_n$.

Next assume that $\chi$ is the character of one of the two irreducible representations $W'$ or $W''$, where $W = W'\oplus W''$ is the restriction of $S^\lambda$ to $G_n$; in this case $\lambda$ must be symmetric.  Thus, $\chi = \chi_\lambda'$ (the case  when $\chi = \chi_\lambda''$ is identical, so we omit it).  If $\lambda$ is not the hook $[(n+1)/2,1,\ldots,1]$, then according to Theorem~\ref{char_values_of_alt}, we have
\[
\chi_\lambda'(C^o)=\sum_{x\in C^o}\chi_\lambda'(x)=\frac{1}{2}\chi^\lambda(1)+(n-1)\frac{1}{2}\chi^\lambda(\sigma).
\]
Thus, as in the previous case, if $\chi_\lambda'(C^o)=0$, then we must have $\chi^\lambda(1)=n-1$ which is a contradiction.

The final case that we need to consider is when $\chi$ is the character of one of the two irreducible representations whose sum is  the representation formed by restricting $S^\lambda$ to $G_n$ where $\lambda=[(n+1)/2,1,\ldots,1]$. Again we assume that $\chi = \chi_\lambda'$ (since the case for $\chi = \chi_\lambda''$ is identical) and using Theorem~\ref{char_values_of_alt}, we have 
\begin{align*}
\chi_\lambda'(C^o)&=\sum_{x\in C^o}\chi_\lambda'(x)\\
&=\chi'_\lambda(1)+\sum_{x\in C^o\cap c'_0}\chi_\lambda'(x)+ \sum_{x \in C^o\cap c''_0}\chi_\lambda'(x)\\
&=\frac{1}{2}\chi^\lambda(1)+r'\, \frac{1}{2}\left[(-1)^\frac{n-1}{2}+\sqrt{(-1)^\frac{n-1}{2} \,n}\right] +r'' \,\frac{1}{2}\left[(-1)^\frac{n-1}{2}-\sqrt{(-1)^\frac{n-1}{2}\,n} \right],
\end{align*}
where $r'=|C^o\cap c'_0|$ and $r''=|C^o\cap c''_0|$. Note that $r'+r'' =n-1$.  Hence, if $\chi_\lambda'(C^o)=0$, then we must have 
\begin{equation}\label{eq_on_dim_lambda}
-\chi^\lambda(1)=r'\left[(-1)^\frac{n-1}{2}+\sqrt{(-1)^\frac{n-1}{2} \,n}\right]\,+ \,r''\left[(-1)^\frac{n-1}{2}-\sqrt{(-1)^\frac{n-1}{2}\,n} \right].
\end{equation}
Note that
\[
\chi^\lambda(1)=\frac{2^{n-1}(n-2)!!}{(n-1)!!}.
\]
 Consider the following two cases. If $4\nmid n-1$, then (\ref{eq_on_dim_lambda}) implies that 
\begin{equation}\label{not_divisible_by_4}
-\frac{2^{n-2}(n-2)!!}{(n-1)!!}=-(n-1)+\sqrt{-n}(r'-r'').
\end{equation}
It follows, then, that  $r'=r''$ and so
\[
\frac{2^{n-1}(n-2)!!}{(n-1)!!}=n-1,
\]
since this only holds for $n=3$, this is a contradiction.

On the other hand, if $4\mid n-1$, then (\ref{eq_on_dim_lambda}) implies that 
\begin{equation}\label{divisible_by_4}
-\frac{2^{n-1}(n-2)!!}{(n-1)!!}=r'\, (1+\sqrt{n})\,\,+ \,\, r''\, (1-\sqrt{n});
\end{equation}
that  is,
\begin{align}
\frac{2^{n-1}(n-2)!!}{(n-1)!!} &=-(n-r''-1)\, (\sqrt{n}+1)\,\,+ \,\, r''\, (\sqrt{n}-1)  \nonumber \\
                               &\leq (n-1)\, (\sqrt{n}-1)\leq n^{\frac{3}{2}}. \label{asymptotic_inequality}
\end{align}
Note that
\[
\frac{2^{n-1}(n-2)!!}{(n-1)!!}=\frac{2^{n-1}}{n}\frac{n!!}{(n-1)!!}>\frac{2^{n-1}}{n};
\]
thus (\ref{asymptotic_inequality}) yields
\[
2^{n-1}< n^\frac{5}{2}.
\]
It is easily seen that this inequality fails for all $n\geq 9$. Finally, note that (\ref{divisible_by_4}) and (\ref{not_divisible_by_4}) lead us to contradictions if $n=5$ and if $n=7$, respectively. This completes the proof of the lemma.
\end{proof}

Next we consider the case  when $n$ is even; so we assume $n\geq 6$ and even. According to \cite[Theorem 1.1]{AlspachGSV03}, the arcs of the complete digraph $K_n^\ast$ can be decomposed to $n-1$ pairs of vertex-disjoint directed cycles of length $n/2$. Each of these pairs corresponds to a permutation in $G_n$ which is a product of two cyclic permutations of length $n/2$. Let $C^e$ be the set of these permutations together with the identity element of $G_n$. Then, similar to the previous part, $C^e$ is a clique in $\Gamma_{G_n}$ and similarly we observe the following.
\begin{lem}\label{EKR_alt_even} If $n\geq 6$ is even, then $G_n$ has the EKR property.\qed
\end{lem}

Note that the non-identity elements of $C^e$ lie in a non-split conjugacy class $c$ of $G_n$. Now we prove the equivalent of Lemma~\ref{n_odd} for even $n$, using this set $C^e$.

\begin{lem}\label{n_even} Let $n\geq 6$ be even. Then for any irreducible character $\chi$ of $G_n$, which is not the standard   character, we have $\chi(C^e)\neq 0$.
\end{lem}
\begin{proof}
First consider the case $\chi = \chi_\lambda$ where $\lambda$ is not symmetric. Using the notation of  Theorem~\ref{char_values_of_alt}, we have
\[
\chi_\lambda(C^e)=\sum_{x\in C^e}\chi_\lambda(x)=\chi^\lambda(1)+(n-1)\chi^\lambda(\sigma),
\]
where $\sigma$ is a product of two disjoint cyclic permutations of length $n/2$. Now, suppose $\chi_\lambda(C^e)=0$. Then
\begin{equation}\label{even_nonsymmetric}
-\chi^\lambda(1)=(n-1)\chi^\lambda(\sigma).
\end{equation}
According to  Corollary~\ref{cor:appofMNtwo}, we have $\chi^\lambda(\sigma)\in\{0,\pm 1,\pm2\}$. If $\chi^\lambda(\sigma)=0, 1$ or $2$, then  (\ref{even_nonsymmetric}) yields a contradiction with the fact that $\chi_\lambda(1)$ is strictly positive. Also if $\chi^\lambda(\sigma)=-1$, then we must have $\chi^\lambda(1)=n-1$ which contradicts the fact that the standard representation and its conjugate are the only irreducible representations of $\sym(n)$ of dimension $n-1$. Hence, suppose $\chi^\lambda(\sigma)=-2$.  Then $\chi^\lambda(1)=2n-2$.  According to Lemma~\ref{char_of_two_layer_hook}, $\lambda$ must be a two-layer hook or a symmetric near hook. Then by Lemma~\ref{dimensions_of_near_hooks} and Lemma~\ref{dimensions_of_2_layer_hooks}, the dimension of $\chi$ is strictly greater than $2n-2$. 

Next consider the case where $\chi$ is the character of one of the two irreducible representations in the restriction of the   representation  $S^{\lambda}$ to $G_n$, where $\lambda$ is symmetric; so $\chi = \chi_\lambda'$ or $\chi_\lambda''$. We will show that $\chi_\lambda'(C^e)\neq 0$; the proof that $\chi_\lambda''(C^e)\neq 0$ is similar. We have
\[
\chi_\lambda'(C^e)=\sum_{x\in C^e}\chi_\lambda'(x)=\frac{1}{2}\chi^\lambda(1)+(n-1)\frac{1}{2}\chi^\lambda(\sigma),
\]
where $\sigma$ is a product of two disjoint $n/2$-cycles. If $\chi_\lambda(C^e)=0$, then with the same argument as above, we get a
contradiction.
\end{proof}
Note that Lemma~\ref{EKR_alt_odd} and Lemma~\ref{EKR_alt_even} show that condition (a) holds for $G_n$. We now prove the main theorem of this section; that is, we show that condition (b) of Theorem~\ref{module_method_thm} holds for $G_n$. 

\begin{prop}\label{max_indy_is_in_standard_module}
  Let $S$ be a maximum  intersecting subset of $G_n$. Then the vector $v_S$ is in the direct sum of the trivial and the standard modules of $G_n$.
\end{prop}
\begin{proof} Let $S_{1,1}$ be the point stabilizer for $1$ in $G_n$; so $S_{1,1}$ is an independent set of the maximum size, i.e. $\frac{(n-1)!}{2}$, in $\Gamma_{G_n}$. Then the  cliques $C^o$ and $C^e$, together with $S_{1,1}$, prove that the clique-coclique bound holds with  equality for $\Gamma_{G_n}$.  Given any irreducible character $\chi$  of $G_n$, except the standard character and the trivial character, according to Lemma~\ref{n_odd} and Lemma~\ref{n_even}, there is a maximum clique  $C$, such that $\chi(C)\neq 0$. Hence, according to Corollary~\ref{at_most_one_non-zero}, we have $E_\chi v_S=0$, for any maximum independent set $S$. This implies that $v_S$ is in the direct sum of the trivial and the standard modules of $G_n$.
\end{proof}

\subsection{Proof that $\alt(n)$ has the strict EKR property}\label{main_proof}
In this part, we first show that condition (c) of Theorem~\ref{module_method_thm} holds and  complete the proof of Theorem~\ref{main_Alt}. 

\begin{prop}\label{fullrank_alt}
For all $n\geq 5$, the matrix  $M$ for $G_n$ has full rank.
\end{prop}
\begin{proof} First assume $n$ is odd. Consider the submatrix $M_1$ of $M$ that is comprised of all the rows in $M$ that are indexed by  cyclic permutations of length $n$.  Set $T=M_1^\top M_1$; it suffices to  show that $T$ is non-singular. Consider all types of entries of $T$. If $i,j,k,l$ are in $[n-1]$, then the following are all  possible cases for the pairs $(i,j)$ and $(k,l)$.
\begin{itemize}
\item $i=k$ and $j=l$; in this case $T_{(i,j),(i,j)}=(n-2)!$; because  the number of all $n$-cycles mapping $i$ to $j$ is $(n-2)!$.
\item $i=l$ and $j=k$; in this case $T_{(i,j),(j,i)}=0$; because the only case in which an $n$-cycle can swap $i$ and $j$ is $n=2$.
\item $i=k$ and $j\neq l$; in this case $T_{(i,j),(i,l)}=0$; because  there is no permutation mapping $i$ to two different numbers.
\item $i\neq k$ and $j=l$; again $T_{(i,j),(k,j)}=0$.
\item $i\neq l$ and $j=k$; in this case $T_{(i,j),(j,l)}=(n-3)!$; because the number of all $n$-cycles mapping $i$ to $j$ and $j$ to
  $l$ is $(n-3)!$.
\item $i=l$ and $j\neq k$; in this case $T_{(i,j),(k,i)}=(n-3)!$; with  a similar reasoning as above.
\item $\{i,j\}\cap \{k,l\}=\emptyset$; in this case  $T_{(i,j),(k,l)}=(n-3)!$; because the number of $n$-cycles  mapping $i$ to $j$ and $k$ to $l$ is $\binom{n-3}{1}(n-4)!=(n-3)!$.
\end{itemize} 
Therefore, one can write $T$ as
\begin{equation}\label{4th_eq_alt}
T=(n-2)!I+(n-3)!A(X_n),
\end{equation}
where $I$ is the identity matrix of size $(n-1)(n-2)$ and $A(X_n)$ is the adjacency matrix of the pairs graph $X_n$ defined in Section~\ref{spectral_graph_theory}.  By Lemma~\ref{least_eval_of_X_n}, the least eigenvalue of $X_n$ is greater than or equal to $-(n-3)$; thus using (\ref{4th_eq_alt}), the least eigenvalue of $T$ is at least
\[
(n-2)!-(n-3)!(n-3)=(n-3)!>0;
\]
therefore $T$ is non-singular and the proof is complete for the case $n$ is odd.

Now assume $n$ to be even.  Consider the subset of $\mathcal{D}_{G_n}$ which consists of all the permutations of $G_n$ whose cycle decomposition includes two cycles of length $n/2$ and let $M_2$ be the submatrix of $M$ whose rows are indexed by these permutations. Define $U=M_2^\top M_2$. With a similar approach as for the previous case, one can write $U$ as
\[
U=\frac{2(n-2)!}{n}I+\frac{2(n-3)!}{n}A(X_n).
\]
According to Lemma~\ref{least_eval_of_X_n}, the least eigenvalue of $U$ is at least
\[ 
\frac{2(n-2)!}{n}-\frac{2(n-3)!}{n}(n-3)=\frac{2(n-3)!}{n}>0;
\]
therefore $U$ is non-singular and the proof is complete.
\end{proof}

Now we are ready to prove the main theorem.
\begin{proof}(Theorem~\ref{main_Alt}) According to Lemma~\ref{EKR_alt_odd} and Lemma~\ref{EKR_alt_even} condition (a) holds. Also using Proposition~\ref{max_indy_is_in_standard_module} and Proposition~\ref{fullrank_alt} conditions (b) and (c) hold, respectively; hence the proof is complete.
\end{proof}

An interesting result of Theorem~\ref{main_Alt} is that it implies
that the symmetric group also has the strict EKR property. In fact, Theorem~\ref{main_Alt} along with Theorem~\ref{2transitive_subgroup} provides an alternative proof of the following theorem which was initially proved in
\cite{MR2009400}.
\begin{cor} 
For any $n\geq 2$, the group $\sym(n)$ has the strict EKR property.
\end{cor}


\section{EKR for $\PSL(2,q)$}\label{EKR_for_psl}

Throughout this section, we assume that $q$ is a prime power and use the notation $\F_q$ for the finite field of order $q$.  
Let $\PSL(2,q)$ be the projective special linear group acting on  the projective line $\mathbb{P}_q$. In \cite{MeagherS11}, the authors have proved that the projective general linear group $\PGL(2,q)$ has the strict EKR property. This was a motivation for them to conjecture that  $\PSL(2,q)$ also has the  strict EKR property. 
\begin{conj}\label{PSL_main}
 For any prime power $q$, the group $\PSL(2,q)$ has the strict EKR property.
\end{conj}

In this section,  we will show that $\PSL(2,q)$ has the EKR property; i.e. condition (a) of the module method holds fo $\PSL(2,q)$;  this result has been, also, pointed out in \cite{MeagherS11}. Furthermore, we prove that condition (b) also holds for this group; then we will conclude that in order to prove Conjecture~\ref{PSL_main}, one only needs to show the matrix $M$ for $\PSL(2,q)$ has full rank (i.e. condition (c) of the module method holds for this group). We will,  also, present a proof of the strict EKR property for $\PSL(2,q)$ when $q$ is even.

\subsection{Some properties of $\PSL(2,q)$}\label{PSL_properties}
We start with describing $\PSL(2,q)$ and investigate some of the characteristics of this group which will be useful for our purpose. 
For any $n\geq 2$ and any $q$, the \txtsl{special linear  group}, $\SL(n,q)$, is the group of all $n\times n$ matrices on $\F_q$ whose determinants are $1$. The center $Z(\SL(n,q))$ of this group consists of all matrices of the form $cI$, for some $c\in \F_q$ such
that $c^n=1$. The \txtsl{projective special linear group}, $\PSL(n,q)$, is the quotient
\[
\frac{\SL(n,q)}{Z(\SL(n,q))}.
\]
Note that, for $n=2$, if $q$ is even, then $c=1$ and if $q$ is odd, then $c\in\{-1,1\}$. With this definition, it is straight-forward to determine the size of $\PSL(2,q)$.
\begin{lem}\label{size_PSL} 
For any $q$ a prime power
\[
|\PSL(2,q)| =
\begin{cases}
(q+1)q(q-1), & \text{if } q  \text{ is even};\\
\frac{1}{2}(q+1)q(q-1), & \text{if } q  \text{ is odd}.
\end{cases}
\]
\end{lem}
Throughout this section, we let $G_q=\PSL(2,q)$. The \txtsl{projective line} is the set
\[
\p_q= \left\{ \begin{bmatrix} a \\1  \end{bmatrix} \; | \; a \in \F_q^*\right\}
    \cup \left\{ \begin{bmatrix} 1 \\ 0  \end{bmatrix}\right\}, 
\]
in which $k\mathbf{v}=\mathbf{v}$, for any $k\in \F_q^*$ and $\mathbf{v}\in \p_q$. It is well-known  that $G_q$ acts $2$-transitively on $\p_q$ (see \cite{MR1307623} for a proof of this and for a more detailed discussion on the projective special linear groups). Since $|\p_q|=q+1$, one can, therefore, consider $G_q$ as a subgroup of $\sym(q+1)$ acting $2$-transitively on $[q+1]=\{1,\ldots,q+1\}$.

Now we turn our attention to the character table of $G_q$; from the character table, we will  obtain the spectrum of $\Gamma_{G_q}$ and from this, we will identify the least eigenvalue of $\Gamma_{G_q}$. We will, then, use this to establish conditions (a) and (b) of the module method. First, note that if $q$ is even, then according to Lemma~\ref{size_PSL}, we have $G_q=\SL(2,q)$. Table~\ref{char_psl_even} displays the character table of $\SL(2,q)$, for even $q$ (see \cite{gehles2002ordinary}).

\begin{table}[ht!]
\centering
\[
\begin{array}{|c|c|c|c|c|c|}
\hline
&&&&&\\[-.3cm]
&\text{\small Types of C.C.}
& \left[\begin{matrix}
1 & 0 \\
0 & 1
\end{matrix}\right]
& \left[\begin{matrix}
1 & 0 \\
1 & 1
\end{matrix}\right]
& \left[\begin{matrix}
\nu^l & 0 \\
0 & \nu^{-l}
\end{matrix}\right]
& b^m \\ [.4cm]
\hline
&\text{\small Nr. of C.C.} & 1 & 1 & (q-2)/2 & q/2\\
\hline
&\text{\small Size of C.C.} & 1 & q^2-1 & q(q+1) & q(q-1) \\
\hline
\text{\small Chars} & \text{\small Nr. of Chars} &&&&\\
\hline
\id & 1 & 1 & 1 & 1 & 1 \\
\psi & 1 & q & 0 & 1 & -1 \\
\chi_i  & (q-2)/2 & q+1 & 1 & \rho^{i\ell}+\rho^{-i\ell} & 0 \\
 \theta_j & q/2 & q-1 & -1 & 0 & -\sigma^{jm}-\sigma^{-jm} \\
\hline
\end{array}
\]
\caption{Character table of $\PSL(2,q)=\SL(2,q)$ for even $q$.}
\label{char_psl_even}
\end{table}

In this table, $i, \ell \in \{1, \dots (q-2)/2\}$ and $j,m \in \{1,\dots, q/2\}$.  The first row of the table denotes the
different types of the canonical forms of the conjugacy classes of $\SL(2,q)$ while the first column lists the different types of the
irreducible characters of $\SL(2,q)$. The parameter $\nu$ represents a generator of the cyclic multiplicative group $\F_q^*$, $\rho\in
\mathbb{C}$ is a $(q-1)$-th root of unity, $\sigma\in \mathbb{C}$ is a $(q+1)$-th root of unity and $b$ is a matrix in $\SL(2,q)$ of order $q+1$. It is not hard to see that all the first three types of conjugacy classes have fixed points, while the last type of
matrices have no fixed points. Thus $\dd_{G_q}$ is the union of the conjugacy classes with representatives $b, b^2,\ldots, b^{q/2}$. 

Now using Theorem~\ref{Diaconis}, we compute the eigenvalue of $\Gamma_{G_q}$ corresponding to the character $\psi$. It is easy to
see from Table~\ref{char_psl_even} that 
\[ 
\eta_{\id} = \frac{q^2(q-1)}{2}, \quad \eta_\psi=-\frac{q(q-1)}{2}, \quad \eta_{\chi_i} = 0. 
\]
In order to show $\eta_\psi$ is the least eigenvalue of $\Gamma_{G_q}$, it suffices to compare it with $\eta_{\theta_j}$. For
any $1\leq j\leq \frac{q}{2}$ we have 
\begin{equation}\label{theta_j_even}
\eta_{\theta_j}=\frac{1}{q-1}q(q-1)\sum_{m=1}^\frac{q}{2}\left(-\sigma^{jm}-\sigma^{-jm}\right);
\end{equation}
but since $\theta_j$ is orthogonal to the trivial character, we know that
\[
(q-1)-(q^2-1)+q(q-1)\sum_{m=1}^\frac{q}{2}\left(\sigma^{jm}-\sigma^{-jm}\right)=0;
\]
thus (\ref{theta_j_even}) can be simplified as 
\[
\eta_{\theta_j}=\frac{1}{q-1}(q^2-q)=q.
\]
This proves that $\tau=-\frac{q(q-1)}{2}$ is the least eigenvalue of $\Gamma_{G_q}$.

Next we investigate the case where $q$ is odd. Table~\ref{char_psl_1} and Table~\ref{char_psl_3} are the character tables of $G_q$ for the case $q\equiv 1 \pmod{4}$ and $q\equiv 3 \pmod{4}$, respectively (see \cite{gehles2002ordinary}).

\begin{table}[ht!]
\tiny
\[
\begin{array}{|c|c|c|c|c|c|c|c|}
\hline
&&&&&&&\\[-.2cm]
&\text{Types of C.C.}
& \left[\begin{matrix}
1 & 0 \\
0 & 1
\end{matrix}\right]
& \left[\begin{matrix}
1 & 0 \\
1 & 1
\end{matrix}\right]
& \left[\begin{matrix}
1 & 0 \\
\nu & 1
\end{matrix}\right]
& \left[\begin{matrix}
\nu^l & 0 \\
0 & \nu^{-l}
\end{matrix}\right]
& \left[\begin{matrix}
\nu^{\frac{q-1}{4}} & 0 \\
0 & \nu^{\frac{-(q-1)}{4}}
\end{matrix}\right]
& b^m \\
&&&&&&&\\[-.2cm]
\hline
&&&&&&&\\[-.2cm]
&\text{Nr. of C.C.} & 1 & 1 & 1 & (q-5)/4 & 1 & (q-1)/4\\
&&&&&&&\\[-.2cm]
\hline
&&&&&&&\\[-.2cm]
&\text{Size of C.C.} & 1 & (q^2-1)/2 & (q^2-1)/2 & q(q+1) & q(q+1)/2 &q(q-1) \\
&&&&&&&\\[-.2cm]
\hline
\text{Chars} & \text{Nr. of Chars} &&&&&&\\
\hline
\id & 1 & 1 & 1 & 1 & 1 & 1 & 1 \\ [.1cm]
\psi  & 1 & q & 0 & 0 & 1 & 1 & -1 \\ [.2cm]
\chi_i  & (q-5)/4 & q+1 & 1 & 1 & \rho^{i\ell}+\rho^{-i\ell} &  \rho^{i\frac{q-1}{4}}+\rho^{-i\frac{q-1}{4}} & 0 \\ [.2cm]
\theta_j  & (q-1)/4 & q-1 & -1 & -1 & 0 & 0 & -\sigma^{jm}-\sigma^{-jm} \\ [.15cm]
\xi_1  & 1 & (q+1)/2 & \frac{1+\sqrt{q}}{2} & \frac{1-\sqrt{q}}{2} & (-1)^l & (-1)^{\frac{q-1}{4}}& 0 \\ [.15cm]
\xi_2  & 1 & (q+1)/2 & \frac{1-\sqrt{q}}{2} & \frac{1+\sqrt{q}}{2} & (-1)^l & (-1)^{\frac{q-1}{4}}& 0 \\ [.15cm]
\hline
\end{array}
\]
\caption{Character table of $\PSL(2,q)$ for $q\equiv1$ (mod 4).}
\label{char_psl_1}
\end{table}
\bigskip
\begin{table}[ht!]
\tiny
\[
\begin{array}{|c|c|c|c|c|c|c|c|}
\hline
&&&&&&&\\[-.2cm]
&\text{Types of C.C.}
& \left[\begin{matrix}
1 & 0 \\
0 & 1
\end{matrix}\right]
& \left[\begin{matrix}
1 & 0 \\
1 & 1
\end{matrix}\right]
& \left[\begin{matrix}
1 & 0 \\
\nu & 1
\end{matrix}\right]
& \left[\begin{matrix}
\nu^l & 0 \\
0 & \nu^{-l}
\end{matrix}\right]
& b^m
& b^{\frac{q+1}{4}}\\
&&&&&&&\\[-.2cm]
\hline
&&&&&&&\\[-.2cm]
&\text{Nr. of C.C.} & 1 & 1 & 1 & (q-3)/4 &  (q-3)/4 & 1\\
&&&&&&&\\[-.2cm]
\hline
&&&&&&&\\[-.2cm]
&\text{Size of C.C.} & 1 & (q^2-1)/2 & (q^2-1)/2 & q(q+1)  &q(q-1) & q(q-1)/2 \\
&&&&&&&\\[-.2cm]
\hline
\text{Chars} & \text{ Nr. of Chars} &&&&&&\\
\hline
 \id & 1 & 1 & 1 & 1 & 1 & 1 & 1 \\ [.1cm]
 \psi  & 1 & q & 0 & 0 & 1 & -1 & -1 \\ [.2cm]
 \chi_i  & (q-3)/4 & q+1 & 1 & 1 & \rho^{i\ell}+\rho^{-i\ell} & 0 & 0 \\ [.2cm]
\theta_j & (q-3)/4 & q-1 & -1 & -1 & 0 & -\sigma^{jm}-\sigma^{-jm} &  -\sigma^{j\frac{q+1}{4}}-\sigma^{-j\frac{q+1}{4}} \\ [.15cm]
\varphi_1 & 1 & (q-1)/2 & \frac{-1+\sqrt{-q}}{2} & \frac{-1-\sqrt{-q}}{2} & 0 & (-1)^{m+1} & (-1)^{\frac{q+1}{4}+1} \\ [.15cm]
 \varphi_2  & 1 & (q-1)/2 & \frac{-1-\sqrt{-q}}{2} & \frac{-1+\sqrt{-q}}{2} & 0 & (-1)^{m+1} & (-1)^{\frac{q+1}{4}+1} \\ [.15cm]
\hline
\end{array}
\]
\caption{Character table of $\PSL(2,q)$ for $q\equiv3$ (mod $4$).}
\label{char_psl_3}
\end{table}

In both tables $\nu$ represents a generator of the cyclic multiplicative group $\F_q^*$, $\rho\in \mathbb{C}$ is a $(q-1)$-th root
of unity, $\sigma\in \mathbb{C}$ is a $(q+1)$-th root of unity and $b$ is a matrix in $\SL(2,q)$ of order $q+1$. In Table~\ref{char_psl_1}, $i=2,4,6,\ldots,(q-5)/2$, $j=2,4,6,\ldots,(q-1)/2$, $1\leq \ell \leq (q-5)/4$ and $1\leq m\leq (q-1)/4$, while in Table~\ref{char_psl_3}, $i=2,4,6,\ldots,(q-3)/2$, $j=2,4,6,\ldots,(q-3)/2$, $1\leq \ell \leq (q-3)/4$ and $1\leq m\leq (q-3)/4$.

If $q\equiv 1 \pmod{4}$ the only derangement conjugacy classes of $G_q$ are the ones with representatives $b, b^2,\ldots, b^{\frac{q-1}{4}}$. Therefore, in this case we obtain
\[
\eta_{\id}=\frac{q(q-1)^2}{4}, \quad \eta_\psi=-\frac{(q-1)^2}{4}, \quad \eta_{\chi_i} = \eta_{\xi_1} = \eta_{\xi_2} =0.
\]
Using a similar method as for the case where $q$ is even, we observe that
\[
\eta_{\theta_j}=q.
\]
This implies that $\tau=\eta_\psi$ is the least eigenvalue of $\Gamma_{G_q}$.

If $q\equiv 3 \pmod{4}$, then the derangement conjugacy classes are the ones with representatives $b,b^2,\ldots, b^{\frac{q-3}{4}},
b^{\frac{q+1}{4}}$. So we have that 
\[
\eta_{\id} = \frac{q(q-1)^2}{4}, \quad \eta_{\psi} = -\frac{(q-1)^2}{4}, \quad \eta_{\chi_i}=0, \quad \eta_{\varphi_1} = \eta_{\varphi_2} =q.
\]
(calculating $\eta_{\varphi_i}$ requires considering the case where $q \equiv 1 \pmod{8}$ and where $q \equiv 5 \pmod{8}$ separately).

To calculate $ \eta_{\theta_j}$, similar to the case where $q$ is even, we use the fact that $\theta_j$ is orthogonal to the trivial character to get
\[
\eta_{\theta_j}=q.
\]
We have therefore showed that for any odd $q$, the least eigenvalue of $\Gamma_{G_q}$ is $\tau=-\frac{(q-1)^2}{4}$. As well, we have proved the following.
\begin{cor}\label{spectrum}\begin{enumerate}[(i)]
\item If $q$ is even, then the spectrum of $\Gamma_{G_q}$ is
\[
\left(
\begin{array}{cccc}
  \frac{q^2(q-1)}{2} & -\frac{q(q-1)}{2} & q & 0 \\[.2cm]
  1 & q^2 &\frac{q(q-1)^2}{2} & \frac{(q+1)^2(q-2)}{2}
\end{array}\right);
\]
\item If $q$ is odd, then the spectrum of $\Gamma_{G_q}$ is
\[
\left(
\begin{array}{cccc}
  \frac{q(q-1)^2}{4} & -\frac{(q-1)^2}{4} & q & 0 \\[.2cm]
  1 & q^2 &\frac{(q-1)^3}{4} & \frac{(q+1)^2(q-3)}{4}
\end{array}\right).
\]
\end{enumerate}
\end{cor}

\subsection{Main results}

In this section, using the machinery provided in Subsection~\ref{PSL_properties}, we prove our main result towards solving Conjecture~\ref{PSL_main}. We start with showing that $G_q$ has the EKR property. 

\begin{prop}\label{bound_for_all_q} For any $q$, the group $G_q$ has the EKR property.
\end{prop} 
\begin{proof} 
Assume that $q$ is even. Since $G_q$ acts  transitively on $\p_q$, using the ``orbit-stabilizer'' theorem, the stabilizer of a point under the action of $G_q$ on $[q+1]$, has size
\[
 \frac{|G_q|}{q+1}=\frac{(q+1)q(q-1)}{(q+1)}=q(q-1).
\] 
On the other hand, according to Theorem~\ref{ratio2} and Corollary~\ref{spectrum}, any independent set $S$ of vertices of $\Gamma_{G_q}$ satisfies
\[
|S|\leq \frac{|G_q|}{1-\frac{|\dd_{G_q}|}{\tau}}=\frac{(q+1)q(q-1)}{1-\frac{q^2(q-1)/2}{-q(q-1)/2}}=q(q-1).
\]
The case of $q$ odd is similar.
\end{proof}
Also, using the second part of Theorem~\ref{ratio2}, it is  straight-forward to verify condition (b) for $G_q$.
\begin{lem}\label{module_condition_PSL} Let $S$ be a maximum intersecting set in $G_q$. Then $v_S$ is in the direct sum of the trivial and the standard modules of $G_q$.\qed
\end{lem}
Therefore, in order to apply the module method, we only need to verify (c) for $G_q$. We will prove that if $q$ is even then condition (c) holds,  but for $q$ odd, will leave condition (c) as a conjecture; that is, we propose 
\begin{conj}\label{M_fullrank_PSL} If $q$ is odd, then the matrix $M$ for $G_q$ has full rank.
\end{conj}
Now we consider the case where $q$ is even. Note that in this case, $G_q=PGL(2,q)$ and the matrix $M$ was shown in \cite[Proposition
9]{MeagherS11} to have full rank. We offer another, simpler proof of this result.

Define $N=M^TM$; then $N$ is a positive semi-definite matrix of order $q(q-1)$, whose entry $N_{(x,y),(w,z)}$ is the number of derangements mapping $x\mapsto y$ and $w\mapsto z$. Since $N$ and $M$ have the same rank, it is sufficient to prove that $N$ is invertible. We restate Proposition 9 from \cite{MeagherS11}.

\begin{lem}\label{entries_N}
The entries of $N$ are given by
\[
N_{(x,y),(w,z)}=
\begin{cases}
\frac{q(q-1)}{2},& \text{if }\,\, x=w \,\, \text{and}\,\, $y=z$;\\
0,& \text{if }\,\, x=w \,\, \text{and}\,\, y\neq z;\\
0,& \text{if }\,\, x\neq w \,\, \text{and}\,\, y= z;\\
0,& \text{if }\,\, x=z \,\, \text{and}\,\, y=w;\\
\frac{q}{2}, & \text{otherwise}.\qed
\end{cases}
\]
\end{lem}

Thus, one can write $N$ in the following form:
\begin{equation}\label{N_and_A(X)}
N=\frac{q(q-1)}{2}\, I\,+\, \frac{q}{2}\, A(X_n),
\end{equation}
where $I$ is the identity matrix of order $q(q-1)$ and $A(X_n)$ is the adjacency matrix of the pairs graph defined in Section~\ref{spectral_graph_theory}.  Now determining the rank of the matrix $M$ for $G_q$ is straight-forward.
\begin{lem}\label{full_rank_even}
If $q$ is even, then the matrix $M$ has full rank.
\end{lem}
\begin{proof}
By Lemma~\ref{least_eval_of_X_n}, the least eigenvalue of $X_n$ is greater than or equal to $-(n-3)$; hence according to (\ref{N_and_A(X)}) the least eigenvalue of $N$ is at least
\[
\frac{q(q-1)}{2}\,-\, \frac{q}{2}(q-2)=\frac{q}{2}.
\]
This shows that $N$ has no zero eigenvalue and, therefore, $M$ must be full rank.
\end{proof}

We have therefore presented a simpler proof of the fact that Conjecture~\ref{PSL_main} is true if $q$ is even.
\begin{thm}\label{strict_EKR_PSL_even} If $q$ is a power of $2$, then $G_q$ has the strict EKR property.\qed
\end{thm}

\begin{prop}\label{strict_EKR_PSL_odd} Let $q$ be odd. If the matrix $M$ for $G_q$ has full rank, then $G_q$ has the strict EKR property. \qed
\end{prop}

We point out that using a computer algorithm, we have verified the truth of Conjecture~\ref{M_fullrank_PSL} for all prime powers smaller than or equal to $31$, and we have included  it as  one of the future works of this project (see Chapter~\ref{future}).
Note also that $\PSL(2,q)$ having  the strict EKR property is  very interesting since it implies by Theorem~\ref{2transitive_subgroup} that several other groups (including $\PGL(2,q)$) also have the strict EKR property.
In fact, Theorem~\ref{strict_EKR_PSL_even} and Theorem~\ref{2transitive_subgroup} provide
an alternative proof of the fact that $\sym(2^k+1)$ has the strict EKR property, for any $k\geq 0$.

\section{EKR for some sporadic permutation groups}\label{sporadic}

In this section, using the module method, we establish the strict EKR property for a celebrated family of sporadic groups, namely the Mathieu groups. In fact, we show that the $2$-transitive Mathieu groups have the strict EKR property. Then we conclude that all $4$-transitive permutation groups have the strict EKR property. We also provide examples of groups which do not have either the EKR or the strict EKR property. Since the family of Mathieu groups is finite, the main approach of this problem uses a computer program to show some facts; in particular, proving that the condition (c) of Theorem~\ref{module_method_thm}, i.e. the fact that the matrix  $M$ has full rank, is mainly handled by a computer program. All of these programs have been run in the \textbf{GAP} programming system \cite{GAP4}.

Following the standard notation, we will denote the Mathieu group of degree $n$ by $M_n$. Note that $M_n\leq \sym(n)$ and we consider the natural action of $M_n$ on the set $[n]$, as usual.
\textsl{Sporadic groups}\index{sporadic groups} are the $26$ finite simple groups which show up in the classification non-abelian simple groups (see \cite[Theorem 4.9]{cameron1999permutation}). Five of the sporadic groups were discovered by Mathieu in the 1860s, namely $M_{11}, M_{12}, M_{22}, M_{23}$ and $M_{24}$. Then the Mathieu groups $M_{10}$, $M_{20}$ and $M_{21}$ were defined to be the point-stabilizers in the groups $M_{11}$, $M_{21}$ and $M_{22}$, respectively. Note that these are not sporadic simple groups.
 Table~\ref{transitivity_mathieu} lists some of the properties of the Mathieu groups which will be useful for our purpose (see \cite{cameron1999permutation}).

\begin{table}[ht!]
\[
\begin{tabular}{|c|c|c|c|}
\hline
{\bf Group} & {\bf Order} & {\bf Transitivity} & {\bf Simplicity}\\
\hline
$M_{10}$ & 720 & sharply 3-transitive & not simple\\
\hline
$M_{11}$ & 7920 & sharply 4-transitive & simple\\
\hline
$M_{12}$ & 95040 & sharply 5-transitive & simple\\
\hline
$M_{20}$ & 960 & 1-transitive & not simple\\
\hline
$M_{21}$ & 20160 & 2-transitive & simple\\
\hline
$M_{22}$ & 443520 & 3-transitive & simple\\
\hline
$M_{23}$ & 10200960 & 4-transitive & simple\\
\hline
$M_{24}$ & 244823040 & 5-transitive & simple\\
\hline
\end{tabular}
\]
\caption{Order and transitivity table for Mathieu groups}
\label{transitivity_mathieu}
\end{table}

Note that the only Mathieu group not listed in Theorem~\ref{main_mathieu} is $M_{20}$ which is not 2-transitive.  The main theorem of this section is that all the Mathieu groups have the strict EKR property except $M_{20}$.
\begin{thm}\label{main_mathieu}
The Mathieu groups $M_n$, for $n\in \{10,11,12,21,22,23,24\}$, have the strict EKR property.
\end{thm}


 
The following fact can be verified by a computer program.
\begin{lem}\label{standard_is_least_mathieu} Let $n\in \{10,11,12,21,22,23,24\}$. Then conditions (a) and (b) of Theorem~\ref{module_method_thm} hold for $M_n$.\qed
\end{lem}

\begin{lem}\label{M_10_11_12_21} $M_n$ has the strict EKR property, for $n\in \{10, 11,12,21\}$.
\end{lem}
\begin{proof} A computer code can be employed to verify that the matrix $M$ for $M_n$, for these values of $n$, has full rank. Now, the module method and Lemma~\ref{standard_is_least_mathieu} complete the proof.
\end{proof}

\begin{lem}\label{M_22}  $M_{22}$ has the strict EKR property.
\end{lem}
\begin{proof} As in the proof of Lemma~\ref{M_23_24}, it  suffices to show the matrix $M$ has full rank. Let $\mathcal{C}_{22}$ be one of the (two) conjugacy classes of $M_{22}$ whose elements are product of two disjoint $11$-cycles. Similar to the proof of Lemma~\ref{M_23_24}, set $N=M_{\mathcal{C}_{22}}^\top M_{\mathcal{C}_{22}}$. Using a computer code we can establish 
\[
N=1920\,I\,\,+\,\,96\,A(X_n),
\]
where $A(X_n)$ is the adjacency matrix of the pairs graph $X_n$ defined in Section~\ref{spectral_graph_theory}. Then Lemma~\ref{least_eval_of_X_n} shows that the least eigenvalue of $N$ is at least $1920-96(19)=96$. This shows that $N$ is non-singular and we are done.
\end{proof}

Since the groups $M_{23}$ and $M_{24}$ have huge sizes, calculating the rank of their $M$ matrices is not  practical; but thanks to their 4-transitivity, we can apply a more efficient method to show these matrices are full rank. 

\begin{lem}\label{M_23_24}
 $M_n$ has the strict EKR property, for $n\in \{23,24\}$.
\end{lem}
\begin{proof} Using Lemma~\ref{standard_is_least_mathieu}, one only  needs to prove that the matrix $M$ for $M_n$ has full rank. 

Let $\cc_{23}$ be one of the (two) conjugacy classes of $M_{23}$ of permutations that are $23$-cycles and let $\cc_{24}$ be the (only) conjugacy class of $M_{24}$ whose elements are the  product of two disjoint $12$-cycles. Set $t_n=|\cc_n|$, for $n=23,24$. Assume $M_{\cc_n}$ to be the submatrix of $M$, with the rows labeled by $\cc_n$ and set $N_n=M_{\cc_n}^\top M_{\cc_n}$, for $n=23,24$.  We now calculate the entries of $N$. Since $M_n$ is $4$-transitive, the entry $((a,b),(c,d))$ in $N_n$ depends only on the intersection of $\{a,b\}$ and $\{c,d\}$. To see this, consider the pairs $(a,b),(c,d)$ from $[n-1]$.  If an element $\pi\in \cc_n$ maps $a\mapsto b$ and $c\mapsto d$, then for any pairs $(a',b'),(c',d')$ of elements of $[n-1]$, the permutation $g^{-1}\pi g\in \cc_n$ maps $a'\mapsto b'$ and $c'\mapsto d'$, where $g\in M_n$ is a permutation which maps $(a',b',c',d')$ to $(a,b,c,d)$. Therefore, we have
\begin{equation}\label{entries_of_N}
\hspace{-1cm}(N_n)_{(a,b),(c,d)}=
\begin{cases}
(N_n)_{(1,2),(1,2)}, &\quad \text{if } (c,d)=(a,b);\\
(N_n)_{(1,2),(2,1)}, & \quad\text{if } (c,d)=(b,a);\\
(N_n)_{(1,2),(2,3)}, & \quad\text{if }  a\neq d \quad\text{and}\quad b=c;\\
(N_n)_{(1,2),(2,3)}, & \quad\text{if }  a=d \quad\text{and}\quad b\neq c;\\
(N_n)_{(1,2),(3,4)}, & \quad\text{if } a,b,c,d\quad\text{are distinct}.\\
\end{cases}
\end{equation}
Because of the $2$-transitivity of $M_n$, we have $(N_n)_{(1,2),(1,2)}=t_n/(n-1)$ and using a simple computer  code we can check that 
\[(N_n)_{(1,2),(2,3)}=(N_n)_{(1,2),(3,4)}=\frac{t_n}{(n-1)(n-2)}.\]
 Also since elements of $\cc_n$ do not include a cycle of length $2$ in their cycle decomposition, we have $(N_n)_{(1,2),(2,1)}=0$. Thus we can re-write (\ref{entries_of_N}) as
\[
(N_n)_{(a,b),(c,d)}=
\begin{cases}
\frac{t_n}{n-1}, & \quad\text{if } (c,d)=(a,b);\\
0, & \quad\text{if } (c,d)=(b,a);\\
\frac{t_n}{(n-1)(n-2)}, &\quad \text{otherwise }.\\
\end{cases}
\]
This means that one can write
\[
N_n=\frac{t_n}{n-1}\,I\,\,+\,\,\frac{t_n}{(n-1)(n-2)}\,A(X_n),
\]
Again, Lemma~\ref{least_eval_of_X_n} shows that the least eigenvalue of $N_n$ is at least 
\[
\frac{t_n}{n-1}\left(1-\frac{n-3}{n-2}\right)>0.
\]
We conclude that $N_n$ and, consequently, the matrix $M$ are full rank.
\end{proof}

Lemmas~\ref{M_10_11_12_21}, \ref{M_23_24} and \ref{M_22} complete the proof of Theorem~\ref{main_mathieu}.  It is well-known that the only finite $4$-transitive groups are $\sym(n)$, for $n\geq 4$, $\alt(n)$, for $n\geq 6$ and the Mathieu groups  $M_n$, for $n\in\{11,12,23,24\}$ (See \cite[Theorem 4.11]{cameron1999permutation}). This, along with the fact that the symmetric group and alternating group have the strict EKR property \cite{2013arXiv1302.7313A}, proves the following result.
\begin{cor} All 4-transitive groups have the strict EKR property.\qed
\end{cor}
 Note that the Mathieu group $M_{20}$ is, in fact, the stabilizer of a point in $M_{21}$ under its natural action on $\{1,2,\dots,21\}$; it is not $2$-transitive and its standard character is not  irreducible. Hence we cannot use the module method to establish the EKR or strict EKR property for $M_{20}$. Indeed, $M_{20}$ fails to have the EKR property; a computer search shows that one can find independent sets of size $64$ in $\Gamma_{M_{20}}$, while all the point stabilizers have size $48$.

It may seem that all the $2$-transitive groups have the strict EKR property; but this is not the case. For instance a $2$-transitive subgroup of $\sym(8)$, of order $56$, is of the form $H=(\mathbb{Z}_2\times \mathbb{Z}_2\times \mathbb{Z}_2) \rtimes \mathbb{Z}_7$ (see Table~\ref{small_perm_groups}). The size of a point stabilizer is $7$. However, the graph $\Gamma_{H}$ is isomorphic to the union of $7$ complete graphs $K_8$; hence one can build a maximum independent set which is not a coset of a point stabilizer. 

Knowing the least eigenvalue of $\Gamma_G$ is an essential part of the ratio bound and, consequently, is an important tool for the module method.  Recall (Proposition~\ref{standard_is_irr}) that if a group $G$ is $2$-transitive, then the standard character is irreducible. In many examples, the least eigenvalue of  $\Gamma_G$ is attained (only) by the standard character. However, this is not true in general; for example, a $2$-transitive subgroup of $\sym(8)$, of order $168$, is of the form $((\mathbb{Z}_2\times \mathbb{Z}_2\times \mathbb{Z}_2) \rtimes \mathbb{Z}_7)\rtimes \mathbb{Z}_3$  (see Table~\ref{small_perm_groups}) whose least eigenvalue, $-9$,  is not given by the standard character (see Appendix~\ref{appA}). Also the least eigenvalue of the Mathieu group $M_{10}$, $-36$, is attained by the standard character and by a linear character. See Appendix~\ref{appA} for more examples.
\\\\\\
{\bf Computer verification of EKR and strict EKR}

 In Appendix~\ref{appA}, we present Table~\ref{small_perm_groups} in which we have checked the EKR and strict EKR properties for $2$-transitive permutation groups of degree at most $20$ using the module method. The tool for showing a group has EKR property (i.e. condition (a) of the module method) is the ratio bound (the first part of Theorem~\ref{ratio2}) and the clique-coclique bound (the first part of Theorem~\ref{clique_coclique_bound}). In the first method, we check if the least eigenvalue of $\Gamma_G$ is given by the standard character of $G$; according to Theorem~\ref{ratio2}, then, this will imply that the maximum size of an intersecting set will be the size of a point-stabilizer in $G$. If the condition of the first method fails, we apply the second method. In this  method, we check for a clique of size $n$ in $\Gamma_G$. (Currently, for this goal, our algorithm only searches for cycles of length $n$ whose existence implies a maximum clique in $\Gamma_G$; however, there clearly are other algorithms to verify existence of maximum cliques.) Then according to Theorem~\ref{clique_coclique_bound} the maximum size of  an independent set in $\Gamma_G$ will be $|G|/n$, which is the size of a point stabilizer in $G$. Hence both methods will show that $G$  has the EKR property. In  cases where neither  of these methods work, we try to find an example of an  intersecting set in $G$ whose size exceeds the size of a point stabilizer. Note that the search of maximum independent set has a high complexity and, hence, it is not practical to apply it in all the mentioned cases.

Furthermore, according to the second part of Theorem~\ref{ratio2}, if the least eigenvalue of $\Gamma_G$ is given only by the standard character, then condition (b) of the module method is satisfied for $G$. If the least eigenvalue is given by the standard and some other characters, we apply the second method; that is, provided there is a clique $C$ of size $n$ in $\Gamma_G$  (i.e. the clique-clique bound holds with equality), similar to the method used in Subsection~\ref{standard_module}, the program evaluates $\chi(C)$ for any irreducible character $\chi$ of $G$. If $\chi(C)\neq 0$, for any $\chi$ other than the standard character, then according to Corollary~\ref{at_most_one_non-zero}, the characteristic vector of any maximum independent set of $\Gamma_G$ has to lie in the direct sum of the trivial and the standard characters of $G$; i.e. condition (b) holds for $G$. Note that if for a maximum clique the mentioned condition does not hold, then there can be other maximum cliques for which this condition holds. We conclude that the failure of this method for a clique, does not imply that condition (b) fails for the group. Therefore, the success in this method, depends on the choice of a ``suitable'' maximum clique.

Finally we check if the matrix $M$ for $G$ has full rank (condition (c)). All these steps are implemented by a \textbf{GAP} program. If all the conditions of the module method hold, the program indicates that the group has the strict EKR property. Otherwise we will have to apply other approaches or find non-canonical maximum intersecting sets in $G$ to show that $G$, indeed, fails to have the strict EKR property. We point out that this search also requires a huge amount of time for most of the groups and, in some cases, the algorithms employed for these searches are not strong enough to answer our needs. However, in most of the cases where the all the conditions except condition (c) hold, the reason turns out to be that $\Gamma_G$ has exactly two distinct eigenvalues and, hence, by Proposition~\ref{graphs_with_2_evalues}, $\Gamma_G$ is a disjoint union of complete graphs in which one can choose maximum independent sets which are not cosets of any point-stabilizer.

%% file: chap-single_CC.tex
\chapter{Cayley Graphs on $\sym(n)$ Generated by Single Conjugacy Classes}\label{single_CC}

Recall from Chapter~\ref{introduction} that the main objective in the {E}rd{\H o}s-{K}o-{R}ado theorem is finding an upper bound for the size of ``intersecting'' subsets and then characterizing the intersecting subsets of maximum size. In Chapter~\ref{EKR_perm_groups}  we considered a type of ``intersection'' for  permutations and then defined the EKR and the strict EKR properties which would be the bound and characterization in  the {E}rd{\H o}s-{K}o-{R}ado theorem for the permutations. Then we observed that the EKR problem for  permutation groups can be translated to the problem of determining the  independent sets in some graphs of the maximum size; the key for this translation was a family of Cayley graphs, namely  the derangement graphs.  In this chapter we will consider a generalization the EKR and strict EKR property to the case where the set of all the derangements of the group  is replaced by an arbitrary union  $\cc$ of the  conjugacy classes of the derangements of the group. In this case, the EKR and the strict EKR property will be changed to the EKR and the strict EKR property ``with respect to $\cc$'', respectively.    Then we will see that the module method (Theorem~\ref{module_method_thm}) can be generalized to this case (see Section~\ref{EKR_generalization}). 
 Let $c$ be any single conjugacy class in $\sym(n)$. In this chapter, similar to Chapter~\ref{EKR_perm_groups}, the main objective is to find  the maximum size of an independent set in the Cayley graphs $\Gamma(\sym(n);c)$ and, then, to characterize the ones  of maximum size. To this goal, we will apply the generalized module method (Theorem~\ref{module_method_thm_general}).

We consider the case where $c$ is an arbitrary conjugacy class in Section~\ref{singleCC}. Having described the Cayley graph $\Gamma(\sym(n);c)$, we will investigate its maximum independent sets. In Section~\ref{EKR_cyclic_perms} we will study the case when $c=\cc_m$, where $\cc_m$ is the single conjugacy class of $m$-cycles in $\sym(n)$, and will show that for the alternating group $\alt(n)$, the strict EKR property holds with respect to $\cc_n$. This will also provide  a classification of all maximum intersecting sets of $\sym(n)$ with respect to $\cc_n$.

\section{Generalization of the EKR property}\label{EKR_generalization}

As in previous chapters, we let $G\leq \sym(n)$ be a permutation group with the natural action on the set $[n]$.  Throughout this chapter we let $\cc$ be a union of  conjugacy classes of $G$. Two permutations $\pi,\sigma\in G$ are said to be \textsl{adjacent with respect to $\cc$}\index{adjacency wrt a conjugacy class} if $\pi\sigma^{-1}\in \cc$. A subset $S\subseteq G$ is, then, called \textsl{independent with respect to $\cc$}\index{independent set wrt to conjugacy class} if no pair of its elements are adjacent with respect to $\cc$. 
A conjugacy class $c$ of $G$ is said to be  a \txtsl{derangement conjugacy class} if the elements of $c$ have no fixed point under this action.
For the case where $\cc$ is a union of derangement conjugacy classes of $G$, clearly, the stabilizer of a point is an independent set in $G$ with respect to $\cc$ (as is any coset of the stabilizer of a point).   Note, also, that if $\cc$ is the union of all derangement conjugacy classes of $G$, then two permutations $\pi,\sigma\in G$ are adjacent with respect to $\cc$ if and only if $\pi\sigma^{-1}$ has no fixed point and, therefore, a subset $S\subseteq G$ is  independent with respect to $\cc$ if and only if any pair of its elements intersect. 

Assume $\cc$ is  a union of derangement conjugacy classes of $G$. We say the group $G$ has the \txtsl{EKR property} with respect to $\cc$, if the size of any independent subset of $G$ with respect to $\cc$ is bounded above by the size of the largest point-stabilizer in $G$. Further, $G$ is said to have the \txtsl{strict EKR property} with respect to $\cc$ if the only maximum independent  subsets of $G$ with respect to $\cc$ are the cosets of the point-stabilizers. It is clear from the definition that if a group has the strict EKR property with respect to $\cc$, then it will have the EKR property with respect to $\cc$. Also using Proposition~\ref{alpha_of_Cayley_subgraphs}, we  observe the following.
\begin{prop}\label{EKR_for_cc_implies_EKR} If a group $G$ has the EKR property with respect to a $\cc$, then it has the EKR property.\qed
\end{prop}

Note that two permutations in $G$ are  adjacent with respect to $\cc$ if and only if their corresponding vertices are adjacent in $\Gamma(G;{\cc})$. Therefore, similar to Chapter~\ref{EKR_perm_groups}, the problem of classifying the maximum independent subsets of $G$ with respect to $\cc$ is equivalent to characterizing the maximum independent sets of vertices in $\Gamma(G;{\cc})$.
For the case where $\cc$ is a union of derangement conjugacy classes of $G$, we define the \textsl{derangement graph of $G$ with respect to $\cc$}\index{derangement graph!with respect to a conjugacy class}, denoted by $\Gamma_G^{\cc}$, to be the Cayley graph $\Gamma(G;\cc)$. In particular, the derangement graph of $G$ with respect to $\dd_G$, $\Gamma_G^{\dd_G}$, is the derangement graph of $G$. In this case, the key tool for classifying the maximum independent sets of $G$ with respect to $\cc$,  is the following theorem which is a generalization of the module method (Theorem~\ref{module_method_thm}). Before we state the theorem, we define the matrix $M$ of a $G$ with respect to $\cc$ to be the $|\cc|\times (n-1)(n-2)$ matrix whose rows are indexed  by the elements of $\cc$ and whose columns are indexed by the pairs $(i,j)$ with $i,j\in[n-1]$ and $i\neq j$; then entry $(\sigma,(a,b))$ of $M$ is $1$ if $\sigma(a)=b$ and $0$ otherwise, for any $\sigma\in \cc$ and $(a,b)\in [n-1]^2$.

\begin{thm}[Generalized module method]\label{module_method_thm_general} Let $G\leq \sym(n)$ be 2-transitive and let $\cc$ be a union of the derangement conjugacy classes of $G$. Assume the following conditions hold:
\begin{enumerate}[(a)]
\item $G$ has the EKR  property with respect to $\cc$;
\item for any maximum independent set $S$ in $G$, with respect to $\cc$, the vector $v_S$ lies in the direct sum of the trivial and the standard modules of $G$; and 
\item the matrix $M$ of $G$ with respect to $\cc$ has full rank.
\end{enumerate}
Then $G$ has the strict EKR property with respect to $\cc$.\qed
\end{thm}
We omit the proof of this theorem since it is similar to the proof of Theorem~\ref{module_method_thm}.


\section{Arbitrary conjugacy classes}\label{singleCC}

Let $c$ be an arbitrary conjugacy class of $\sym(n)$. In this section we consider the problem of finding the maximum independent sets of $\sym(n)$ with respect to $c$. In other words, we want to classify the maximum independent sets of the Cayley graph $\Gamma(\sym(n);c)$. Recall from Section~\ref{cayley} that $\Gamma(\sym(n);c)$ is a vertex-transitive graph of valency $|c|$. 

If an element of $\sigma$ of $c$  has the disjoint cycle decomposition $\sigma=\sigma_1\sigma_2\cdots\sigma_k$, where $\sigma_i$ is of length $r_i$, and $r_1\geq\cdots\geq r_k$, then we will use the partition notation $c=[r_1,r_2,\dots,r_k]$ to depict the cycle structure of the elements of $c$. We say $c$ is even (odd) if a permutation in $c$ is even (odd).

 Recall from Section~\ref{rep_sym} that the irreducible representations of $\sym(n)$ are all the Specht modules $S^{\lambda}$, where $\lambda$ ranges over all the partitions of $n$. The following is an immediate consequence of Theorem~\ref{Diaconis} and Corollary~\ref{list_of_irrs_of_sym}.
\begin{prop}\label{evals_of_der_graph_single_CC} For any  conjugacy class $c$ of $\sym(n)$, the eigenvalues of $\Gamma(\sym(n);c)$ are given by
\[
\eta_{\lambda}=\frac{|c|}{\chi^{\lambda}(\id)}\chi^{\lambda}(\sigma), \quad \sigma\in c,
\]
where  $\lambda$ ranges over all the partitions of $n$. Moreover, the multiplicity of the eigenvalue $\eta_\lambda$ is $\chi_{\lambda}(\id)^2$.\qed
\end{prop}

It follows from Proposition~\ref{evals_of_der_graph_single_CC} that the eigenvalues of $\Gamma(\sym(n);c)$ corresponding to the trivial partition $\lambda=[n]$ is $|c|$. This is, in fact, the degree of $\Gamma(\sym(n);c)$. In this section we let $\Gamma=\Gamma(\sym(n);c)$. 

If $\lambda=[1^n]$, then Proposition~\ref{evals_of_der_graph_single_CC} also yields that $\eta_{\lambda}=|c|\sgn\sigma$, where $\sigma\in c$. Therefore,
\begin{equation}\label{eta_[1^n]}
\eta_{\lambda}=\begin{cases}
|c|,& \text{ if }\,\, c \text{ is even};\\
-|c|, & \text{ if }\,\, c \text{ is odd}.
\end{cases}
\end{equation}

We will need the following lemma to prove the next proposition.
\begin{lem}\label{char_of_non-identities} Let $n\geq 5$ and let $\lambda\neq [n],[1^n]$ be a partition of $n$. Then for any element $\sigma\in \sym(n)$, we have $|\chi^{\lambda}(\sigma)|=\chi^{\lambda}(\id)$ if and only if $\sigma=\id$.
\end{lem}
\begin{proof}
Let $\Lambda: \sym(n)\to GL(S^{\lambda})$ be the irreducible representation of $\sym(n)$ corresponding to the Specht module $S^{\lambda}$ (see Section~\ref{rep_basic}) and let $d=\dim (\Lambda)=\chi^{\lambda}(\id)$. According to Corollary~\ref{dim_non-identities}, we have $d>1$. Consider the matrix representation of the endomorphism $\Lambda(\sigma)$ on the space $S^{\lambda}$; thus $\Lambda(\sigma)$ is a $d\times d$ matrix. Note that the character value $\chi^{\lambda}(\sigma)$ is, in fact, the sum of all the eigenvalues of the matrix $\Lambda(\sigma)$ in $\mathbb{C}$. On the other hand, since $\sigma^r=\id$, for some $r$, we have that $\left(\Lambda(\sigma)\right)^r=I_{d}$, where $I_d$ is the identity matrix of size $d$. We conclude that the eigenvalues of $\Lambda(\sigma)$ are some roots of unity. Now, suppose $|\chi^{\lambda}(\sigma)|=d$; then all the eigenvalues of $\Lambda(\sigma)$ have to be the same root of unity (otherwise the absolute value of their sum cannot be equal to $d$). This implies that $\Lambda(\sigma)=\theta I_d$ , where $\theta$ is some root of unity. Therefore, $\Lambda(\sigma)$ is in the center of the image of $\Lambda$; that is, for any $\pi\in \sym(n)$, we have 
\[
\Lambda(\pi)\Lambda(\sigma)=\Lambda(\sigma)\Lambda(\pi);
\]
hence $\pi\sigma\pi^{-1}\sigma^{-1}$ must be in the kernel of $\Lambda$ which is a normal subgroup of $\sym(n)$. But the only proper normal subgroups of $\sym(n)$ are $\alt(n)$ and $\{\id\}$. Since $\lambda\neq [n],[1^n]$, we have that  $\ker(\Lambda)\neq \alt(n)$. Thus $\ker(\Lambda)=\{\id\}$. From this, we conclude that $\pi\sigma\pi^{-1}\sigma^{-1}=\id$; that is, $\sigma$ is in the center of $\sym(n)$ which is $\{\id\}$. This completes the proof.
\end{proof}

The following proposition determines the eigenvalues of $\Gamma$ whose absolute value~is~$|c|$.

\begin{prop}\label{eval_degree_by_one_dimensionals} Let $n\geq 5$ and let $c$ be a derangement conjugacy class of $\sym(n)$ and consider the eigenvalue $\eta_{\lambda}$ of  $\Gamma$, where $\lambda\vdash n$. Then $|\eta_{\lambda}|=|c|$ if and only if $\lambda=[n]$ or $[1^n]$.
\end{prop}
\begin{proof}
First note that (\ref{eta_[1^n]}) proves the ``if'' part of the proposition. For the converse, note that if $|\eta_{\lambda}|=|c|$, for a partition $\lambda$, then according to Proposition~\ref{evals_of_der_graph_single_CC}, we must have 
\[
\frac{|c|}{\chi^{\lambda}(\id)}|\chi^{\lambda}(\sigma)|= |c|,
\]
for some $\sigma\in c$; hence we must have $|\chi^{\lambda}(\sigma)|=\chi^{\lambda}(\id)$. Therefore, according to  Lemma~\ref{char_of_non-identities}, we must have $\lambda=[n]$ or $[1^n]$.	
\end{proof}
We, then, observe the following.
\begin{prop}\label{bipartite}
The graph $\Gamma$ is bipartite if and only if $c$ is odd. Furthermore, if $c$ is odd, then $(\alt(n) , (1\,\,2)\alt(n))$ is a bipartition  of  $\Gamma$.
\end{prop}
\begin{proof}
 If $c$ is odd, then no two even permutations in $\sym(n)$ can be adjacent. Similarly, no two odd permutations can be adjacent. It turns out that in the case where $c$ is odd, $\Gamma$ is a bipartite graph with $\alt(n)$ and $(1\,\,2)\alt(n)$ as its parts. For the converse, assume  $c$ is even. First consider $n\geq 5$. According to~(\ref{eta_[1^n]}), the eigenvalue of $\Gamma$ arising from the representation of $\sym(n)$ corresponding to the partition $\lambda=[1^n]$ is $|c|$. Thus, according to Proposition~\ref{eval_degree_by_one_dimensionals}, for any  $\lambda\vdash n$ we have $\eta_{\lambda}\neq -|c|$. This means that the spectrum of $\Gamma$ is not symmetric about the origin. Therefore, Theorem~\ref{symmetric_about_origin}  yields that $\Gamma$ is not bipartite. 

It, hence, remains to verify the theorem for $n\leq 4$. The theorem clearly holds for $n=2$. If $n=3$, then $c$ will be the set of $3$-cycles and $\Gamma\cong K_3\cup K_3$. Also if $n=4$, then   $c$ is either the set of $3$-cycles or $[2,2]$. In the first case the odd cycle 
\[
\id\text{---}(1\,\,2\,\,3) \text{---} (1\,\,3\,\,2) \text{---} \id,
\]
and in the second case the odd cycle 
\[
\id\text{---}(1\,\,2)(3\,\,4) \text{---} (1\,\,3)(2\,\,4) \text{---} \id
\]
will be in $\Gamma$. Hence  the theorem holds for these cases, as well. 
\end{proof}
We can also describe  the connectivity  of $\Gamma$.
\begin{thm}\label{components}
The graph $\Gamma$ is connected if and only if $c$ is odd. Furthermore, if $n\neq 4$  and $c$ is even,  then $\Gamma$ has exactly two connected components isomorphic to $\Gamma(\alt(n);c)$.
\end{thm}
\begin{proof}
Let $H$ be the subgroup of $\sym(n)$ generated by $c$. Since $c$ is a  conjugacy class in $\sym(n)$, the group $H$ is a normal subgroup of $\sym(n)$. But it is well-known (see \cite{MR600654}, for example) that if $n=3$ or $n\geq 5$, then the only nontrivial normal subgroups of $\sym(n)$ are $\sym(n)$ and $\alt(n)$. Now if $c$ is odd, then $\alt(n)\lneq H\trianglelefteq \sym(n)$;  hence $H=\sym(n)$, which means $\Gamma$ is connected (see Proposition~\ref{index}). However, in the case where  $c$ is even, by~(\ref{eta_[1^n]}), the partitions $\lambda=[n]$ and $\hat{\lambda}=[1^n]$ give the same eigenvalue $|c|$, the degree of $\Gamma$. Thus the multiplicity of the degree is more than one and, hence, using Theorem~\ref{multiplicity}, $\Gamma$ must have more than one connected component; that is, $H$ cannot be $\sym(n)$. Therefore $H=\alt(n)$. Now, according to Proposition~\ref{index}, the number of connected components of $\Gamma$ is $[\sym(n):\alt(n)]=2$ and each component is isomorphic to $\Gamma(\alt(n);c)$.

It remains to verify the theorem for the case $n=4$.  Note that the only proper normal subgroups of $\sym(4)$ are $\alt(4)$ and the Klein four-group
\[
V=\{\id, (1\,\,2)(3\,\,4), (1\,\,3)(2\,\,4), (1\,\,4)(2\,\,3)\}.
\]
Hence, if $c$ is odd, then we have that $H=\sym(4)$ and $\Gamma$ is connected; and if $c$ is even (i.e. if $c=[3]$ or $[2,2]$), then the same reasoning as above  shows that $H$ cannot be $\sym(4)$. Thus either $c=[3]$ and $H=\alt(4)$, or $c=[2,2]$ and $H=V$. It follows that  $\Gamma$ has either $2$  or  $[\sym(4):V]=6$ connected components.
\end{proof}

Now we turn our attention to the problem of finding the maximum independent sets in $\Gamma$. First we classify all the maximum independent sets in $\Gamma$ when $c$ is odd.
\begin{thm}\label{max_indy_for_odd_single_CC} If $c$ is odd, and $S$ is an independent set in $\Gamma$, then
\[
|S|\leq \frac{n!}{2},
\]
and equality holds if and only if $S$ is either $\alt(n)$ or $(1\,\,2)\alt(n)$.
\end{thm}
\begin{proof}
First note that, by~(\ref{eta_[1^n]}), the least eigenvalue of $\Gamma$ is $\tau=-|c|$. Thus by the ratio bound (Theorem~\ref{ratio2}),
\begin{equation}\label{ratio_for_odd_single_CC}
\alpha(\Gamma)\leq \frac{n!}{1-\frac{|c|}{-|c|}}=\frac{n!}{2}.
\end{equation}
On the other hand, by Proposition~\ref{bipartite}, $\alt(n)$ and $(1\,\,2)\alt(n)$ both form independent sets of size $n!/2$ in $\Gamma$. We show that these are the only independent sets meeting the bound in (\ref{ratio_for_odd_single_CC}). To do this, suppose $S=X\cup Y$ is an independent set in $\Gamma$ of size $n!/2$, where $\emptyset\neq X\subset \alt(n)$ and $\emptyset\neq Y\subset (1\,\,2)\alt(n)$. Let $\sigma\in c$ and define
\[
X'=\sigma X,
\]
to be the image of $X$ in $(1\,\,2)\alt(n)$ under the multiplication by the $\sigma$. Since multiplication by $\sigma$ induces a matching between $X$ and $X'$, we observe that
\[
(1\,\,2)\alt(n)=X'\,\bigcupdot\,Y,
\]
and that all the edges with a vertex in $X$ have their other vertex in $X'$. Since $\Gamma$ is regular, no vertex in $X'$ can be adjacent to a vertex in $\alt(n)\backslash X$; we conclude that there is no edge from $X\cup X'$ to the vertices in $\sym(n)\backslash(X\cup X')$. Thus $\Gamma$ must be disconnected which is a contradiction according to Theorem~\ref{components}. This shows that either $X=\emptyset$ or $Y=\emptyset$ and the claim is proved.
\end{proof}

For the case where $c$ is even, according to Theorem~\ref{components}, we have that 
\[
\Gamma\cong \Gamma(\alt(n);c) \bigcupdot \Gamma(\alt(n);c).
\]
Hence in order to classify the maximum independent sets in $\Gamma$, it is necessary and sufficient to classify those of $\Gamma(\alt(n);c)$. This problem, for an arbitrary  even conjugacy class $c$, does not seem to be easy. In Section~\ref{EKR_cyclic_perms}, however, we will solve this problem when $c$ is the class of all $n$-cycles, for odd $n$; that is, we prove that $\alt(n)$  has the strict EKR property with respect to this class.

Nevertheless, for the cases where $c$ is an  even derangement conjugacy class which is not $[2,2,\dots,2]$ and $n\geq 10$,  our feeling is that the maximum size of an independent set in $\Gamma$, is $(n-1)!$, which is the size of a point-stabilizer in $\sym(n)$; that is, it seems that in this case, $\sym(n)$ has the EKR property with respect to $c$ (see Proposition~\ref{EKR_for_even_single_CC}). The  main motivation for this is the following conjecture. 

\begin{conj}\label{character_conj}
Let $n\geq 10$ and let $c$ be a derangement conjugacy class of $\sym(n)$ which is not $[2,2,\dots,2]$. Then for any partition $\lambda\vdash n$ such that $\lambda\neq [n], [1^n], [n-1,1]$, we have 
\[
|\chi^\lambda(\sigma)|< \frac{\chi^\lambda(\id)}{n-1},
\]
where $\sigma$ is an element of $c$.
\end{conj}
We have verified this character theoretical conjecture for all $n\leq 30$ with a computer program. In the examples, it looks like the gap between the character value $|\chi^\lambda(\sigma)|$ and the dimension of $\chi^{\lambda}$ grows very fast and this is the reason that for small $n$ this inequality does not hold.
Note also that for some specific conjugacy classes, it is not difficult to verify the conjecture. For example, using the Murnaghan-Nakayama rule, the conjecture is true for $c=[n]$.

The truth of Conjecture~\ref{character_conj}  will yield the following.
\begin{prop}\label{EKR_for_even_single_CC}
Let $n\geq 10$ and let $c\neq[2,2,\dots,2]$ be an even derangement conjugacy class in $\sym(n)$. Then, provided that Conjecture~\ref{character_conj} is true, $\sym(n)$ has the EKR property with respect to $c$.
\end{prop}
\begin{proof}
First let the partition $\lambda_0\vdash n$ be equal to $[n-1,1]$. Then, according to Example~\ref{standard rep}, $\chi^{\lambda}$ is the standard character of $\sym(n)$. Let $\sigma$ be a permutation in $c$. According to Proposition~\ref{evals_of_der_graph_single_CC} and Theorem~\ref{hl_formula}, the eigenvalue $\eta_{\lambda_0}$ of $\Gamma$ is 
\[
\eta_{\lambda_0}=\frac{|c|}{\chi^{\lambda}(\id)}\chi^{\lambda}(\sigma)=\frac{-|c|}{n-1}.
\]
Now assuming Conjecture~\ref{character_conj}  is true, for any partition $\lambda\neq [n],[1^n]$ of $n$, we have 
\[
|\eta_{\lambda}|=\frac{|c|}{\chi^{\lambda}(\id)}|\chi^{\lambda}(\sigma)|\leq \frac{|c|}{n-1};
\]
thus $\eta_{\lambda}\geq \eta_{\lambda_0}$. Since $c$ is even, according to~(\ref{eta_[1^n]}), we have 
\[
\eta_{[n]}=\eta_{[1^n]}=|c|.
\]
We have, therefore, proved that $-|c|/(n-1)$  is the least eigenvalue of $\Gamma$. Hence, according to the ratio bound (Theorem~\ref{ratio2}), for any independent set $S$ in $\Gamma$, we have 
\[
|S|\leq \frac{n!}{1-\frac{|c|}{-|c|/(n-1)}}=(n-1)!.
\]
Since the size of any point-stabilizer in $\sym(n)$ is $(n-1)!$, the proof is complete.
\end{proof}

Recall that Renteln \cite{Renteln2007} has proved that $-|\dd_G|/(n-1)$ is the least eigenvalue of the derangement graph $\Gamma_{\sym(n)}$ of $\sym(n)$. His proof involves facts from symmetric functions and is relatively hard. The other important goal for studying the problem of Conjecture~\ref{character_conj}, is to provide an alternative proof for Renteln's result. In fact, we show  in Proposition~\ref{renteln's} that once Conjecture~\ref{character_conj}  is proved, Renteln's result holds for odd $n$. To do this,  we first  recall the following fact which has been shown in \cite{Recounting}. Let $E(n)$ and $O(n)$ denote the number of even and odd derangements of $\sym(n)$, respectively.
\begin{lem}\label{even_odd_difference} For any $n\geq 1$, 
\[
E(n)-O(n)=(-1)^{n-1}(n-1).\qed
\]
\end{lem}
\begin{prop}\label{renteln's} Assume $n\geq 11$ and odd. Provided Conjecture~\ref{character_conj} is true, $-|\dd_{\sym(n)}|/(n-1)$ is the least eigenvalue of the derangement graph $\Gamma_{\sym(n)}$ of $\sym(n)$.
\end{prop}
\begin{proof}
Note that $-|\dd_{\sym(n)}|/(n-1)$ is the eigenvalue given by the partition $\lambda_0=[n-1,1]$. We prove that for any partition $\lambda\neq [n-1,1]$, the eigenvalue $\eta_{\lambda}$ is greater than $-|\dd_{\sym(n)}|/(n-1)$.  We first assume  $ \lambda\neq [1^n]$. Since $n$ is odd, $\sym(n)$ does not have a conjugacy class $c=[2,2,\dots,2]$. Recall that $CC(G)$ is the set of all derangements  of the group $G$. According to Corollary~\ref{evals_of_der_graph}, we have 
\begin{align*}
|\eta_{\lambda}|&=\left|\sum_{c\in CC(\sym(n))} \frac{|c|}{\chi^{\lambda}(\id)}\chi^{\lambda}(\sigma)\right|\\[.2cm]
                &\leq \sum_{c\,\in CC(\sym(n))} \frac{|c|}{\chi^{\lambda}(\id)}|\chi^{\lambda}(\sigma)|\\[.2cm]
								&< \sum_{c\,\in CC(\sym(n))} \frac{|c|}{n-1}=\frac{|\dd_{\sym(n)}|}{n-1},
\end{align*}
which shows that $\eta_{\lambda}> -|\dd_{\sym(n)}|/(n-1)$. Now let  $\lambda=[1^n]$. Then according to Lemma~\ref{even_odd_difference},
\[
\eta_{\lambda}=\sum_{c\in CC(G)} \frac{|c|}{\chi^{\lambda}(\id)}\chi^{\lambda}(\sigma)=E(n)-O(n)=n-1>-|\dd_{\sym(n)}|/(n-1).\qedhere
\]
\end{proof}

\section{Conjugacy classes of cyclic permutations}\label{EKR_cyclic_perms}

Let $2\leq m\leq n$ be integers. Define the graph $\Gamma_{n,m}$ as the Cayley graph $\Gamma(\sym(n), \cc_m)$, where $\cc_m$ is the conjugacy class of all cyclic permutations of length $m$ in $\sym(n)$. Note that $\Gamma_{n,m}$ is a normal Cayley graph.
Note also that  $\Gamma_{n,n}$ is equal to $\Gamma_{\sym(n)}^{\cc_n}$,  the derangement graph of a group $\sym(n)$ with respect to $\cc_n$. First we study the graphs $\Gamma_{n,m}$, for general $m$, in Subsection~\ref{cyclic_basics}. Then in Subsection~\ref{EKR_for_gamma_nn} we will try to characterize the maximum independent sets of $\Gamma_{n,m}$.

\subsection{Basic facts}\label{cyclic_basics}

In this part, we investigate some basic properties of $\Gamma_{n,m}$. 
First note that $\Gamma_{n,m}$ is a vertex-transitive graph of valency $|\cc_m|= (m-1)!\binom{n}{m}$.  In addition, the following is a consequence of Proposition~\ref{bipartite} and Theorem~\ref{components}.

\begin{prop}\label{connected_bipartite_Gamma_nm}
The graph $\Gamma_{n,m}$  is bipartite if and only if $m$ is even.  Furthermore, $\Gamma_{n,m}$ is connected if and only if $m$ is even.\qed
\end{prop}

Next we calculate the eigenvalues of $\Gamma_{n,m}$. As usual, the main tool for this goal is Theorem~\ref{Diaconis}. We start with the following lemma.

\begin{lem}\label{gamma_evalues1}
The eigenvalues of $\Gamma_{n,m}$ are given by
\[
\eta_{\lambda}=\frac{(m-1)!\binom{n}{m}}{\dim\chi_{\lambda}} \sum_{\mu} (-1)^{r(\mu)} \dim\chi_{\mu},
\]
where the sum is over all partitions $\mu$ of $n-m$ that are obtained from $\lambda$ by removing a skew hook of length
$m$, and $\lambda$ ranges over all partitions of $n$. Moreover, the multiplicity of $\eta_{\lambda}$ is $(\dim \chi_\lambda)^2$.
\end{lem}
\begin{proof}
According to Theorem~\ref{Diaconis}, for any partition $\lambda$ of $n$, 
\[
\eta_{\lambda}=\frac{1}{\chi_{\lambda}(\id)}\sum_{s\in \cc_m}\chi_{\lambda}(s)=\frac{1}{\dim\chi_{\lambda}}(m-1)!\binom{n}{m}\chi_{\lambda}(\sigma),
\]
where $\sigma$ is an $m$-cycle. On the other hand, by Theorem~\ref{nakayama}
\[
\chi_{\lambda}(\sigma)=\sum_{\mu} (-1)^{r(\mu)}\chi_{\mu}(h);
\]
here, $h$ is the identity; therefore 
\[
\eta_{\lambda}=\frac{1}{\dim\chi_{\lambda}}(m-1)!\binom{n}{m} \sum_{\mu} (-1)^{r(\mu)}\dim\chi_{\mu}.\qedhere
\]
\end{proof}
Another result is obtained using the hook length formula (Theorem~\ref{hl_formula}).
\begin{cor}\label{gamma_evalues2}
The eigenvalues of $\Gamma_{n,m}$ are given by
\[
\eta_{\lambda}=\frac{\hl(\lambda)}{m} \sum_{\mu} \frac{(-1)^{r(\mu)}}{\hl(\mu)},
\]
where the sum is over all partitions $\mu$ of $n-m$ that are obtained from $\lambda$ by removing a skew hook of length $m$, and
$\lambda$ ranges over all partitions of $n$. Moreover, the multiplicity of $\eta_{\lambda}$ is $(\dim\chi_\lambda)^2$.
\end{cor}

\begin{proof}
Using Lemma~\ref{gamma_evalues1} and the hook length formula, we have
\begin{align*}
\eta_{\lambda}&=\frac{(m-1)!\binom{n}{m}}{n!/\hl(\lambda)} \sum_{\mu} (-1)^{r(\mu)} \frac{(n-m)!}{\hl(\mu)}&\\[.2cm]
&= \frac{\hl(\lambda)}{m(n-m)!}\sum_{\mu} (-1)^{r(\mu)} \frac{(n-m)!}{\hl(\mu)} &\\[.2cm]
&=\frac{\hl(\lambda)}{m} \sum_{\mu} \frac{(-1)^{r(\mu)}}{\hl(\mu)},&
\end{align*}
and the corollary is proved.
\end{proof}

\begin{cor}\label{evals_hook}
If $m=n$ then the eigenvalues of $\Gamma_{n,m}$ are given by
\[
\eta_{\lambda}=\begin{cases} 
(r-1)!(n-r)!(-1)^{n-r}, &\, \text{if }\, \lambda=[r,1^{n-r}],\,\, \text{for some } \, 1\leq r\leq n; \\
0, & \,\text{otherwise.}
\end{cases}
\]
\end{cor}
\begin{proof}
It is enough to note that the only cases in which $\lambda$ can contain a skew hook of length $n$, is when $\lambda$ is a hook $[r, 1^{n-r}]$, for some $1\leq r\leq n$. In this case we have $\hl(\lambda)=n(r-1)!(n-r)!$. Now using Corollary~\ref{gamma_evalues2}, the result follows.
\end{proof}

With a similar reasoning as Corollary~\ref{evals_hook}, one can observe the following.
\begin{cor}\label{evals_nearhook}
If $m=n-1$ then the eigenvalues of $\Gamma_{n,m}$ are given by
\[
\eta_{\lambda}=\begin{cases}
 n(n-2)!, & \, \text{if }\, \lambda=[n]; \\[.2cm]
\frac{r!(n-r)!(-1)^{n-r+1}}{(r-1)(n-r-1)}, &  \,\text{if }\, \lambda=[r,2,1^{n-r-2}],\,\, \text{for some } \, 2\leq r\leq n-2; \\[.2cm]
(-1)^{n}n(n-2)!, & \,\text{if }\, \lambda=[1^n];\\[.2cm]
0, & \,\text{otherwise}.\qed
\end{cases}
\]
\end{cor}

Using Corollary~\ref{evals_hook}, we can find the rank of the adjacency matrix of $\Gamma_{n,m}$ for the case $m=n$ as follows.
\begin{cor}\label{rank_hooks}
The rank of the adjacency matrix of  $\Gamma_{n,n}$ is $\binom{2n-2}{n-1}$.
\end{cor}
\begin{proof}
It suffices to count the partitions $\lambda$ of $n$ for which the eigenvalue $\eta_{\lambda}$ is non-zero, taking the multiplicities into account. Using Corollary~\ref{evals_hook}, the number of non-zero eigenvalues of $\Gamma_{n,n}$ with their multiplicities is
\[
\sum_{r=1}^n(\dim\chi_{\lambda_r})^2,
\]
where $\lambda_r$ is the hook $[r,1^{n-r}]$, for $r=1,\ldots,n$. Since
\[
\dim\chi_{\lambda_r}=\frac{n!}{\hl(\lambda_r)}=\frac{n!}{n(r-1)!(n-r)!}=\frac{(n-1)!}{(r-1)!(n-r)!}=\binom{n-1}{r-1},
\]
we have that
\[
\rank(\Gamma_{n,n})=\sum_{r=1}^n\binom{n-1}{r-1}^2=\sum_{r=0}^{n-1}\binom{n-1}{r}^2=\sum_{r=0}^{n-1}\binom{n-1}{r}\binom{n-1}{n-r-1};
\]
but the last sum is, indeed, the total number of ways for choosing a set of size $n-1$ from a set of size $2(n-1)$. Therefore
\[
\rank(\Gamma_{n,n})=\binom{2(n-1)}{n-1}.\qedhere
\]
\end{proof}

The following result shows how the eigenvalues corresponding to a partition $\lambda$ are related to the eigenvalues corresponding to the transpose $\hat{\lambda}$ of $\lambda$.
\begin{prop}\label{transpose}
Let $\lambda \vdash n$ and let $\eta_\lambda$ be an eigenvalue of $\Gamma_{n,m}$.
\begin{enumerate}[(a)]
\item If $m$ is odd, then $\eta_\lambda=\eta_{\hat{\lambda}}$, and
\item if $m$ is even, then $\eta_\lambda=-\eta_{\hat{\lambda}}$.
\end{enumerate}
\end{prop}
\begin{proof}
Let $h$ be a skew hook of $\lambda$ and $\mu\vdash (n-m)$ be the partition obtained from $\lambda$ by removing $h$. Then it is easy
to see that the image $\hat{h}$ of $h$ in $\hat{\lambda}$ is a skew hook for $\hat{\lambda}$. Let $\hat{\mu}$ be the corresponding
partition of $n-m$. Note that the number of vertical steps in $\hat{h}$ is equal to the number of horizontal steps in $h$, and the
number of vertical steps in $h$ is $r(\mu)$. On the other hand,
\[
(\text{number of vertical steps of } h)\,\,+ \,\,(\text{number of horizontal steps of } h)=m-1.
\]
Therefore
\[
r(\mu)+r(\hat{\mu})=m-1.
\]
Thus,  $m$ is odd if and only if $r(\mu)$ and $r(\hat{\mu})$ have the same parity. For part (a), assume $m$ is odd. Then using  Corollary~\ref{gamma_evalues2} we have
\[
\eta_{\lambda}=\frac{\hl(\lambda)}{m}\sum_{\mu}\frac{(-1)^{r(\mu)}}{\hl(\mu)}=\frac{\hl(\hat{\lambda})}{m}\sum_{\mu}\frac{(-1)^{r(\hat{\mu})}}{\hl(\hat{\mu})}=\eta_{\hat{\lambda}}.\]
The proof of (b) is similar.
\end{proof}

\subsection{Maximum independent sets}\label{EKR_for_gamma_nn}

This subsection is devoted to characterizing the maximum independent sets of $\Gamma_{n,m}$. 
When $m$  is even, from Theorem~\ref{max_indy_for_odd_single_CC}  one can observe that the structure of independent sets in $\Gamma_{n,m}$ is very  simple. Note also that, in this case, the least eigenvalue is easily determined and, hence, one can easily apply the ratio bound. However, when $m$  is odd, finding the least eigenvalue of $\Gamma_{n,m}$ and, consequently, the classification of maximal independent sets (using the ratio bound) is not as easy as the previous case. 
For some special cases, however, we can still find the least eigenvalues; for instance, using Corollary~\ref{evals_hook} and Corollary~\ref{evals_nearhook} one can observe the following. 
\begin{lem}\label{leastevalue}
If $n$ is odd, the least eigenvalue of  $\,\Gamma_{n,n}$ is
\[
\tau=-(n-2)!,
\]
and if $n$ is even, the least eigenvalue of
$\Gamma_{n,n-1}$ is
\[
\tau=\frac{-2(n-2)!}{n-3}. \qed
\]
\end{lem}

For the rest of this subsection, we assume $m=n$ is odd; therefore, since $\cc_n$ is a derangement conjugacy class, the main problem of this section will reduce to bounding the size of maximum independent sets in $\Gamma_{\sym(n)}^{\cc_n}$ and then characterizing the sets which meet the bound; in other words, we will return to the problem of establishing the EKR  and the strict EKR property for  $\sym(n)$ with respect to $\cc_n$.   First, note that the valency of $\Gamma_{n,n}=\Gamma_{\sym(n)}^{\cc_n}$ is $|\cc_n|=(n-1)!$ and that Theorem~\ref{ratio2} along with Lemma~\ref{leastevalue} implies that if $S\subset \sym(n)$ is  an independent set of $\Gamma_{n,n}$, then
\begin{equation}\label{raio_for_Gamma_nn}
 |S|\leq \frac{n!}{1-\frac{(n-1)!}{-(n-2)!}}=\frac{n!}{1+(n-1)}=(n-1)!.
\end{equation}
Note that $(n-1)!$ is the size of a point stabilizer in $\sym(n)$. This shows that $\sym(n)$ has the EKR property with respect to $\cc_n$. However, by Theorem~\ref{components}, $\Gamma_{n,n}$ is of the form 
\begin{equation}\label{Gamma_nn_decomposition}
\Gamma_{n,n} \cong \Gamma(\alt(n),\cc_n) \,\,\bigcupdot\,\, \Gamma(\alt(n),\cc_n);
\end{equation}
this means that $\sym(n)$ does not have the strict EKR property with respect to $\cc_n$ as there are maximum independent sets which are not cosets of any point-stabilizer (for example,  $S_{1,1}\cup (1,2)S_{1,1}$, where $S_{1,1}$ is the stabilizer of $1$ in $\alt(n)$, is a maximum independent set in $\Gamma_{n,n}$ which is not a coset of a point-stabilizer). But $\sym(n)$ is not so far away from this property. More precisely, according to (\ref{Gamma_nn_decomposition}), once we characterize the maximum independent sets of $\Gamma(\alt(n),\cc_n)$, we can give a characterization of those of $\Gamma_{n,n}$. 

If $S$ is an independent set of $\Gamma_{n,n}$, then by Theorem~\ref{ratio2}, $S$ meets the ratio bound in (\ref{raio_for_Gamma_nn}) if and only if
\begin{equation}\label{Gamma_nn_eigen_eq}
A(\Gamma_{n,n})\left(v_S-\frac{|S|}{n!}\mathbf{1}\right)=\tau\left(v_S-\frac{|S|}{n!}\mathbf{1}\right).
\end{equation}

Define  $\Gamma'_{n,n}$ to be $\Gamma_{\alt(n)}^{\cc_n}$, the derangement graph of $\alt(n)$ with respect to $\cc_n$.   Assume $A$ and $A'$ are the adjacency matrices of $\Gamma_{n,n}$ and $\Gamma'_{n,n}$, respectively. Then one can arrange the rows and the columns of $A$ and $A'$ in such a way that we can write
\begin{equation}\label{decomp_A}
A=
\begin{bmatrix}
A' & 0\\
0 & A'
\end{bmatrix}.
\end{equation}
Note that, this implies that $\tau=-(n-2)!$ is also the least eigenvalue of $\Gamma'_{n,n}$.  Define $S'$ to be $S\cap \alt(n)$ and let $z'=v_{S'}$ be its characteristic vector. By the  ratio bound for $\Gamma'_{n,n}$, we have
\[
|S'|\leq \frac{n!/2}{1-\frac{(n-1)!}{(n-2)!}}=\frac{(n-1)!}{2}\quad \text{and}\quad |S\backslash S'|\leq  \frac{(n-1)!}{2},
\]
thus  $|S'|=|S|/2$. On the other hand  according to (\ref{Gamma_nn_eigen_eq}) and (\ref{decomp_A}), $|S'|=|S|/2$ if and only if
\[
A'\left(z'-\frac{|S'|}{n!/2}\mathbf{1}\right)=\tau\left(z'-\frac{|S'|}{n!/2}\mathbf{1}\right).
\]
Since $\tau$ is the least eigenvalue of $\Gamma'_{n,n}$, this is equivalent to the fact that $S'$ is an independent set in
$\Gamma'_{n,n}$ which meets the ratio bound of the graph $\Gamma'_{n,n}$.

For a subset $S$ of $\alt(n)$, let $(1\,\,2)S=\{(1\,\,2)s\,:\, s\in S\}$ be the image of $S$ in the coset $ (1\,\,2)\alt(n)$. We have, thus, proved the following.

\begin{prop}\label{S'andS''} Let $n\geq 3$ be odd and $S$ be an independent set in $\Gamma_{n,n}$. Then
\[
|S|\leq (n-1)!,
\]
and equality holds if and only if
\[
S=S'\, \cup \, (1\,\,2)S'',
\]
where $S'$ and $S''$ are  independent sets of size $(n-1)!/2$ in $\Gamma'_{n,n}$.\qed
\end{prop}

Therefore, in order to complete the classification, one has to know what the maximum independent sets in $\Gamma'_{n,n}$ look like. That is, the problem of classification of the maximum independent sets of $\Gamma_{n,n}$ reduces to the same problem for $\Gamma'_{n,n}$.  In the rest of the subsection we apply the module method to prove that for $n\geq 5$, the alternating group $\alt(n)$ has the strict EKR property with respect to $\cc_n$.

Note that $\Gamma'_{n,n}$ is a connected normal Cayley graph (of valency $(n-1)!$), whose least eigenvalue is $-(n-2)!$, and as stated above, if $S$ is an independent set of $\Gamma'_{n,n}$, then $|S|\leq (n-1)!/2$. On the other hand, this is the size of any point-stabilizer in $\alt(n)$. We deduce that condition~(a) of Theorem~\ref{module_method_thm_general} holds:
\begin{lem}\label{EKR_for_alt_wrt_C_n}
If $n\geq 3$, then $\alt(n)$ has the EKR property with respect to $\cc_n$.\qed
\end{lem}
Also, Lemma~\ref{leastevalue} and the second statement of the ratio bound for $\Gamma'_{n,n}$ yields condition~(b).
\begin{lem}\label{module_condition_alt_wrt_C_n} If $S$ is an intersecting subset of $\alt(n)$ with respect to $\cc_n$ of size $(n-1)!/2$, then $v_S$ lies in the direct sum of the trivial and the standard modules.\qed
\end{lem}

Next we show that condition (c) holds.
\begin{lem}\label{M_fullrank_alt_wrt_C_n}
For all $n\geq 5$, rank of $M$ has full rank.
\end{lem}
\begin{proof} It suffices to note that the matrix $M$ for $\alt(n)$ with respect to $\cc_n$ is identical to the matrix $M_1$ in the proof of Proposition~\ref{fullrank_alt}. Since $M_1$ has full rank, we are done.
\end{proof}

Now we can prove the main theorem of this section.
\begin{thm}\label{main_alt_wrt_C_n}
Let $n\geq 5$ be odd. For any intersecting set $S$ of $\alt(n)$ with respect to $\cc_n$, we have
\[
|S|\leq \frac{(n-1)!}{2},
\]
and the equality holds if and only if $S$ is a coset of a point-stabilizer.
\end{thm}
\begin{proof} Lemmas~\ref{EKR_for_alt_wrt_C_n}, \ref{module_condition_alt_wrt_C_n} and \ref{M_fullrank_alt_wrt_C_n} show that all the  conditions of the generalized module method (Theorem~\ref{module_method_thm_general}) hold for $\alt(n)$, with respect to $\cc_n$. Thus the theorem follows by the generalized module method.
\end{proof}

We conclude this section by noting that Proposition~\ref{S'andS''} and Theorem~\ref{main_alt_wrt_C_n} prove the following result which classifies all the maximum intersecting subsets of $\sym(n)$ with respect to $\cc_n$. Recall that $S_{i,j}$ defined in (\ref{canonical}) are the canonical independent subsets of the symmetric group. 
\begin{cor}\label{final}
Let $n\geq 5$ be odd. For any intersecting subset $S$ of $\sym(n)$  with respect to $\cc_n$, we have
\[
|S|\leq (n-1)!,
\]
and equality holds if and only if $S=S_{i,j}\,\cup\,(1\,\,2)S_{k,l}$, for some $i,j,k,l\in[n]$.\qed
\end{cor}

We, finally, point out that Corollary~\ref{final} also holds for the case $n=3$ except that in this case equality  holds if and only if $S=\{\pi\}\,\cup\,(1\,\,2)\{\pi'\}$, where $\pi,\pi'\in \alt(3)$. Note that since $\Gamma'_{3,3}\cong K_3$, the independence number of $\Gamma'_{3,3}$ is 1 which agrees with the ratio bound for $\Gamma'_{3,3}$.

%% file: chap-future_works.tex
\chapter{Future Work}\label{future}

In this chapter we present a list of open questions and conjectures which we have come across during our recent studies on the topics discussed in the thesis. 

 The first open problem is about the matrix $M$ for $\PSL(2,q)$. Let $q$ be a prime power. Recall from Section~\ref{EKR_for_psl} that if $q$ is even, then the group $\PSL(2,q)$ has the strict EKR property, but the problem is still open for odd prime powers $q$, and that the problem would be solved if Conjecture~\ref{PSL_main} is true. A similar problem has been solved in \cite{MeagherS11} for $\PGL(2,q)$, where in order to show the matrix $M$ for $\PGL(2,q)$ is full rank, the authors show that the matrix $N=M^{\top} M$ is non-singular. They were able to calculate all the entries of $N$. Applying the same method to the case of $\PSL(2,q)$ is not convenient as the entries of $N$ for this case are not easy to evaluate. 
Hence  we state the following important question.
\vspace{.3cm}
\\
{\bf Question 1.} \textsl{Does the matrix $M$ for $PSL(2,q)$, for odd $q$, have full rank?}
\vspace{.3cm}

The ratio bound (Theorem~\ref{ratio2}), has been of great importance in this thesis. In fact, this is one of the key methods for establishing conditions (a) and (b) of the  module method, provided that the least eigenvalue of the corresponding Cayley graph is known. We have also seen that for any $2$-transitive group $G$, the standard representation of $G$ is irreducible (Proposition~\ref{standard_is_irr}). Furthermore, if the eigenvalue of the  derangement graph of $G$ arising from the standard representation is the least one, then condition (a) of the module method holds; that is, $G$ has the EKR property. If, in addition, the standard representation is the only one giving the least eigenvalue, then condition (b) also holds. It turns out that it is very important to ask what conditions a $2$-transitive permutation group must have in order to ensure that the standard representation is the (only) representation giving the least eigenvalue of the derangement graph. Although many $2$-transitive groups have this property, as we mentioned in Section~\ref{sporadic}, $2$-transitivity is not a sufficient condition. We, therefore, propose the following question.
\vspace{.3cm}
\\
{\bf Question 2.} \textsl{For which permutation groups  is the least eigenvalue given (only) by the  standard representation?}
\vspace{.3cm}

Based on the examples of the $2$-transitive groups in Table~\ref{small_perm_groups} which fail to have the strict EKR property,  we notice that such groups have ``many'' factors of $\zz_2$ or $\zz_3$ in their cyclic group decomposition. So it is reasonable to ask if a $2$-transitive  group with ``many'' factors of $\zz_2$ or $\zz_3$ fails to have the strict EKR property. Another interesting observation in Table~\ref{small_perm_groups} is that among the $2$-transitive groups for which we could check the EKR property, there is no one which fails to have this property. We therefore ask the following.
\vspace{.3cm}
\\
{\bf Question 3.} \textsl{Do all $2$-transitive groups have the EKR property?}
\vspace{.3cm}

For the module method, if we cannot find the least eigenvalue of the derangement graph, or if the least eigenvalue does not give a tight bound for the size of maximum intersecting sets,  then we cannot apply the ratio bound in order to prove the group has the EKR property (condition (a) of the module method). In this situation, the other important approach is using the clique-coclique bound (Theorem~\ref{clique_coclique_bound}). In this method,  which we call it the \textsl{clique-coclique method}, if $G\leq \sym(n)$ is transitive, then we try to find a clique of size $n$. Then the size of any intersecting set in $G$ will be bounded above by $|G|/n$ which is the size of point-stabilizer; that is, $G$ will have the EKR property. Furthermore, if $G$ is $2$-transitive  and for any irreducible character $\chi$ of $G$ which is not the trivial or the standard character of $G$, there is a maximum clique $C$ for which $\chi(C)\neq 0$, then according to Corollary~\ref{clique_vs_coclique}, condition (b) of the module method holds. Therefore, if this holds, then in order to prove $G$ has the strict EKR property, one only needs to prove the matrix $M$ for $G$ is full rank. We have used the clique-coclique method in in Section~\ref{EKR_for_alt}. We have, also, used this method for the $2$-transitive group  $G=((\zz_2\times\zz_2\times\zz_2)\rtimes \zz_7)\rtimes\zz_3\leq \sym(8)$, for which the standard character does not give the least eigenvalue (see Table~\ref{small_perm_groups}).  For this group, the clique
\begin{align*}
C=&\{\id,\, (1\,\,5)(2\,\,6)(3\,\,7)(4\,\,8),\, (1\,\,2)(3\,\,8)(4\,\,7)(5\,\,6),\, (1\,\,6)(2\,\,5)(3\,\,4)(7\,\,8),\\
  &  (1\,\,4\,\,6\,\,7\,\,2\,\,8)(3\,\,5), \,(1\,\,8\,\,5\,\,7\,\,6\,\,3)(2\,\,4),\, (1\,\,3\,\,2\,\,7\,\,5\,\,4)(6\,\,8),\, (1\,\,7)(2\,\,3\,\,6\,\,4\,\,5\,\,8)\}
\end{align*}
has the property that for any irreducible character $\chi$ of $G$, which is not  the standard character, we have $\chi(C)\neq 0$. Thus, using the clique-coclique method, condition (b) holds for $G$. The matrix $M$ for $G$ is also verified to be of full rank; therefore $G$ has the strict EKR property.

 Hence,  for $2$-transitive groups for which condition (b) cannot be proved using the ratio bound, we may still be able to prove the strict EKR property using the clique-coclique method. However, there are other groups in Table~\ref{small_perm_groups} for which we could not find suitable cliques $C$ to establish condition (b) (so we have left question marks for those columns). Thus, we propose the following problem.
\vspace{.3cm}
\\
{\bf Question 4.} \textsl{For the $2$-transitive groups for which the strict EKR property cannot be proved using the ratio bound, are there maximum cliques that can be  used to prove the strict EKR property using the clique-coclique method? }
\vspace{.3cm}

The least eigenvalue is also important for the generalized module method. In other words, for any union $\cc$ of the derangement conjugacy classes of a $2$-transitive group $G$, if  the least eigenvalue of $\Gamma_G^{\cc}$ is given (only) by the standard representation of $G$, then condition (a) (and condition (b)) of Theorem~\ref{module_method_thm_general} holds for $G$. In particular, if $c$ is an even derangement conjugacy class of $\sym(n)$, if we know that the eigenvalue $-|c|/(n-1)$ is given by the standard representation of $\sym(n)$, then we can conclude that $\sym(n)$, and also $\alt(n)$, have the EKR property with respect to $c$ (see Section~\ref{singleCC}). For this goal, it is sufficient to prove Conjecture~\ref{character_conj}.
The bound in Conjecture~\ref{character_conj} seems to be very helpful not only for our purpose, but also in general study of characters of the symmetric group. There have been significant attempts in the literature to provide bounds for the absolute values of the character values of the symmetric group in general (not only at the derangement conjugacy classes). For instance in \cite{feray2011asymptotics, sharp_bounds} and \cite{Roichman_upper_bound} some asymptotic upper bounds for the absolute values of the character values of $\sym(n)$ have been established. The fact that the dimensions of the irreducible representations of the symmetric group, which are not $1$-dimensional or standard, grow very rapidly with $n$, and that we have verified   Conjecture~\ref{character_conj} for any $n\leq 30$, are  strong motivations to propose the following.
\vspace{.3cm}
\\
{\bf Question 5.} \textsl{Is Conjecture~\ref{character_conj} true? Specifically, is it true that if $n\geq 10$, then for any derangement conjugacy class $c\neq[2,2,\dots,2]$ of $\sym(n)$ and any partition $\lambda\neq [n], [1^n], [n-1,1]$ of $n$,
\[
|\chi^\lambda(\sigma)|< \frac{\chi^\lambda(\id)}{n-1},
\]
where $\sigma$ is an element of $c$? Or, partially, for which  conjugacy classes $c$ does it  hold?}
\vspace{.3cm}
Recall that from Section~\ref{singleCC} that the conjecture is true for $c=[n]$.

The next problem deals with any possible relationship between the EKR and the strict EKR property of a group and those of the corresponding quotient groups. 
Recall from Section~\ref{quotient} that if $G$ is a group with  a normal subgroup $N$, then any irreducible representation of $G/N$ produces an irreducible representation for $G$ (Lemma~\ref{induces}). The purpose of proving this and the consequent results was mainly to provide tools to deduce that $G$ or $G/N$ has the EKR (strict EKR) property if the other one does so. In fact, we hope that we can locate the least eigenvalue of $\Gamma_G$ or $\Gamma_{G/N}$ if we know the least eigenvalue of the other one. This would, then, help us to show what conditions $G$ and $N$ must have in order to deduce the EKR property for $G$ or $G/N$, provided the other one has this property. We summarize this problem as the following.
\vspace{.3cm}
\\
{\bf Question 6.} \textsl{Let $G$ be a group with a normal subgroup $N$. Assume $G$ (respectively $G/N$) has the EKR property. Then, under what conditions does $G/N$ (respectively $G$) have the EKR property? How about the strict EKR property?}
\vspace{.3cm}

We finally point out that we have mainly studied the EKR problem for the $2$-transitive groups. Hence, the main problem is still open when we drop the $2$-transitivity condition. Furthermore, as  we mentioned in Chapter~\ref{introduction}, a generalized version of the EKR theorem deals with $t$-intersecting systems of $k$-subsets of an $n$-set. The analogous version of this problem for the permutations is the problem of finding the maximum $t$-intersecting subsets of a permutation group, where two permutations $\alpha,\beta\in \sym(n)$ are said to be \textsl{$t$-intersecting}\index{$t$-intersecting permutations} if $\alpha\beta^{-1}$ fixes at least $t$ elements of $[n]$. This provides many interesting problems and opens new areas of research.

%% file: app.tex
\chapter{Module Method for Small Groups}\label{appA}
In this appendix we present a table of all $2$-transitive groups with degree at most $20$ on which we have applied the module method to establish the strict EKR property. The table shows if the conditions of the module method (Theorem~\ref{module_method_thm}) hold for each group. In particular, it indicates whether  or not we can say the group has the EKR property. This work was implemented by a program in \textbf{GAP}. Note that since all the groups $\sym(n)$ and $\alt(n)$ have the strict EKR property, they are excluded in  the table.
In the table we use the following terminology:
\begin{itemize}
\item \textbf{\textit{n}}: degree of the group;
\item {\bf least}: a ``Yes'' in this column means that the least eigenvalue of the derangement graph is given by the standard character; 
\item {\bf clique}: a ``Yes'' in this column means that  the program has found a clique of size $n$ in $\Gamma_n$ (hence the clique-coclique bound holds with equality); the symbol ``--'' means that we don't need to find a maximum clique,  and the symbol ``?''  means that the program failed to find such a clique;
\item {\bf EKR}:  a ``Yes'' in this column means that the group has the EKR property, i.e. condition (a) of the module method holds; the symbol ``?'' indicates that the program could not verify this;
\item {\bf unique}:  a ``Yes'' in this column means that the standard character is the only character giving the least eigenvalue; hence condition (b) of the module method holds; 
\item {\bf clique-coclique}:  a ``Yes'' in this column means that  using the clique-coclique method (see Chapter~\ref{future}),  the characteristic vector of any maximum independent set of $\Gamma_G$  lies in the direct sum of the trivial and the standard characters of $G$; hence condition (b) of the module method holds; the symbol ``--'' means that we don't need to verify this,  and the symbol ``?''  means that the program could not find suitable cliques to apply the  clique-coclique method;
\item {\bf rank}: a ``Yes'' in this column means that the matrix $M$ for the group $G$ is full rank, i.e. condition (c) of the module method holds;  the symbol ``--'' means that we don't need to check this;
\item {\bf strict}: a ``Yes'' in this column means that $G$ has the strict EKR property; the symbol ``?''  means that the program could not verify this.
\end{itemize}

\begin{center}
\tiny
\begin{longtable}{|c|c|c|c|c|c|c|c|c|c|c|}
\caption[EKR and strict EKR property for small $2$-transitive groups]{EKR and strict EKR property for small $2$-transitive groups} \label{small_perm_groups} \\
\hline
 \multicolumn{1}{|c|}{$n$} & \multicolumn{1}{|c|}{Group} & \multicolumn{1}{|c|}{size} & \multicolumn{1}{|c|}{least}& \multicolumn{1}{|c|}{max. clique}& \multicolumn{1}{|c|}{{\bf EKR}}& \multicolumn{1}{|c|}{unique}& \multicolumn{1}{|c|}{clique-coclique} & \multicolumn{1}{|c|}{rank}& \multicolumn{1}{|c|}{{\bf strict}}\\ \hline\hline 
\endfirsthead

\hline 
\multicolumn{1}{|c|}{$n$} & \multicolumn{1}{|c|}{Group} & \multicolumn{1}{|c|}{size} & \multicolumn{1}{|c|}{least}& \multicolumn{1}{|c|}{max. clique}& \multicolumn{1}{|c|}{{\bf EKR}}& \multicolumn{1}{|c|}{unique}& \multicolumn{1}{|c|}{clique-coclique} & \multicolumn{1}{|c|}{rank}& \multicolumn{1}{|c|}{{\bf strict}}\\ \hline \hline
\endhead


$5$ & $\zz_5\rtimes \zz_4$ & $20$  & Yes & -- & Yes & Yes & -- & No & No\\
\hline
$6$ & $\sym(5)$ & $120$  & Yes & Yes & Yes & No & ? & -- & ?\\
\hline
$6$ & $\alt(5)$ & $60$  & Yes & -- & Yes & Yes & -- & Yes & Yes\\
\hline
$7$ & $\PSL(3,2)$ & $168$  & Yes & Yes & Yes & No & -- & -- & No\\
\hline
$7$ & $(\zz_7\rtimes \zz_3)\rtimes \zz_2$ & $42$  & Yes & -- & Yes & Yes & -- & No & No\\
\hline
$8$ & $(\zz_2\times\zz_2\times\zz_2)\rtimes \PSL(3,2)$ & $1344$  & Yes & -- & Yes & Yes & -- & Yes & Yes\\
\hline
$8$ & $\PSL(3,2)\rtimes\zz_2$ & $336$  & Yes & Yes & Yes & No & ? & -- & ?\\
\hline
$8$ & $((\zz_2\times\zz_2\times\zz_2)\rtimes \zz_7)\rtimes\zz_3$ & $168$  & No & Yes & Yes & N/A & Yes & Yes & Yes\\
\hline
$8$ & $\PSL(3,2)$ & $168$  & Yes & -- & Yes & Yes & -- & Yes & Yes\\
\hline
$8$ & $(\zz_2\times\zz_2\times\zz_2)\rtimes \zz_7$ & $56$  & Yes & -- & Yes & Yes & -- & No & No\\
\hline 
$9$ & $\PSL(2,8)\rtimes\zz_3$ & $1512$  & Yes & -- & Yes & Yes & -- & Yes & Yes\\
\hline
$9$ & $(((\zz_3\times\zz_3)\rtimes Q_8)\rtimes\zz_3)\rtimes \zz_2$ & $432$  & Yes & ? & Yes & No & ? & -- & ?\\
\hline
$9$ & $((\zz_3\times\zz_3)\rtimes Q_8)\rtimes \zz_3$ & $216$  & No & ? & ? & N/A & ? & -- & ?\\
\hline 
$9$ & $\PSL(2,8)$ & $504$  & Yes & -- & Yes & Yes & -- & Yes & Yes\\
\hline
$9$ & $((\zz_3\times\zz_3)\rtimes \zz_8)\rtimes \zz_2$ & $144$  & No & ? & ? & N/A & ? & -- & ?\\
\hline
$9$ & $(\zz_3\times\zz_3)\rtimes \zz_8$ & $72$ & Yes & -- & Yes & Yes & -- & No & No\\
\hline
$9$ & $(\zz_3\times\zz_3)\rtimes Q_8$ & $72$ & Yes & -- & Yes & Yes & -- & No & No\\
\hline
$10$ & $(\alt(6)\times\zz_2)\rtimes \zz_2$ & $1440$ & No & Yes & Yes & N/A & ?& -- & ?\\
\hline
$10$ & $M_{10}$ & $720$  & Yes & -- & Yes & Yes & -- & Yes & Yes\\
\hline
$10$ & $\alt(6)\cdot \zz_2$ & $720$  & Yes & ? & Yes & No & ? & -- & ?\\
\hline
$10$ & $\alt(6)\times\zz_2$ & $720$ & Yes & Yes & Yes & No & ?& -- & ?\\
\hline
$10$ & $\alt(6)$ & $360$  & Yes & -- & Yes & Yes & -- & Yes & Yes\\
\hline
$11$ & $M_{11}$ & $7920$  & Yes & -- & Yes & Yes & -- & Yes & Yes\\
\hline
$11$ & $\PSL(2,11)$ & $660$  & Yes & -- & Yes & Yes & -- & No & ?\\
\hline
$11$ & $(\zz_{11}\rtimes \zz_5)\rtimes \zz_2$ & $110$  & Yes & -- & Yes & Yes & -- & No & No\\
\hline
$12$ & $M_{12}$ & $95040$  & Yes & -- & Yes & Yes & -- & Yes & Yes\\
\hline
$12$ & $M_{11}$ & $7920$  & Yes & -- & Yes & Yes & -- & Yes & Yes\\
\hline
$12$ & $\PSL(2,11)\rtimes\zz_2$ & $1320$  & Yes & Yes & Yes & No & ? & -- & ?\\
\hline
$12$ & $\PSL(2,11)$ & $660$  & Yes & -- & Yes & Yes & -- & Yes & Yes\\
\hline
$13$ & $\PSL(3,3)$ & $5616$  & Yes & -- & Yes & Yes & -- & No & No\\
\hline
$13$ & $(\zz_{13}\rtimes\zz_4)\rtimes \zz_3$ & $156$  & Yes & -- & Yes & Yes & -- & No & No\\
\hline
$14$ & $\PSL(2,13)\rtimes\zz_2$ & $2184$  & Yes & Yes & Yes & No & ? & -- & ?\\
\hline
$14$ & $\PSL(2,13)$ & $1092$  & Yes & -- & Yes & Yes & -- & Yes & Yes\\
\hline
$15$ & $\alt(8)$ & $20160$  & Yes & -- & Yes & Yes & -- & No & ?\\
\hline
$15$ & $\alt(7)$ & $2520$  & Yes & -- & Yes & Yes & -- & No & ?\\
\hline
$16$ & $(\zz_2\times\zz_2\times\zz_2\times\zz_2)\rtimes \alt(8)$ & $322560$  & Yes & -- & Yes & Yes & -- & Yes & Yes\\
\hline
$16$ & $((\zz_2\times\zz_2\times\zz_2\times\zz_2)\rtimes \alt(6))\rtimes\zz_2$ & $11520$  & No & ? & ? & N/A & -- & -- & ?\\
\hline
$16$ & $(((\zz_2\times\zz_2\times\zz_2\times\zz_2)\rtimes \alt(5))\rtimes\zz_3)\rtimes\zz_2$ & $5760$  & No & ? & ? & N/A & -- & -- & ?\\
\hline
$16$ & $((\zz_2\times\zz_2\times\zz_2\times\zz_2)\rtimes \alt(5))\rtimes\zz_3$ & $2880$  & Yes & ? & Yes & No & -- & -- & ?\\
\hline
$16$ & $(\zz_2\times\zz_2\times\zz_2\times\zz_2)\rtimes \alt(7)$ & $40320$  & Yes & -- & Yes & Yes & -- & Yes & Yes\\
\hline
$16$ & $(\zz_2\times\zz_2\times\zz_2\times\zz_2)\rtimes \alt(6)$ & $5760$  & Yes & ? & Yes & No & -- & -- & ?\\
\hline
$16$ & $((\zz_2\times\zz_2\times\zz_2\times\zz_2)\rtimes \alt(5))\rtimes \zz_2$ & $1920$  & No & ? & ? & N/A & -- & -- & ?\\
\hline
$16$ & $(\zz_2\times\zz_2\times\zz_2\times\zz_2)\rtimes \alt(5)$ & $960$  & No & ? & ? & N/A & -- & -- & ?\\
\hline
$16$ & $(((\zz_2\times\zz_2\times\zz_2\times\zz_2)\rtimes \zz_5)\rtimes\zz_3)\rtimes\zz_4$ & $960$  & No & ? & ? & N/A & -- & -- & ?\\
\hline
$16$ & $(((\zz_2\times\zz_2\times\zz_2\times\zz_2)\rtimes \zz_5)\rtimes\zz_3)\rtimes\zz_2$ & $480$  & No & ? & ? & N/A & -- & -- & ?\\
\hline
$16$ & $((\zz_2\times\zz_2\times\zz_2\times\zz_2)\rtimes \zz_5)\rtimes\zz_3$ & $240$  & Yes & -- & Yes & Yes & -- & No & No\\
\hline
$17$ & $\PSL(2,16)\rtimes \zz_4$ & $16320$  & Yes & -- & Yes & Yes & -- & Yes & Yes\\
\hline
$17$ & $\PSL(2,16)\rtimes \zz_2$ & $8160$  & Yes & -- & Yes & Yes & -- & Yes & Yes\\
\hline
$17$ & $\PSL(2,16)$ & $4080$  & Yes & -- & Yes & Yes & -- & Yes & Yes\\
\hline
$17$ & $\zz_{17}\rtimes\zz_{16}$ & $272$  & Yes & -- & Yes & Yes & -- & No & No\\
\hline
$18$ & $\PSL(2,17)\rtimes \zz_2$ & $4896$  & Yes & ? & Yes & No & -- & -- & ?\\
\hline
$18$ & $\PSL(2,17)$ & $2448$  & Yes & -- & Yes & Yes & -- & Yes & Yes\\
\hline
$19$ & $(\zz_{19}\rtimes\zz_{9})\rtimes\zz_2$ & $342$  & Yes & -- & Yes & Yes & -- & No & No\\
\hline
$20$ & $\PSL(2,19)\rtimes \zz_2$ & $6840$  & Yes & ? & Yes & No & -- & -- & ?\\
\hline
$20$ & $\PSL(2,19)$ & $3420$  & Yes & -- & Yes & Yes & -- & Yes & Yes\\
\hline

\end{longtable}
\end{center}